\documentclass{amsart}

\usepackage{
amsfonts,
latexsym,
amssymb,
amsmath,
amsthm,
enumerate,
verbatim,
mathrsfs,
epigraph,
}

\usepackage{url}

\newcommand{\labbel}[1]{\label{#1} [[{\bf #1}]]}  
\renewcommand{\labbel}{\label}

\newcommand{\ciro}{{\hspace {1pt} \circ \hspace {1pt}}}

\makeatletter
\newcommand{\xleftrightarrow}[2][]{\ext@arrow 3359\leftrightarrowfill@{#1}{#2}}
\newcommand{\xdashrightarrow}[2][]{\ext@arrow 0359\rightarrowfill@@{#1}{#2}}
\newcommand{\xdashleftarrow}[2][]{\ext@arrow 3095\leftarrowfill@@{#1}{#2}}
\newcommand{\xdashleftrightarrow}[2][]{\ext@arrow 3359\leftrightarrowfill@@{#1}{#2}}
\def\rightarrowfill@@{\arrowfill@@\relax\relbar\rightarrow}
\def\leftarrowfill@@{\arrowfill@@\leftarrow\relbar\relax}
\def\leftrightarrowfill@@{\arrowfill@@\leftarrow\relbar\rightarrow}
\def\arrowfill@@#1#2#3#4{%
  $\m@th\thickmuskip0mu\medmuskip\thickmuskip\thinmuskip\thickmuskip
   \relax#4#1
   \xleaders\hbox{$#4#2$}\hfill
   #3$%
}
\makeatother

\newtheorem{theorem}{Theorem}[section]
\newtheorem{lemma}[theorem]{Lemma}

\newtheorem{proposition}[theorem]{Proposition} 
 
\newtheorem{corollary}[theorem]{Corollary}

\newtheorem*{theorem*}{Theorem}
\newtheorem*{corollary*}{Corollary}

\theoremstyle{definition}
\newtheorem{definition}[theorem]{Definition}

\newtheorem{problem}[theorem]{Problem}

\theoremstyle{remark}
\newtheorem{remark}[theorem]{Remark}

\newtheorem{construction}[theorem]{Construction}

\numberwithin{equation}{section}

\newcommand{\brfrt}{\hspace{0 pt}}

\newcommand{\hy}{\frac{\ }{\ }}

\newcommand{\xfix}{0}
\newcommand{\zfix}{1}
\newcommand{\ztox}{^{\rightarrow}}
\newcommand{\xtoz}{^{\leftarrow}}

 \allowdisplaybreaks[1]

\begin{document}

\title{Day's Theorem is sharp for $n$ even}

\author{Paolo Lipparini} 
\address{Dipartimento di Matematica\\Viale   Giornaliero della  Ricerca
 Scientifica\\Universit\`a di Roma ``Tor Vergata'' 
\\I-00133 ROME ITALY\\
ORCiD: 0000-0003-3747-6611}
\urladdr{http://www.mat.uniroma2.it/\textasciitilde lipparin}
\email{lipparin@axp.mat.uniroma2.it}


\date{\today}

\thanks{Work performed under the auspices of G.N.S.A.G.A. Work 
partially supported by PRIN 2012 ``Logica, Modelli e Insiemi''.
The author acknowledges the MIUR Department Project awarded to the
Department of Mathematics, University of Rome Tor Vergata, CUP
E83C18000100006.
}

\begin{abstract}
We solve some problems
about relative lengths of Maltsev conditions,
in particular, we give an affirmative answer
to a classical problem raised 
by A. Day
more than 
fifty years ago.

In detail, both congruence distributive and congruence modular varieties  admit
 Maltsev characterizations by means of  the existence of
a finite but variable number of appropriate terms.
A.\ Day showed that
 from J{\'o}nsson terms $t_0, \dots, t_n$ witnessing congruence distributivity
it is possible to construct  terms
$u_0, \dots, \allowbreak  u _{2n-1} $ 
 witnessing congruence modularity.
We show that Day's result
about the number of such terms 
 is sharp when $n$ is even.

We also deal with other kinds of terms,
such as alvin, Gumm, Pixley, directed, specular, mixed 
and defective.
All the results hold  when  restricted 
 to  locally finite varieties.
 We introduce some  families  of congruence distributive varieties 
and characterize many congruence identities they satisfy.
\end{abstract} 

\maketitle

\section{Introduction} \labbel{intro}

A \emph{variety} is a class of \emph{algebraic 
systems} (henceforth, \emph{algebras}, for short)
of the same type and 
closed under taking products, subalgebras and 
homomorphic images.
Groups and rings form a variety,
fields do not.  
A \emph{congruence} on an algebra
is the equivalence relation associated 
to some homomorphism. 
Described internally,
a congruence is
an equivalence relation respecting all the operations.
E.~g., the  laterals of a normal subgroup 
in a group, or the laterals of a
bilateral ideal in a ring
 are the equivalence classes associated to a congruence.
From the point of view of general algebra the case of groups
and rings
is very special;
 in general, a congruence cannot be described in terms
of a single subset  \cite{GU}. For example, just work out the congruences
of a four-elements linearly ordered set, considered as a lattice
with the operations of $\min$ and $\max$. 

Tight structure theorems have been proved for varieties 
all whose algebras have congruences satisfying
appropriate properties \cite{HMK,KK}.
Two congruences $\alpha$ and $\beta$
of the same algebra   are said to \emph{permute}
if $ \alpha \circ \beta = \beta \circ \alpha $, where
$\circ$ is relational composition.
Varieties consisting of algebras with permuting congruences,
\emph{congruence permutable varieties}, for short,  
 share some properties in common with groups \cite{Sm}.
The set of all congruences of some algebra
has a natural lattice structure.
Varieties all whose algebras have a distributive
congruence lattice share many properties in common with 
lattices \cite{JD} (notice that here lattices play a two-fold role).
 If both congruence permutability and distributivity are present,
we get varieties sharing some good properties 
in common with Boolean algebras.
Turning to weaker conditions, the
 \emph{modular} identity is a lattice identity corresponding
to the Dedekind law holding for normal subgroups of a group.
All congruence permutable algebras 
and all congruence distributive algebras are congruence modular.
Astonishingly,
it has been discovered that  virtually all the properties which
hold both in congruence permutable and in congruence distributive varieties 
hold in congruence modular varieties, as well \cite{FMK,G}. 

The key to many such results 
are  characterizations obtained by means of  the existence
of certain terms (aka words, expressions, polynomials).
The first and simplest such characterization is due 
to A. Maltsev \cite{M}: a variety $\mathcal V$ 
is congruence permutable  if and only if
there is  a term $t$ in the language of $\mathcal V$ 
such that $x=t(x,y,y)$
and $t(x,x,y)=y$ are identities holding in 
all algebras in $\mathcal V$.    
For groups $t(x,y,z)=x y^{-1}z $
is such a term. For quasigroups
$(x / (y \backslash y))(y \backslash z)$ works.  
As an immediate consequence,
varieties with additional operations
are congruence permutable, too.  
This applies to rings,
Boolean algebras,  modules,  groups and quasigroups with operators.

Similar characterizations have been discovered for
congruence distributive \cite{JD} and congruence  modular varieties \cite{D},
with a significant difference. Contrary to the case of congruence permutability,
 in these cases the characterization does not involve a single
term, rather, it involves a sequence of terms, and the length of such a sequence
should be allowed to vary in the set of natural numbers. See Section \ref{prel}
  for examples.
Many other such \emph{Maltsev conditions}  have been subsequently discovered, characterizing
many more properties \cite{HMK,KK}.  

In passing, let us notice that the mentioned conditions
 involving terms have been originally
devised only in search for conditions implying ``good algebraic structure''.
Quite unexpectedly, however, 
more recently the very same conditions, or small variations
thereof, 
 have come to prominence in the field of computational complexity,
particularly in connection with 
aspects related to the algebraic constraint satisfaction problem.
See, e.~g., \cite{Bar,BKW,JKN} for a survey. 

Summarizing, a Maltsev condition is a condition asserting the existence
 of some sequence of terms
satisfying  appropriate sets of equations in some variety.
In general, the above sequences  have variable length, roughly, a Maltsev
condition is a statement of the form
\emph{there is some $n \in \mathbb N$ and there are terms
$t_0, \dots, t_n$ such that so and so.} 

Deep results are known connecting various distinct Maltsev
conditions, both in abstract and in more
specific settings. 
In fact, as we hinted above,
there are many nontrivial results asserting that 
if all the algebras in some variety have some property,
then all the algebras in that variety have some other property.
In general, the proof is obtained by translating the relevant properties
into appropriate Maltsev conditions. 
Notice that, generally, such results hold only
at the ``global'' level of varieties, not at the ``local''  level 
of single algebras.  Namely,
if the first property holds in some algebra, it is not necessarily the case
that the second property holds in that algebra.
Usually only global properties holding within a variety 
do imply Maltsev conditions, this fact partially explains the reason
for the usefulness of Maltsev conditions.
For example, a $3$-elements algebra $\mathbf A$ with no operation
has a  modular  congruence lattice, hence it satisfies a comparatively
strong local property, but the variety $\mathcal V$ generated
by $\mathbf A$ is trivial, hence, in a sense,  it satisfies no nontrivial 
global property. In any case, $\mathcal V$ satisfies no nontrivial
Maltsev condition.

 A significant specific example 
of applications of Maltsev conditions
concerns congruence modularity.
There is a bunch of really different-looking Maltsev conditions,
all of them
 characterizing congruence modularity
 \cite{D,DF,DKS,FJ,G1,G,adjt,Na,T}.  Hence
all such Maltsev conditions are actually equivalent.
Each condition has specific features 
leading to particular applications which are difficult or virtually
impossible to obtain using the other conditions.
Most proofs are highly nontrivial. 
Some  conditions characterizing congruence modularity
shall be explicitly recalled and studied here; see, in particular,
Sections \ref{prel} and  \ref{gummdef}. 

Other deep results connecting apparently distinct Maltsev conditions
can be found in  \cite{HMK,KK}.
In passing, let us also mention an unexpected quite recent result:
there is the weakest strong Maltsev condition among all the idempotent nontrivial 
Maltsev conditions \cite{Sig,Ols}. A Maltsev  condition
is \emph{idempotent} if all the terms defining it are idempotent, that is,
they satisfy $t(x,x,x, \dots ) = x$. The product $xy$ in a group is not
idempotent, but the Maltsev term $t(x,y,z)=x y^{-1}z$ witnessing
congruence permutability is  idempotent.  Similarly, the majority
of the interesting Maltsev  conditions are indeed idempotent.

While, as we have just mentioned, 
many deep results are known about Maltsev conditions
and their interplay, really little is known about the
possible  lengths of the corresponding
sequences.
In more detail, 
 if $\mathfrak M$ 
and $\mathfrak M'$ are Maltsev conditions expressed in the standard formulation and 
$\mathfrak M$ implies $\mathfrak M'$,
then, for every $n \in \mathbb N$, there is $m \in \mathbb N$   
such that, whenever some variety $\mathcal V$ 
satisfies $\mathfrak M$, as witnessed by $n$ terms,
then   $\mathcal V$ 
satisfies $\mathfrak M'$, as witnessed by $m$ terms.
See the following subsection for details, exemplified in the specific case
of interest here.
We are in the  really strange situation that 
highly nontrivial results are known about different
Maltsev conditions, while only a few facts are  known
about the relationship among the specific lengths of
the corresponding sequences.

Among the  few results known in this direction,
a relatively easy argument by A. Day \cite{D} 
shows that a variety with $n$ J{\'o}nsson 
terms witnessing congruence distributivity
has $2n-1$ Day terms witnessing congruence modularity.
Shortly after the proof of his theorem, Day 
himself asked whether the result is
optimal. To the best of our knowledge, the problem has 
not been solved before. We show that Day's Theorem
is optimal when $n$ is even.
When $n$ is odd, some arguments from 
Lampe, Taylor and Tschantz \cite{LTT}  show that 
Day's  value can be improved by $1$  and our results here
imply that for $n$ odd Day's  value cannot be improved by $3$, leaving
open the possibility that it can be improved by $2$.  

The arguments from \cite{LTT}
use an assumption weaker than congruence distributivity;
in particular \cite{LTT} is mainly concerned with two
distinct Maltsev conditions both characterizing congruence modularity,
the already mentioned condition by Day and
another condition discovered by Gumm.
The arguments in the present note show
also that the evaluation in \cite{LTT} of the number of
Day terms obtainable from a sequence of Gumm terms is optimal,
for $n$ even.
The converse problem, already asked in 
\cite{LTT}, namely, the evaluation of the minimal number
of Gumm terms obtainable from a sequence of Day's terms
is still largely open.
The relevance of this problem has been stressed by a recent result
\cite{LGA} 
in which we show that a congruence distributive variety 
with $n$ Gumm terms has $n+1$ J{\'o}nsson terms.
Together with \cite{LTT}, the above result implies that 
\emph{in a congruence distributive variety} 
the number of Day terms (for congruence modularity)
directly influences the number of J{\'o}nsson terms
(for congruence distributivity)! A
 solution of  Lampe, Taylor and Tschantz' problem
is necessary in order to evaluate exactly the effect 
of this influence.
See Problem \ref{!} for further comments and details.  
 
A similar problem about lengths of sequences appears in 
\cite{adjt}, where the classical Maltsev condition by J{\'o}nsson 
for congruence distributivity is showed to be equivalent
to another condition involving ``directed'' terms.  
In \cite{adjt}  directed Gumm terms are also introduced. 
Directed terms have proved to be very interesting and
useful \cite{Z,adjt,jds,KV}. 
Again, we show that the evaluation of the number
of ``undirected'' terms obtainable from a sequence of directed terms
is optimal, the reverse direction being still open. 

As the reader has surely already appreciated,
the main emphasis in this introduction concerns how much
is known about Maltsev conditions and how little is known
about the lengths of the corresponding sequences of terms.
Astonishingly, already the examples dealing with a single Maltsev condition
turn out to be  generally rather tricky.
Let us explain the point in more detail.
As we already mentioned, say, congruence distributivity
can be characterized by 
sequences of Maltsev terms of length not prescribed in advance.
This suggests that, for example, 
there is a congruence distributive variety 
whose congruence distributivity is witnessed by a sequence 
consisting of 5 terms, but cannot be witnessed by 4 terms or less.
Similar examples necessarily can be found for
any other Maltsev condition whose definition---in contrast, for example, with the
 Maltsev condition characterizing congruence permutability---depends
 on a sequence
of variable length.
 While such examples can be worked out abstractly,
essentially as some hard exercise on term-rewriting \cite{Fi},
until recently \cite{FV,csmc} no natural \cite[p. 368]{CV}  example was known.
As a by-product of our techniques, we provide new examples 
for congruence distributivity and,
 possibly, the first explicit
locally finite examples for congruence modularity.
 See Section \ref{moreex}, in particular, 
Corollaries \ref{corsumupdist}, \ref{corsumupmod} 
and \ref{corsumupgumm}.   
The   examples we present
in Section \ref{moreex}
are locally finite  congruence distributive 
varieties generated by $2$-elements algebras.
These examples have probably  intrinsic interest
 for their own sake. 

As hinted in the above paragraphs, 
our techniques are quite general and go far beyond
the proof that Day's 
 Theorem 
is sharp. Our basic constructions outlined
in Section \ref{mainc} work in a rather abstract context
and have applications
outside the theory of congruence modular varieties.
In fact, we have results about much weaker congruence
identities; see Remark \ref{polin}. 
In Section \ref{specm} we notice that all the constructions
we perform produce varieties with \emph{specular} terms,
a notion weaker than full symmetry  
and which probably deserve further study, as suggested also by
\cite{chic}.
 In Section \ref{mixed}
we hint a connection with Maltsev conditions
 associated to certain directed paths and devised in 
\cite{KV}. Moreover, as  a byproduct of our techniques, we 
evaluate exactly certain congruence identities satisfied by 
the variety of nearlattices,  complementing some results
from \cite{B}.   See 
Proposition \ref{baker+} and Corollary \ref{nl}. 

We now summarize the paper
in more detail, 
touching upon a few significant technical issues,
when convenient.
We assume the reader is familiar
with the basic notions of universal algebra.
See, e.~g., \cite{G,HMK,CV,KK,MMT}
for background.

\subsection*{The ways congruence 
modularity follows from distributivity} 
 \labbel{twc} 
It is plain that every 
congruence distributive variety is
congruence modular. 
Since both congruence distributivity and modularity admit
Maltsev characterizations, it is theoretically possible 
to construct a sequence of terms 
witnessing congruence modularity from any
 sequence of terms witnessing congruence distributivity.

In more detail, a classical theorem by B.\ J{\'o}nsson  
\cite{JD} asserts that a variety is congruence distributive 
if and only if, for some natural number
$n$,  there are  terms
$t_0, \dots, t_n$ satisfying an appropriate set of equations.
A parallel result has been proved by A.\ Day \cite{D}
with respect to congruence modularity.
See Section \ref{prel} for details and explicit definitions. 
If, for each $n$, we consider a variety $\mathcal {V}$ with exactly 
 operations
satisfying J{\'o}nsson equations, then $\mathcal {V}$ is
congruence distributive, hence congruence modular, so finally,
by Day's theorem, $\mathcal {V}$ has a certain number
$u_0, \dots, u_m$ of Day terms.
Working in an appropriate free algebra,
we can thus express Day terms in function of J{\'o}nsson terms.
Since, by its very definition, $\mathcal {V}$ is interpretable in any variety 
having J{\'o}nsson terms $t_0, \dots, t_n$,
we get that every variety $\mathcal W$ with
J{\'o}nsson terms $t_0, \dots, t_n$
has  Day terms $u_0, \dots, u_m$ for the same
$m$ as above, actually, the $u_k$'s are expressed 
in the same way. 

While the above general argument is useful 
in very complex situations,
in the special case at hand the construction of Day
terms from J{\'o}nsson terms can be obtained in a relatively simple way.
We shall recall Day's argument in
the proof of Corollary \ref{dayopt}(i) below
and variations on it appear in Theorem \ref{dir}(ii) 
and  Lemma \ref{dayoptlem}.
The latter  merges the original  Day's argument
with some ideas from \cite{LTT}.

It is customary to say that a variety is
\emph{$n$-distributive} if it has 
J{\'o}nsson terms  $t_0, \dots, t_n$; the definition of
an \emph{$m$-modular} variety is similar.
See Section \ref{prel} for details.   
The following theorem already appeared in 
\cite{D}, in the same paper where the Maltsev characterization 
of congruence modularity has been presented.

\begin{theorem} \labbel{day}
(A.\ Day \cite[Theorem on p.\ 172]{D}) 
If $n>0$, then every $n$-distributive variety
is $2n{-}1$-modular.  
 \end{theorem} 

Among other, in the present paper we show that
Day's Theorem  is the best possible result
for $n$ even.

\begin{theorem} \labbel{daysh}
If $n>0$ and $n$ is even, then there is a
locally finite  $n$-distributive variety
which is not  $2n{-}2$-modular.  
 \end{theorem}

Theorem \ref{daysh} is a special case of Theorem \ref{buh}(i)
which shall be proved below. 
The reader interested only in a quick tour
towards the proof of Theorem \ref{daysh}
might go directly to Section \ref{mainc},
turning back to Section \ref{prel} only if necessary
to check notation and terminology.
A large part of Section \ref{mainc} is not necessary for the 
proof of Theorem \ref{daysh}, either.
A quick route to the proof of Theorem \ref{buh}(i)
goes through Constructions \ref{c}, \ref{bak}, \ref{ba},  
Lemma \ref{lembak}(i) 
and Theorems \ref{thmB}, \ref{thmbak} and \ref{thmba}(i)(ii).
An explicit counterexample witnessing Theorem \ref{daysh}
is the variety $\mathcal V_n^a$ which shall be introduced in 
 Definition \ref{simpldef2}(a).
See Theorem \ref{sumup}.   

 However, in the author's humble opinion,
the present paper contains some further results which
might be of interest to scholars working on congruence distributive 
and congruence modular varieties. 
In particular, as a by-product of our proof of 
Theorem \ref{daysh}, we get the best evaluation for the
modularity levels of varieties with Gumm terms, again
in the case $n$ even. See Proposition \ref{gummopt}.
Moreover, similar arguments can be used to show that a theorem   
by A.\ Mitschke \cite{Mi}, too, gives the best possible evaluation.
See \cite{misharp}. As we are going to explain soon, we get quite neat
results about certain directed, reversed and specular conditions.
Occasionally, we shall also deal with  congruence identities
which fail to imply congruence modularity. See Remark \ref{polin}
below.

Concerning the case  $n$   odd
in Day's Theorem \ref{day}, 
we mentioned  in \cite{ntcm} that
if $n>1$ and $n$ is odd, then 
 Day's Theorem can be improved
(at least) by $1$.
The improvement follows already from the arguments 
in the proof of \cite[Theorem 1 (3) $\rightarrow $   (1)]{LTT},
actually, under a hypothesis weaker than congruence distributivity.
See Lemma \ref{dayoptlem} and Section \ref{gummdef},
in particular, Proposition \ref{gummopt}(i). 
 We do not know what 
is the best possible result for $n$ odd. 
We shall briefly discuss the issue 
 in Remark \ref{odd}
below. 

\subsection*{``Reversed'' conditions and hints to 
 the inductive proof that Day's result is sharp} \labbel{rev}  
Let us now comment a bit about the proof of 
Theorem \ref{daysh}.
We shall recall J{\'o}nsson and Day terms in Section \ref{prel}.
In both cases, the terms are required to satisfy distinct
conditions for even and odd indices.
If we exchange odd and even in the condition
for congruence distributivity, we get the so called alvin terms
\cite[Theorem 4.144]{MMT}. While a variety is congruence distributive 
if and only if it has J{\'o}nsson terms, if and only if it
has alvin terms,  we get different conditions, in general, if 
we keep the number of terms fixed \cite{FV}.   
Similarly, let us say that a variety $\mathcal {V}$ is $m$-reversed-modular
if  $\mathcal {V}$ has terms 
$u_0, \dots, u_m$
satisfying Day's conditions with odd and even exchanged.
See Definitions \ref{daydef} and \ref{nm} for precise details. 
We have results also for the reversed conditions.
Actually, the use of the
reversed conditions seems fundamental in our arguments;
 the proof of Theorem \ref{daysh}  proceeds through a simultaneous
induction
which deals alternatively 
with 
  \begin{enumerate}[(a)]    
\item   distributivity in combination with reversed modularity,
and 
\item
 alvin in combination with modularity.
  \end{enumerate} 
In the next theorem we state our main results
about the reversed conditions.

\begin{theorem} \labbel{dayshr}
Suppose that  $n \geq 4$ and $n$ is even.
   \begin{enumerate}[(i)]
    \item  
Every $n$-alvin variety is $2n{-}3$-reversed-modular.
\item
There is a locally finite  $n$-alvin variety which is not $2n{-}3$-modular,
 \end{enumerate} 
 \end{theorem}  

Theorem \ref{dayshr} follows from  
 Lemma \ref{dayoptlem}(a)  and Theorem \ref{buh}(ii)  which shall be proved below. 
Notice that the conclusions in Theorems \ref{day} and \ref{daysh}
deal with $2n-1$ and  $2n-2$, while Theorem \ref{dayshr}
deals with the different parameter   $2n-3$. 
The proof  of Theorem \ref{dayshr}(i)
uses the methods from Day \cite{D} with a small known
variation ``on the outer edges''.
An alternative proof using 
relation identities is presented in Section \ref{mixed}. 

The bounds in 
Theorems \ref{day} and  \ref{dayshr} 
are shown to be optimal by constructing appropriate counterexamples
by induction.     
In each case, the induction at step
$n$ uses the counterexample constructed for $n-2$
in the parallel theorem. We shall prove a slight
improvement of Theorem \ref{daysh}, to the effect that,
for $n \geq 2$ and $n$ even, there is an $n$-distributive variety
which is not $2n{-}1$-reversed-modular.
 Then the induction goes as follows: from an
$n{-}2$-alvin not $2n{-}7$-modular
variety we construct an
$n$-distributive variety
which is not  $2n{-}1$-reversed-modular.
In the other case,   
from an $n{-}2$-distributive variety
which is not  $2n{-}5$-reversed-modular
we construct an $n$-alvin not $2n{-}3$-modular
variety. 
Notice that the shift in the modularity level is
$6$ in the former case and $2$ in the latter case.  
On average, we get a shift by $8$ each time
$n$ is increased by $4$, in agreement with the 
statements of the results.   

A more explicit description of varieties furnishing the 
counterexamples is presented in Section \ref{moreex}.
The varieties introduced in Section \ref{moreex}
are described in a relatively simple way;
however, as far as we can see, there is no direct proof
that such varieties do the requested job.
Any possible proof seems to use and mimic---more or less in 
disguise---the general constructions performed
 in Section \ref{mainc} and the inductive arguments presented in
Section \ref{opti}. 
In fact, in a parallel situation, in \cite[Section 4]{misharp}
we worked the explicit details: the resulting computations
turned out to occupy almost the same space as the  
abstract general considerations, even though 
some details were only  hinted.
This is the reason why, in the present 
situation, we shall  exhibit  the
more abstract  constructions
in Section \ref{mainc}.  
Such constructions have also the advantage of furnishing a uniform treatment
for various different situations and have further applications
to some extended settings, as we shall hint 
in
Remarks \ref{merge} - \ref{bestnu},   \ref{rmkka} - \ref{expres}.

\subsection*{Directed and mixed conditions} \labbel{dam}    
We shall also deal with other  conditions equivalent to 
congruence distributivity.  
Kazda,  Kozik,   McKenzie, Moore
\cite{adjt} have studied a ``directed'' variant of J{\'o}nsson condition
and they proved that this directed J{\'o}nsson condition, too, is equivalent to 
congruence distributivity.
For the directed variant there is no distinction between 
even and odd indices.
We shall prove optimal results also
for the modularity levels of varieties with such directed J{\'o}nsson terms.
See Theorem \ref{dir}.
Furthermore, Kazda,  Kozik,   McKenzie and Moore
noticed that every variety with directed J{\'o}nsson 
terms\footnote{Warning: 
here we are using a different counting convention for the terms,
in comparison with \cite{adjt}.
As far as the techniques and results here are concerned, it will subsequently appear 
evident that this counting convention is terminologically
more convenient.
See Remark \ref{count} for a discussion of this aspect.}
 $d_0, \dots, d_n$ has J{\'o}nsson  
terms $t_0, \dots, t_{2n-2}$.
We shall show in Theorem \ref{dirthm}(ii)
that the above observation from \cite{adjt} 
cannot be sharpened, as far as the
number of terms is concerned. In the other direction,
a hard result from \cite{adjt} shows that from
some sequence of J{\'o}nsson terms one can construct a sequence
of directed J{\'o}nsson terms. 
In Theorem \ref{dirthm}(i) we show that, in general, 
the latter sequence cannot be taken
to be shorter than the former sequence. However it would
be really surprising if this turns out to be the best possible result.
See Remark \ref{dirmin}(b)
and the comment after Proposition \ref{4dir}.
  
The idea of directed J{\'o}nsson terms
suggests an even  more general
notion of mixed J{\'o}nsson terms; see Definition \ref{mix}.
Mixed terms
have been first introduced in Kazda and Valeriote \cite{KV}
using different notation and terminology.
Roughly,  a mixed condition
involves, for each index,  either the J{\'o}nsson 
or the directed 
or the specular 
directed 
equation,  with no  a priori prescribed pattern.
We study some basic facts about this notion in Section \ref{mixed},
proving that any kind of mixed condition in this sense
implies congruence distributivity.
We roughly evaluate the modularity and distributivity  levels 
which follow from  such  mixed conditions; in many special cases
we know that we have found the optimal values; 
actually, all the previously mentioned 
 results about modularity levels 
(on the positive side, not the counterexamples)
can be seen as special cases of 
our main result about 
mixed conditions, Corollary \ref{propmixmod}.

\subsection*{Gumm terms as defective alvin terms} \labbel{cwg} 
There is another observation suggesting that the notion of 
a mixed condition is interesting.
In order to explain this in detail,  we have to recall
the notions of Gumm and of directed Gumm terms.
An astonishing characterization
of congruence modularity has been found by
H.-P.\ Gumm \cite{G1,G}, using terms which 
``compose'' the conditions for permutability 
and distributivity.
 
Gumm terms can be seen 
from another
perspective. 
Observe that a Maltsev term for permutability 
can be considered as a Pixley term 
when the equation $x=t(x,y,x)$ is not assumed
(a Pixley term 
is the nontrivial term given by the  alvin condition in the case $n=2$.
See Definition \ref{jondef}).
In other words, a Maltsev term can be seen as 
a ``defective'' alvin term, for $n=2$.
Similarly, Gumm terms can be seen as
defective alvin terms, for possibly larger values of $n$. 
See \cite[Remark 4.2]{ntcm}
for a hopefully complete discussion of this aspect.
See also \cite{KV}, \cite[p.\ 12]{jds}
 and Remarks \ref{galt}(c) and \ref{spend} below. 
The explicit definition of Gumm terms shall be given in 
Definition \ref{defgumm}.

\subsection*{Directed Gumm terms as defective mixed terms} \labbel{dirgsub} 
 In \cite{adjt} a directed variant of Gumm terms
has been introduced, too. While, as 
we have just mentioned,  it is possible
to introduce Gumm terms as defective alvin terms,
on the other hand, it is not possible to define  directed
Gumm terms
as defective  directed J{\'o}nsson terms.
Actually, 
defective J{\'o}nsson (directed or not) terms generally
provide a trivial condition. Cf.\  \cite{KV} and   Remark \ref{gummrmk}(c).
In order to get directed Gumm terms in the sense of 
\cite{adjt} we need to consider a mixed distributivity
condition in which the equations for directedness
are considered ``in the middle'', while an alvin-like condition  
  is taken into account on some ``outer edge''.
Directed Gumm terms are then obtained by considering the
defective version of such a mixed condition.
We can also deal with a more
symmetric condition in which an alvin-like condition
 is assumed on both
outer edges and then take the
defective variant on both edges.
In this way we get terms which are ``directed Gumm on both heads''. See 
Definition \ref{dirgumm}.
The above conditions appear
 interesting for themselves and we evaluate exactly 
the modularity level they imply in Theorem \ref{thm2hg}.
The above discussion suggests the naturalness of the more
general mixed conditions. See \cite{KV} and  Section \ref{mixed}.

\subsection*{Specular terms} \labbel{spec}    
As another generalization, our constructions can all be made 
more symmetrical, in
the sense that the terms we construct
can be always chosen in such a way that they satisfy
the ``specular'' condition
\begin{equation*} 
t_i(x,y,z) = t_{n-i}(z,  y,x),
\end{equation*}    
for all indices $i$.
Thus we get still other conditions
implying congruence distributivity.
We show in Section \ref{specm} that,
for every form
of distributivity under consideration,
in the case when the index $n$ is even
 our results turn out to be essentially the same 
if we impose the above condition of specularity.

\subsection*{A brief summary} \labbel{bsum}    
In detail, the paper is divided as follows.
In Section \ref{prel} we recall some basic notions
about congruence distributive and congruence modular 
varieties; in particular, we recall J{\'o}nsson 
and Day terms, together with some variants.
 We stress the useful fact that many conditions
can be translated in the form of a congruence identity.
See Remarks \ref{conid}, \ref{conidcomm}  and \ref{conidmod}. 

In Section \ref{mainc} we present our main constructions.
Roughly, starting from, say, an $n$-distributive  algebra,
we shift the J{\'o}nsson terms and add trivial projections 
both at the beginning and   at the end. Since the original
J{\'o}nsson terms are shifted by $1$, the role of odd
and even is exchanged,  thus we obtain
an $n{+}2$-alvin algebra; this is the trivial part of the argument.
 We then 
consider another $n{+}2$-alvin algebra
by taking 
 an appropriate reduct of some Boolean algebra
(Construction \ref{ba}).
If the types of the above two algebras are
arranged in such a way that they match, then
their product is $n{+}2$-alvin, too.
We also need a parallel construction 
which starts from an $n$-alvin algebra.
Adding trivial projections, we get
 an $n{+}2$-distributive algebra,
then we take the product with an appropriate reduct
of a distributive lattice (Construction \ref{bak}).
Let us call  $\mathbf E$ anyone of the above products.
 The most delicate part of the argument
(Construction \ref{c} and Theorem \ref{thmB})
allows us to find a subalgebra $\mathbf B$
of $\mathbf E$  
in such a way that 
$\mathbf B$ witnesses the failure of $m$-modularity, for 
the desired $m$.  
 We present general  conditions
 ensuring that a certain  subset  $B$ 
of some algebra $\mathbf E$ constructed as above
is the universe for some substructure.
Such conditions do not necessarily involve
congruence distributivity and
 find applications in different contexts.
In fact, we shall occasionally deal with congruence identities
strictly weaker than congruence modularity. 
See, e.~g., Theorems \ref{buhbis}, \ref{dirthm}(iii),
Remarks \ref{merge},  \ref{polin}
and what we call the switch levels in 
Definition \ref{deflev}  
and Theorem \ref{sumup}.

In Section \ref{opti} the constructions from
Section \ref{mainc} are put together in order to 
show that Day's result is optimal for $n$ even.
Essentially, we present 
the details for the induction we have hinted 
in the above subsection 
about reversed conditions.
Sections \ref{dirsec}, \ref{specm}, \ref{gummdef} and  
\ref{mixed} 
deal, respectively, with directed,
 specular, Gumm and mixed terms.
Sections \ref{dirsec} - \ref{mixed} rely heavily on 
Sections \ref{mainc} and \ref{opti}, but are largely independent one from another.
In Section \ref{moreex} we present a more explicit description
of the varieties furnishing our counterexamples,
summing up most of our results, actually, adding a bit more.
Finally,  Section \ref{fur} is reserved for  additional remarks
and problems.

\section{A review of congruence distributive and modular varieties} \labbel{prel} 

\subsection*{Aspects of congruence distributivity} \labbel{subsdist}
For later use, we insert the definitions
of J{\'o}nsson, alvin and directed J{\'o}nsson terms
in a quite general context.
Since our main concern
are terms and equations,
 in what follows we shall be  somewhat informal and 
say 
that a sequence of terms satisfies some set of equations
to mean that some variety under
consideration, or some algebra, have such 
a sequence of terms
 and the equations are satisfied 
in the variety or in the algebra.
When no confusion is possible,
 we shall speak of, say, J{\'o}nsson terms, instead of  
using the more precise expression
  \emph{sequence of J{\'o}nsson terms}. 
In case the terms are actually operations
of the algebra or of the variety under consideration, we shall
sometimes say that the algebra or the variety has 
J{\'o}nsson operations, to mean
that the operations
satisfy the corresponding equations.

\begin{definition} \labbel{jondef}    
Fix some natural number $n$ and 
suppose that 
$t_0, \dots, t_n$ is a sequence of $3$-ary terms. 
In the present section, and for most of the paper,
 all the sequences of 
$3$-ary terms under consideration will satisfy
\emph{all} the following basic equations
\begin{equation}
\labbel{bas}    \tag{B} 
\begin{aligned}    
  x&= t_0(x,y,z),  \qquad \qquad t_{n}(x,y,z)=z, 
 \\
  x  &=t_h(x,y, x), \quad 
\text{for } 0 \leq h \leq n,
\end{aligned}
 \end{equation}       
as well as \emph{some} appropriate equations from the following list
\begin{align}
 \labbel{m1} \tag{M$\xfix$} 
 t_{h}(x,x,z) &=
t_{h+1}(x,x,z),
\\ 
 \labbel{m3} \tag{M$\zfix$} 
 t_{h}(x,z,z) &=
t_{h+1}(x,z,z),
 \\ 
 \labbel{m2} \tag{M$\ztox$}
 t_{h}(x,z,z) &=
t_{h+1}(x,x,z).
\end{align}   

We now define precisely the relevant conditions.

\smallskip 

 \emph{(J{\'o}nsson terms)} The sequence $t_0, \dots, t_n$
is a sequence of \emph{J{\'o}nsson terms} \cite{JD}  
for some variety, or even for a single algebra,
if the sequence satisfies the equations \eqref{bas} and 
\begin{equation}\labbel{j3} \tag{J}
\text{ equation \eqref{m1} for $h$ even,
  equation \eqref{m3} for $h$ odd, $0 \leq h < n$.} 
  \end{equation}    
If $t_0, t_1, t_2$ is a sequence of J{\'o}nsson  terms, then
$t_1$ is a  \emph{majority term}.

\smallskip 

\emph{(Alvin terms)}
If we exchange the role of even and odd in 
the  definition of J{\'o}nsson terms,
 we get a sequence of alvin terms.
In detail, a sequence $t_0, \dots, t_n$ is a sequence
of \emph{alvin terms}  \cite{MMT} if 
the sequence satisfies the equations \eqref{bas} and
\begin{equation}\labbel{al} \tag{A}
\text{ equation \eqref{m3} for $h$ even,
  equation \eqref{m1} for $h$ odd, $0 \leq h < n$.} 
  \end{equation} 
If $t_0, t_1, t_2$ is a sequence of alvin terms, then
$t_1$ is a \emph{Pixley term} for arithmeticity. Cf.\ \cite{P}.   
By the way, this suggests that,  even for larger $n$,  the alvin condition
shares some aspects in common with congruence permutability.
See Remarks \ref{galt}(c) and \ref{spend}  for further details.

\smallskip 

\emph{(Directed J{\'o}nsson terms)} 
 Finally, if we always use \eqref{m2}, we get a sequence 
of directed J{\'o}nsson terms. In detail,
a sequence of  \emph{directed J{\'o}nsson terms}
\cite{adjt,Z} is a sequence which satisfies \eqref{bas}, 
as well as \eqref{m2}, for all $h$,  $0 \leq h < n$.
In the case of directed terms there is no distinction
between even and odd $h$'s.

A sequence   $t_0, t_1, t_2$ is a sequence of
directed J{\'o}nsson terms if and only if 
it is a sequence of J{\'o}nsson terms.
On the other hand, we shall see that the notions are distinct for
larger $n$'s. 
 \end{definition}

Notice that if some algebra $\mathbf A$ has, say,
J{\'o}nsson terms $t_0, \dots, t_n$, then 
the variety $\mathcal {V}$ generated by $\mathbf A$ has 
J{\'o}nsson terms $t_0, \dots, t_n$.
Thus the above notions are actually notions about varieties.
However, in certain cases, as a matter of terminology,  it will be 
convenient to deal with algebras.

\begin{theorem} \labbel{Jthm}
\cite{JD,adjt,MMT} 
For every variety $\mathcal {V}$, the following conditions are equivalent. \begin{enumerate}[(i)] 
  \item 
$\mathcal {V}$ is congruence distributive.
\item
$\mathcal {V}$ has a sequence of J{\'o}nsson terms.
\item
$\mathcal {V}$ has a sequence of alvin terms.
\item
$\mathcal {V}$ has a sequence of directed J{\'o}nsson terms.
  \end{enumerate} 
 \end{theorem} 

The equivalence of (i) and (ii) is due to J{\'o}nsson \cite{JD}.
The equivalence of (ii) and (iii) is easy and almost
immediate (compare the proof of (ii) $\Leftrightarrow $  (iii)
in Theorem \ref{Dth} below).
Anyway, the equivalence of (i) and (iii)  appears explicitly in \cite{MMT}.
The equivalence of (ii) and (iv) is proved in \cite{adjt}.

\smallskip 

For a given congruence distributive variety, the lengths of  the shortest sequences
given by  (ii) - (iv) above might be different. Henceforth
it is interesting to classify varieties  according to such lengths,
both in the case of congruence distributivity, as well as  in parallel situations. 
See, e.\ g., \cite{D,FV,Gr,JD,adjt,csmc,LTT,LGA,ntcm}.
See also Theorem \ref{sumup} below.

\begin{definition} \labbel{nd}   
A variety or an algebra  is \emph{$n$-distributive}
(\emph{$n$-alvin},
  \emph{$n$-directed-distributive}) 
if it has a sequence
$t_0, \dots, t_n$
of J{\'o}nsson (alvin, directed J{\'o}nsson) terms.
\end{definition}

\begin{remark} \labbel{count}    
\emph{(Counting conventions.)} 
Notice that each of the conditions
in Definition \ref{nd} 
actually involves $n+1$ terms,
including 
 the
two projections $t_0$ and  $t_n$, so that 
the number of nontrivial terms is $n-1$. 
This aspect might ingenerate terminological 
confusion and the present author apologizes
for having sometimes contributed to 
the confusion.
For example, an $n$-directed-distributive variety
in the above sense has been called
a variety with $n+1$ directed J{\'o}nsson terms in \cite{jds},
while in other works such as \cite{ia} we have called it a variety 
with $n-1$ directed J{\'o}nsson terms, counting only
the nontrivial terms. The latter is also the counting convention
adopted in \cite{adjt}. However, here it is extremely convenient to maintain
a strict parallel with the universally adopted terminology concerning
``undirected'' $n$-distributivity.  

Of course, the ``outer'' terms
$t_0$ and $t_n$ 
are trivial projections,
hence it makes no substantial difference whether they 
are listed or not in the above conditions.
However, here it is notationally convenient to include them, since, otherwise,
say in the case of the J{\'o}nsson condition,
we should have divided the definition into two cases,
asking for $t_{n-1}(x,z,z)=z$ for $n$ even,
and  for $t_{n-1}(x,x,z)=z$ for $n$ odd.
Considering the trivial terms $t_0$ and $t_n$,
instead,  we can fix
equations \eqref{bas} once and for all, independently
of the kind of terms we are dealing with and, in particular, independently
from the parity of $n$.

However, let us point out that, in a different context, it will
be convenient not to include $t_0$ and $t_n$. See Remark \ref{lrr}.
We shall try to be consistent, and, in case
we do not include the trivial projections, we shall list the terms
as $t_1, \dots, t_{n-1}$, so that the value of $n$
will not change.

We have insisted on this otherwise marginal aspect 
since we shall have frequent occasion to shift the indices of the terms
in use, hence particular care is needed in numbering and labeling.
See Remark \ref{merg}, Constructions \ref{bak}, \ref{bakbis}, \ref{ba},
Lemma \ref{lembak}, Theorem \ref{thmba}(i),
Definition \ref{simpldef3}     
and the proofs of Theorems \ref{Dth} and  \ref{sumup}.

By the way, there is a deeper reason showing that the above
counting terminology is necessarily 
conventional. Using arguments from \cite{Ta}
it can be shown that every idempotent  Maltsev 
condition  can be expressed
by conditions involving a single term. In this situation, the 
varying quantity is the number
of variables, not the number of terms.
  \end{remark}

\begin{remark} \labbel{conid}
It is frequently convenient to translate the above 
notions in terms of congruence identities.
For $\beta$ and $\gamma$ congruences, let
$\beta \circ \gamma \circ {\stackrel{k}{\dots}}  $
denote the relation $ \beta \circ \gamma \circ \beta \circ \gamma \circ \dots$
with $k$ factors, that is, $k-1$ occurrences of  $ \circ $.
If we know that, say, $k$ is even, then, according
to convenience, we  write
$\beta \circ \gamma \circ {\stackrel{k}{\dots}} \circ \gamma  $
 to make clear that $\gamma$ is the last factor.
It is also convenient to consider  the extreme cases, so that
$ \beta \circ   \gamma \circ {\stackrel{1}{\dots}}
= \beta   $ and
$  \beta \circ  \gamma \circ {\stackrel{0}{\dots}}
= 0$, where $0$ is the minimal congruence in the algebra under consideration. 

In congruence identities
we use  juxtaposition to denote intersection. 
     
Under the above notation, 
we have that, within a variety, each condition on the left in the following table
is equivalent to the condition on the right.
\begin{align*}
& \text{ $n$-distributive  }   
&&
 \alpha ( \beta \circ \gamma )
 \subseteq  
\alpha \beta \circ \alpha \gamma \circ {\stackrel{n}{\dots}} 
\\ 
& \text{ $n$-alvin  }   
&& 
\alpha ( \beta \circ \gamma )
 \subseteq 
 \alpha \gamma  \circ  \alpha \beta  \circ {\stackrel{n}{\dots}} 
\end{align*}    
 
The equivalence of the above conditions  is now a standard fact
\cite{Pal,W}. In the specific case
at hand, as well as in similar situations
described below,
the proof is quite simple and direct, see, e.~g.,  
 \cite{CV,contol,jds,ntcm,T}  for examples and further comments.
See also Remark \ref{useofbm}(b) below.
Notice that the conditions are equivalent only within a variety;
indeed, the conditions on the right are locally weaker.  
If some algebra $\mathbf A$ satisfies one of the conditions on the
right, it is not even necessarily the case that 
$\mathbf {Con} (\mathbf A)$ is distributive, let alone
the request that $\mathbf A$ generates a congruence distributive 
variety.

We have stated the above conditions in the form of inclusions,
but notice that an inclusion 
$X \subseteq  Y$ is (set-theoretically)
equivalent to the identity 
$X=XY$, hence we are free to use the expression 
``identity''. 

There seems to be no immediate directly provable condition 
equivalent to the existence of directed J{\'o}nsson terms and which can be
expressed in terms
of congruence identities.
Nevertheless, directed terms are involved in the study of relation identities, see
\cite[Section 3]{jds} and \cite[Section 3]{ia}. See also Section \ref{mixed} below.
 \end{remark}
 
\begin{remark} \labbel{conidcomm}   
 Expressing Maltsev conditions in terms of congruence identities
as in Remark \ref{conid} 
is particularly useful. 

(a)
 For example, from the above characterizations
we immediately get the well-known fact 
that, for $n$ odd, $n$-distributive and  $n$-alvin are equivalent 
conditions;
just take converses and exchange $\beta$ and $\gamma$.
When dealing with the conditions involving terms, 
one obtains the equivalence by
reversing both the order of variables and of terms \cite[Proposition 7.1(1)]{FV},
 but this seems
intuitively less clear. 

(b) As another example, arguing in terms of congruence identities
it is immediate to see that $n$-distributive implies 
$n{+}1$-alvin, and symmetrically that 
  $n$-alvin implies 
$n{+}1$-distributive. In fact, to prove, say,
the former statement, just notice that
$\alpha \beta \circ \alpha \gamma \circ {\stackrel{n}{\dots}} 
\subseteq 
\alpha \gamma \circ \alpha \beta \circ {\stackrel{n+1}{\dots}} $. 

See also Remarks \ref{conidmod}, \ref{daystrong},
\ref{conidabog}, \ref{gummrmk} and \ref{equiv}  for related observations.
In fact, in our basic Theorems \ref{thmbak}, \ref{thmbakbis}, \ref{thmba}, 
and \ref{buhbis} we shall deal with congruence identities
and then we shall use Remark \ref{conid} and the corresponding
Remark \ref{conidmod} below about congruence modularity
 in order to translate  the basic results
in terms of distributivity and modularity levels.
For example, this technique applies to 
Theorems \ref{buh}, \ref{dir}(i), \ref{dirthm} and \ref{thm2hg}(ii),
as well as to the main results in \cite{LGA,misharp}.

Identities dealing with reflexive and admissible relations 
shall be heavily used  in Section \ref{mixed}.
See, in particular, Theorem \ref{propmix2} and its consequences. 
 \end{remark}    

\subsection*{Aspects of congruence modularity} \labbel{subsmod}
Now for A.\ Day's characterization of congruence modularity.

\begin{definition} \labbel{daydef}    
A sequence of \emph{Day terms} \cite{D}  
for some variety, or even for a single algebra,
is a sequence
$u_0, \dots, u_m$, for some $m$, of $4$-ary terms   satisfying
\begin{align}
 \labbel{d0}    \tag{D0}      
x &=u_k(x,y,y,x),  \quad \quad  \ \,    \text{ for  } 0 \leq k \leq m,  
\\
\labbel{d1}    \tag{D1} 
 x&=u_0(x,y,z,w), 
\\ 
\labbel{d2}    \tag{D2}
\begin{split}     
 u_k(x,x,w,w)&=u_{k+1}(x,x,w,w), \quad \text{ for  even $k$, }  0 \leq k < m, 
\\
 u_k(x,y,y,w)&=u_{k+1}(x,y,y,w), \quad \; \text{ for  odd $k$, }  0 \leq k <  m,
\end{split}
\\ 
\labbel{d3}    \tag{D3}
 u_{m}(x,y,z,w)&=w.
\end{align}

If we exchange even and odd in \eqref{d2}
we get a sequence of \emph{reversed Day terms}.  
\end{definition}

\begin{theorem} \labbel{Dth}
\cite{D} 
For every variety $\mathcal {V}$, the following conditions are equivalent. \begin{enumerate}[(i)] 
  \item 
$\mathcal {V}$ is congruence modular.
\item
$\mathcal {V}$ has a sequence of Day terms.
\item
$\mathcal {V}$ has a sequence of reversed Day terms.
 \end{enumerate} 
 \end{theorem} 

\begin{proof}
The equivalence of (i) and (ii) is due to Day \cite{D}.

If $u_0, \dots, u_m$ is a sequence of Day 
terms (reversed Day terms), 
then we get a sequence 
$u'_0, \dots, u'_{m+1}$
of reversed Day terms
(Day terms)
by taking $u'_0$ 
to be the projection onto the first coordinate and 
$u'_{k+1} = u_k$, for $k  \geq 0$. 
 Hence (ii) and (iii) are equivalent.
Otherwise, one can apply 
the second statement in Proposition \ref{modeq}  below.  
\end{proof}
 
We shall recall further conditions equivalent to 
congruence modularity in Theorem \ref{Gth} below.  

\begin{definition} \labbel{nm}   
A variety or an algebra  is \emph{$m$-modular}
(\emph{$m$-reversed-modular}) 
if it has a sequence
$u_0, \dots, u_m$
of Day  (reversed Day) terms.
 \end{definition}

\begin{remark} \labbel{conidmod}
Day's condition, too, can be  translated in terms of congruence identities.
Within a variety, each condition on the left in the following table
is equivalent to the condition on the right.
\begin{align*}
& \text{ $m$-modular  }   
&&
 \alpha ( \beta \circ \alpha  \gamma \circ \beta )
 \subseteq  
\alpha \beta \circ \alpha \gamma \circ {\stackrel{m}{\dots}} 
\\ 
& \text{ $m$-reversed-modular  }   
&& 
 \alpha ( \beta \circ \alpha  \gamma \circ \beta )
 \subseteq 
 \alpha \gamma  \circ  \alpha \beta  \circ {\stackrel{m}{\dots}} 
\end{align*}    
 \end{remark}  

The above congruence identities explain the reason for our 
choice of the expression ``reversed modularity''.
  
Arguing as in Remark \ref{conidcomm} we get the  
 facts stated in the following proposition. As far as the second
statement is concerned, compare also the proof of the equivalence
of (ii) and (iii) in Theorem \ref{Dth}.

\begin{proposition} \labbel{modeq}
If $m$ is even, then $m$-modularity 
and $m$-reversed-modularity are equivalent notions.

For every $m>0$, we have that
$m{-}1$-modularity 
implies $m$-reversed-modularity
and that 
$m{-}1$-reversed-modularity 
implies $m$-modularity.
 \end{proposition}   

In Corollary \ref{corsumupmod}
we shall show that Proposition \ref{modeq}
cannot be improved for $m \geq 4$.

\section{The main constructions} \labbel{mainc}

\begin{remark} \labbel{merg}    
If $n \geq 2$ and  some algebra $\mathbf D$
has alvin operations 
$s_0, \dots, s_{n-2}$, then, by relabeling the operations
as $t_1=s_0, \dots, t _{n-1} = s_{n-2} $ and taking 
$t_0$ to be the  projection onto the first coordinate, we get a sequence 
 $t_0, \dots, t _{n-1}  $ of J{\'o}nsson operations,
since then the role of even and odd
is exchanged.
At the end, we can possibly add
$t_n$,  taken to be the  projection
onto the third coordinate, 
getting a longer sequence  $t_0, \dots, t _{n}  $
of J{\'o}nsson operations.
Let $\mathbf A_4$ denote the algebra with
the relabeled operations
(the labels and the ordering of the operations
will be highly relevant in all the arguments below). 
If $\mathbf A$ is an algebra 
with  J{\'o}nsson operations
 $t_0, \dots, t _{n}  $, 
then also the product $\mathbf A\times \mathbf A_4$
has J{\'o}nsson operations
 $t_0, \dots, t _{n}  $, provided the types of the algebras
are arranged in a such a way they match.

For short, we have showed that we can combine an
$n{-}2$-alvin algebra with an
$n$-distributive algebra, getting a new
$n$-distributive algebra.
Though the added operations are trivial,
we shall see  that the construction provides nontrivial results.
In fact, all our counterexamples will be constructed in this way.
Symmetrically, 
from an $n{-}2$-distributive algebra and an
$n$-alvin algebra we can  get another
$n$-alvin algebra.

Of course, in the general case, an alvin
algebra $\mathbf D$ has alvin terms, not
necessarily operations.
However, for our purposes, it will be an inessential modification
to consider the expansion of 
$\mathbf D$ in which the alvin terms are added as operations.
Actually, it will be always  no loss of generality to assume
that $\mathbf D$ has no other  operation.
This assumption will simplify
arguments and notation.
 \end{remark}

\subsection*{Constructing subalgebras} 
The relevant point in our constructions
is to find some appropriate subalgebra $B$ 
of an algebra like $\mathbf A \times \mathbf A_4$
described in Remark \ref{merg}.
Actually, the algebra $\mathbf A$ will be constructed as  the product
of three algebras of the same type, 
but of course this  makes no essential difference,
as far as Remark \ref{merg} is concerned.
The arguments showing that some $B$ as  above 
is indeed a subalgebra use really weak hypotheses, 
so we present them in  generality.
Here the main advantage  of this abstract treatment is that 
the method  works uniformly for J{\'o}nsson,
directed and Gumm  terms. As already mentioned,
 we shall subsequently hint to  further applications.

\begin{construction} \labbel{c}
Fix some natural number $n \geq 3$.  
In what follows the number $n$ will be always explicitly declared
in all the relevant places,
 hence  we shall not indicate it
in the following notation.
A similar remark applies to Constructions
\ref{bak}, \ref{bakbis} and  \ref{ba} below.

\emph{(A) Premises.}
We suppose that $\mathbf A_1$, $\mathbf A_2$,
$\mathbf A_3$ and $\mathbf A_4$ are algebras
with only the ternary operations 
$t_1^{\mathbf A_j}$, $t_2^{\mathbf A_j}$, \dots, 
  $t_{n-1} ^{\mathbf A_j} $,   $j=1,2,3,4$.
Here and in similar situations we shall omit 
the $j$-indexed  
 superscripts
when there is no danger of confusion. 
We further suppose that
the first three algebras
have a special element $0_j \in A_j$, for $j=1,2,3$.
Again, we shall usually omit the subscripts.
It is not necessary to assume that there is some constant 
symbol which is interpreted as the
 $0_j$'s, it is enough to assume the existence
of such elements.

We require that,  for $j=1,2,3$, the following equations hold
in   $\mathbf A_j$,
for all $x,y,z \in A_j$:
\begin{align}\labbel{a}
 0  =t_h(0,y,z),   \qquad &\text{ for $h=1, \dots, n-2$,}
\\ 
\labbel{z}
 t_h(x,y,0) =0,  \qquad  &\text{ for $h=2, \dots, n-1$}.
\end{align}     

The algebra $\mathbf A_4$, instead,  is supposed to satisfy: 
\begin{gather}\labbel{t1}      
\text{ $t_1$ is the projection onto the first coordinate, } 
\\
\labbel{tn}   
\text{ $t_{n-1}$ is the projection onto the third coordinate, } 
\\
\labbel{b}
x=t _h(x,y,x),  \qquad \text{ for $h=2, \dots, n-2$ and all $x,y \in A_4$.} 
\end{gather}

\emph{(B) A useful subalgebra.}
We shall use the above equations in order to show that 
a certain subset $B$ of 
$\mathbf E = \mathbf A_1 \times \mathbf A_2 \times 
\mathbf A_3 \times \mathbf A_4 $ 
 is a subalgebra.
Let $a,d$ be two arbitrary but fixed elements of 
$A_4$. Let $B=B(a,d)$ be the set of  those elements of $E$ 
which have (at least) one of the following forms:
\begin{equation*}
\begin{gathered}
\text{Type I} \\
(\hy, 0, \hy, a)
\end{gathered} 
\qquad\qquad
\begin{gathered}
\text{Type II} \\
(0, 0, \hy, \hy),
  \end{gathered} 
\qquad\qquad
\begin{gathered}
\text{Type III} \\
(0, \hy, \hy, d)
  \end{gathered} 
\qquad\qquad
\begin{gathered}
\text{Type IV} \\
(\hy, \hy, 0, \hy)
  \end{gathered} 
\end{equation*}    
where places denoted by $\hy$
can be filled with arbitrary elements from the 
appropriate algebra. 
 \end{construction}   

\begin{theorem} \labbel{thmB}
Under the assumptions and the definitions in Construction
\ref{c},    
the set $B$ is the universe for a subalgebra 
$\mathbf B$ of 
$\mathbf E$. 
\end{theorem}

 \begin{proof} 
The set $B$ is closed under 
$t_1$, since if 
$x \in E$ has one of the types 
I - IV, then $t_1(x,y,z)$
has the same type,   because of  
 equation  \eqref{a}. 
In case of types I and III we 
need also \eqref{t1}. 

Symmetrically, $B$ is closed under 
$t_{n-1}$. In this case,
$t_{n-1}(x,y,z)$
has the same type 
 of $z$ and we are using   \eqref{z} 
and \eqref{tn}. 

Now let $h \in \{ 2, \dots, n-2 \} $
and $x,y,z \in B$.  

If $x$ has type II or IV,
then $t_h(x,y,z)$ has the same type of $x$,
applying again  \eqref{a}.

Suppose that $x$ has type I.
We shall divide the argument into cases, considering the
possible types of $z$.      
If $z$ has type I, too, then  
$t_h(x,y,z)$ has type I, by \eqref{a} and \eqref{b}. 
We need \eqref{b} to ensure that the fourth component
is $a$. If $z$ has type II or IV, then    
$t_h(x,y,z)$ has the same type of $z$, by \eqref{z}.
Finally, if $z$ has type III, then  
$t_h(x,y,z)$ has type II. Indeed, the second component
is $0$ by \eqref{a}, since $x$ has type I, and
the first component is $0$ by \eqref{z},
 since $z$ has type III.    

The case in which  $x$ has type III is treated in a symmetrical way.     

We have showed that $B$ is closed with respect to 
$t_1$, $t_2$, \dots, $t_{n-1}$, hence $B$ is the universe
for  a subalgebra of $\mathbf E$.     
Notice that we have not used the assumption that
$y \in B$.
\end{proof}

\subsection*{Alvin in the middle} 
If $k \geq 1$, we define the lattice   $\mathbf C_{k}$ 
to be the $k$-elements chain 
with underlying set 
$C_{k} = \{0, \dots, k-1 \} $, with
 the standard ordering and
the standard lattice operations $\min$ and $\max$.
Lattice operations shall be denoted by 
juxtaposition and $+$.

\begin{construction} \labbel{bak}
Fix some natural number $n \geq 3$.
For every lattice $\mathbf L$, let     $\mathbf L^r$
be the following  term-reduct
  of  $\mathbf L$.
The operations of 
$\mathbf L^r$
are 
\begin{multline*} 
t_1(x,y,z) = x(y+z), \qquad
t_2(x,y,z) = xz, \qquad
t_3(x,y,z) = xz, \qquad  \dots, \qquad 
\\
\dots, \qquad t_{n-2}(x,y,z) = xz,
\qquad t_{n-1}(x,y,z) = z(y+x).
\end{multline*}   

In particular, 
the above definition 
applies to the lattices  $\mathbf C_{k}$.
 Limited to the present construction, 
we let $\mathbf A_1=\mathbf A_2= \mathbf C_{4}^r$
and   $\mathbf A_3 = \mathbf C_{2}^r$.

Suppose that $\mathbf D$
is an algebra with  ternary operations  
$s_0, \dots, s_{n-2}$.
As in Remark \ref{merg}, relabel the operations
as $t_1=s_0, \dots, t _{n-1} = s_{n-2} $ and let
$\mathbf A_4$ be the resulting algebra.
Assume that $\mathbf A_4$ satisfies Conditions
\eqref{t1}, \eqref{tn} and  \eqref{b}.
Notice that Conditions
\eqref{t1}, \eqref{tn} and  \eqref{b} are satisfied 
in case $s_0, \dots, s_{n-2}$ 
satisfy Condition \eqref{bas}
in Definition \ref{jondef}, with 
$n-2$ in place of $n$. 
In particular this applies when
$s_0, \dots, s_{n-2}$
are alvin operations for 
$\mathbf D$ and when $s_0, \dots, s_{n-2}$ are 
directed J{\'o}nsson  operations for 
$\mathbf D$.

Then the  algebras $\mathbf A_1$,
$\mathbf A_2$,
$\mathbf A_3$ and 
$\mathbf A_4$ satisfy the assumptions in
Construction \ref{c},
hence, by Theorem \ref{thmB},
for every choice of elements $a,d \in A_4$, 
 we have an algebra $\mathbf B = \mathbf B(a,d)$ 
 constructed as in \ref{c}(B). 
 \end{construction}   

We shall also use a small variation on Construction
\ref{bak}. This variation is not necessary in order  
to prove Theorem \ref{daysh} and might be skipped at first reading. 

\begin{construction} \labbel{bakbis}
Suppose that $n \geq 3$.
Consider a construction similar to \ref{bak}
above, with the only difference that  
$\mathbf A_1 $ and $ \mathbf A_2 $ are $  \mathbf C_{3}^r$,
rather than $\mathbf C_{4}^r$, while
 $\mathbf A_3= \mathbf C_{2}^r$ remains the 
 same algebra as in Construction
\ref{bak}.
Under the same assumptions on  $\mathbf D$
and $\mathbf A_4$,
 the  algebras $\mathbf A_1$,
$\mathbf A_2$,
$\mathbf A_3$ and 
$\mathbf A_4$ satisfy the assumptions in
Construction \ref{c},
hence, by Theorem \ref{thmB},
for every choice of elements $a,d \in A_4$, 
 we have an algebra $\mathbf B = \mathbf B(a,d)$ 
 constructed as in \ref{c}(B). 
 \end{construction}   

\begin{lemma} \labbel{lembak}
Under the assumptions and the definitions either from Construction 
\ref{bak} or from Construction \ref{bakbis}, the following hold:
  \begin{enumerate}[(i)]
    \item  
 If  $n$ is even and $s_0, \dots, s_{n-2}$ are alvin operations for 
$\mathbf D$, then, for every choice of  $a,d \in A_4$,
the algebra $\mathbf B$ is $n$-distributive.  

\item
 If  $s_0, \dots, s_{n-2}$ are directed J{\'o}nsson  operations for 
$\mathbf D$, then, for every choice of  $a,d \in A_4$,
the algebra $\mathbf B$ is $n$-directed-distributive.  
  \end{enumerate}
 \end{lemma}

 \begin{proof}
It follows from a remark in Construction \ref{bak}
that both in case (i) and in case (ii) the assumptions
in  Construction \ref{c} are satisfied, hence it makes sense to talk of 
$\mathbf B$, by Theorem \ref{thmB}.  
 
(i) Letting $t_0$ be the  projection onto the 
first coordinate and $t_n$ be the projection
onto the third coordinate, it is immediate to see that,
for every lattice $\mathbf L$, 
the terms $t_0$, $t_n$
and the operations of $\mathbf L^r$ 
satisfy J{\'o}nsson conditions.
In particular, this applies when 
$\mathbf L$ is 
$\mathbf C_{2}$,  $\mathbf C_{3}$ or $\mathbf C_{4}$. 
Here we are using the assumption that
$n$ is even, thus $n-1$ is odd,
hence $t_{n-1}$ satisfies the desired
equations $t_{n-2}(x,x,z) = xz = t_{n-1}(x,x,z)  $ 
and $ t_{n-1}(x,z,z) = z $. Notice also that the assumptions
imply $n \geq 4$, thus  $t_{n-2}(x,y,z) $ is actually $ xz $. 

It follows that
the terms $t_0, \dots, t_n$
are J{\'o}nsson on $\mathbf A_1 $, $ \mathbf A_2$ and 
$ \mathbf A_3$.
Since  $\mathbf D$  has alvin operations,
by the assumption in (i), and since
the indices of the original alvin operations of $\mathbf D$ 
are shifted by $1$, we have that the operations on 
$ \mathbf A_4$ satisfy J{\'o}nsson conditions, too,
adding again  the projections
$t_0$ and $t_n$.
Hence the product
$\mathbf A_1 \times \mathbf A_2
\times \mathbf A_3 \times \mathbf A_4$,
as well as any subalgebra,  
 are $n$-distributive.

(ii) Considering again
 the projections
$t_0$ and $t_n$,
the terms $t_0, t_1, \dots, t_n$
given by Constructions \ref{bak} or \ref{bakbis} 
are directed J{\'o}nsson  terms in every lattice.
Notice that this applies also in case $n=3$. 
 Adding the projection onto the first coordinate at the beginning
does not affect the conditions defining
 directed J{\'o}nsson terms and the same holds
adding the projection onto the third coordinate at the end.
Hence $t_0, t_1, \dots, t_n$ are directed J{\'o}nsson terms
for $\mathbf A_4$, as well, hence this applies to any product and 
subproduct.
\end{proof}

Recall the notational
conventions introduced in Remark \ref{conid},
in particular, recall that 
in congruence identities
juxtaposition denotes intersection.  
In what follows, an \emph{expression}
$\chi$ 
is a term in the language
$ \{ \circ, \cap \} $. We shall mention in Remark \ref{expres} below
that our results apply to a much more general context.
In what follows, for simplicity, we shall deal with congruences
$\tilde{\alpha}$, $\alpha$, \dots, but, as we shall point out
in Remark \ref{rmkkb}, many theorems below 
apply  also to tolerances and to 
reflexive admissible relations.

Formally, the statements of Theorems \ref{thmbak}, \ref{thmbakbis}
and \ref{thmba}  below
involve some natural  number $n$,
but notice that $n$ makes no essential appearance  
in the results.

\begin{theorem} \labbel{thmbak}
Let the assumptions and the definitions  from Construction 
\ref{bak} be in charge.
  \begin{enumerate} [(i)]
   \item   
 If there are congruences
$\tilde{\alpha}, \tilde{\beta}, \tilde{\gamma} $ of $  \mathbf A_4$
such that the identity 
\begin{equation*}
\tilde{\alpha}( \tilde{\beta} \circ \tilde{\alpha} \tilde{\gamma} \circ \tilde{\beta} )
\subseteq 
\tilde{\alpha} \tilde{\beta}
\circ 
\tilde{\alpha} \tilde{\gamma} 
 \circ {\stackrel{r}{\dots}} \circ    
 \tilde{\alpha} \tilde{\beta} 
 \end{equation*}    
 fails in $\mathbf A_4$,
for some odd $r$,  then there are $a,d \in A_4$ and 
congruences
$\alpha, \beta, \gamma $ of $\mathbf B = \mathbf B(a,d)$
such that the following identity  fails in $\mathbf B$.
\begin{equation*}\labbel{bakeq}      
\alpha( \beta \circ \alpha \gamma \circ \beta )
\subseteq 
\alpha \gamma    \circ  \alpha \beta 
\circ {\stackrel{r+6}{\dots}} \circ
\alpha \gamma.  
  \end{equation*}
\item
More generally, 
if 
the identity
$\tilde{\alpha}( \tilde{\beta} \circ \tilde{\alpha} \tilde{\gamma} \circ \tilde{\beta} )
\subseteq 
\chi (\tilde{\alpha}, \tilde{\beta},\tilde{\gamma} )$
 fails in $\mathbf A_4$,
for certain congruences $\tilde{\alpha}, \tilde{\beta},\tilde{\gamma} $
and some expression $\chi$, then 
there are $a,d \in A_4$ and 
congruences
$\alpha, \beta, \gamma $ of $\mathbf B = \mathbf B(a,d)$
such that the following identity  fails in $\mathbf B$.
\begin{equation}\labbel{bakeq2}     
\alpha( \beta \circ \alpha \gamma \circ \beta )
\subseteq 
 \gamma    \circ \alpha   \beta \circ \gamma 
 \circ 
\chi (\alpha, \beta ,  \gamma ) \circ   
 \gamma    \circ   \alpha \beta \circ \gamma.      
 \end{equation}
 \end{enumerate} 
\end{theorem}  

\begin{proof} 
Obviously (i) is immediate from (ii), taking 
$\chi( \alpha, \beta , \gamma ) = 
\alpha \beta \circ \alpha \gamma \circ {\stackrel{r}{\dots}} \circ
\alpha \beta    $ and since 
$ \alpha \gamma \circ \alpha \beta \circ \alpha \gamma \subseteq 
 \gamma \circ \alpha  \beta \circ \gamma $.  

In order to prove (ii), let $\tilde{\alpha}$, $\tilde{\beta}$ and $\tilde{\gamma}$ 
be congruences on $\mathbf A_4$
as given by the assumption.
Thus there are elements 
$a,d \in A_4$ such that 
$(a,d) \in  \tilde{\alpha}( \tilde{\beta} \circ \tilde{\alpha} \tilde{\gamma} \circ \tilde{\beta} )$
and $ (a,d) \not \in 
\chi (\tilde{\alpha}, \tilde{\beta},\tilde{\gamma} ) $.
Let $\mathbf B = \mathbf B(a,d)$.
It follows from the considerations in 
Construction 
\ref{bak} and from
Theorem \ref{thmB}
 that $\mathbf B $ is actually an algebra. 
We now consider two congruences
on the factors $\mathbf A_1 = \mathbf A_2$ and 
then construct appropriate congruences on $\mathbf B$. 
  Let
$\beta^*$ be the congruence on 
the lattice $\mathbf C_4$ whose blocks are 
$\{ 0,1\}$ and $\{ 2,3 \}$ and let 
$ \gamma^*$ be the congruence on 
$\mathbf C_4$ whose blocks are 
$\{ 0 \}$, $\{ 1,2 \}$  and $\{ 3 \}$.
Since $ \beta ^*$ and $ \gamma^*$ are congruences on  $\mathbf C_4$, they 
are also congruences on its term-reduct    
$\mathbf A_1 =\mathbf A_2$.
Let $0$ and $1$ denote, respectively, 
the smallest and the largest congruence
on any algebra under consideration.  
The congruence
$ \beta ^* \times  \beta ^* \times 1 \times \tilde{\beta} $ 
of $\mathbf A_1 \times \mathbf A_2
\times \mathbf A_3 \times \mathbf A_4$
induces a congruence $\beta$ on the subalgebra $\mathbf B$.
Similarly,
$ \gamma  ^* \times  \gamma  ^* \times 0 \times \tilde{\gamma} $ 
induces a congruence $\gamma$ on $\mathbf B$.
Finally, let $\alpha$ be induced on $\mathbf B$
by  $ 1 \times 1 \times 0 \times \tilde{\alpha}$. 

Since 
$(a,d) \in  \tilde{\alpha}( \tilde{\beta} \circ \tilde{\alpha} \tilde{\gamma} \circ \tilde{\beta} )$,
then $a \mathrel {\tilde{\alpha}}  d $ and  there are $b,c \in A_4$ 
such that $a \mathrel { \tilde{\beta}} b \mathrel { \tilde{\alpha} \tilde{\gamma}  } c 
\mathrel { \tilde{\beta} } d $.
Consider the following elements of $B$:  
 \begin{equation*}
 c_0 =(3,0,1, a),   \quad 
 c_1 =(2,1,0, b), \quad    c_{2} = (1,2,0,c),
 \quad c_{3} = (0,3,1, d).
\end{equation*} 
To see that the above elements  are actually in $B$, 
we need to recall
the definition of $B$ from Construction
\ref{c}(B).
For the reader's convenience we report here the definition:
$B$ is the set of the elements having (at least) one of the following forms: 
\begin{equation*}
\begin{gathered}
\text{Type I} \\
(\hy, 0, \hy, a)
\end{gathered} 
\qquad\qquad
\begin{gathered}
\text{Type II} \\
(0, 0, \hy, \hy),
  \end{gathered} 
\qquad\qquad
\begin{gathered}
\text{Type III} \\
(0, \hy, \hy, d)
  \end{gathered} 
\qquad\qquad
\begin{gathered}
\text{Type IV} \\
(\hy, \hy, 0, \hy)
  \end{gathered} 
\end{equation*}    

Thus $c_0, c_1,c_2, c_3 \in B $, 
since 
$c_0$ has type  I,
$c_1$ and $c_2$  have type  IV
and $c_3$ has type  III.
Moreover, 
$c_0 \mathrel \alpha  c_3$
and  
$c_0 \mathrel \beta c_1 \mathrel { \alpha \gamma  } c_2
\mathrel { \beta }   c_3$, thus
$(c_0,c_3) \in  \alpha( \beta \circ \alpha \gamma \circ \beta)$.
Suppose by contradiction  that
$(c_0,c_3)$ belongs to the right-hand side of \eqref{bakeq2}, thus
there are elements $g,h \in B$ such that   
 $(c_0, g) \in \gamma    \circ  \alpha  \beta \circ \gamma $,
$ (g,h) \in  \chi (\alpha, \beta , \gamma ) $
and  $(h,c_3) \in  \gamma    \circ  \alpha  \beta \circ \gamma $.
Then 
$ c_0 \mathrel { \gamma }   g_1 \mathrel { \alpha  \beta}
 g_2 \mathrel { \gamma} g $, for certain elements
$g_1$ and $g_2$.   By the $\gamma$-equivalence
of $c_0$ and $g_1$, we have that  
the first component of $g_1$ is $3$
and the third component of $g_1$ is $1$.
 By the $ \alpha  $-equivalence of $g_1$ and $g_2$,
the third component of $g_2$ is  $1$.
By the $ \beta $-equivalence of $g_1$ and $g_2$,
the first component of $g_2$ is either $3$ or $2$.
Similarly, the third component of $g$ is $1$, 
hence $g$ has not type IV. 
The first component of $g$ ranges between $1$ and   $3$, in particular
it is not $0$.
Since the first component of $g$ is not $0$,
$g$ belongs to $B$ and $g$ has not type IV, then $g$ has type $I$,
hence the fourth component of $g$ is $a$.

Symmetrically, the 
fourth component of $h$ is $d$. 
Since 
$(g,h) \in  \chi (\alpha, \beta , \gamma ) $, then, 
recalling the definitions
of $\alpha$, $\beta$ and $\gamma$,
and since $\chi$ is a $\{ \circ, \cap\}$-term,  
 we get 
$(a,d) \in \chi (\tilde{\alpha}, \tilde{\beta},\tilde{\gamma} ) $,
a contradiction.
\end{proof} 

\begin{remark} \labbel{ifin}
(a) Notice that the element $g_1$
from the proof of Theorem \ref{thmbak}(ii)
must have type I, hence the fourth component of 
$g_1$ is $a$, so we actually have $c_0=g_1$.     

(b) Using similar arguments, we have that, under the assumptions
in Theorem \ref{thmbak}(ii), the following identities fail
in $\mathbf B$ 
\begin{equation*}\labbel{bakeq3}     
\alpha( \beta \circ \alpha \gamma \circ \beta )
\subseteq 
\phi  \circ 
\chi (\alpha, \beta ,  \gamma ) \circ   
\psi   
 \end{equation*}
where each of $\phi$ and $\psi$ can be taken to be either
$\alpha( \gamma    \circ   \beta \circ \gamma )$,
$ \gamma    \circ \alpha (  \beta \circ \gamma )$,
$\alpha( \gamma    \circ   \beta) \circ \gamma $, or
$ \gamma    \circ  \alpha  \beta \circ \gamma $.
 \end{remark}

The next subsection  is not necessary in order  
to prove Theorem \ref{daysh}.

\subsection*{Bounds for $ \alpha ( \beta \circ \gamma )$ }
\begin{theorem} \labbel{thmbakbis}
Let
 the assumptions and the definitions  from Construction 
\ref{bakbis} be in charge.
  \begin{enumerate}[(i)]
    \item   
 If there are congruences
$\tilde{\alpha}, \tilde{\beta}, \tilde{\gamma} $ of $  \mathbf A_4$
such that the identity 
$\tilde{\alpha}( \tilde{\beta} \circ \tilde{\gamma}  )
\subseteq 
\tilde{\alpha} \tilde{ \gamma }
\circ 
\tilde{\alpha} \tilde{ \beta } 
 \circ {\stackrel{r}{\dots}} \circ    
 \tilde{\alpha} \tilde{ \beta } $
 fails in $\mathbf A_4$,
for some even $r$,  then there are $a,d \in A_4$ and 
congruences
$\alpha, \beta, \gamma $ of $\mathbf B = \mathbf B(a,d)$
such that the identity
 $\alpha( \beta \circ  \gamma  )
\subseteq 
\alpha \gamma    \circ  \alpha \beta 
\circ {\stackrel{r+4}{\dots}} \circ
\alpha \beta         $
 fails in $\mathbf B$.
    \item   
 More generally,
for every expression $\chi$,  if there are congruences
$\tilde{\alpha}, \tilde{\beta}, \tilde{\gamma} $ of $  \mathbf A_4$
such that the identity 
$\tilde{\alpha}( \tilde{\beta} \circ \tilde{\gamma}  )
\subseteq 
\chi (\tilde{\alpha}, \tilde{\beta},\tilde{\gamma}) $
 fails in $\mathbf A_4$,
 then there are $a,d \in A_4$ and 
congruences
$\alpha, \beta, \gamma $ of $\mathbf B = \mathbf B(a,d)$
such that the identity
 $\alpha( \beta \circ  \gamma  )
\subseteq 
\alpha (\gamma \circ \beta )   \circ \chi ( \alpha, \beta, \gamma ) 
\circ \alpha (\gamma \circ \beta )$
 fails in $\mathbf B$.
 \end{enumerate}
\end{theorem}  

\begin{proof}
(i) is immediate from the special case
$\chi ( \alpha, \beta, \gamma ) =
\alpha \gamma  \circ \alpha  \beta \circ  {\stackrel{r} {\dots}}
\circ \alpha \beta     $ of (ii), so let us prove (ii). 

Under the assumptions in (ii), 
there are 
$a,d \in A_4$ such that 
$ (a,d) \in \tilde{\alpha}( \tilde{\beta} \circ \tilde{\gamma}  )$
and $(a,d) \not\in 
\chi (\tilde{\alpha}, \tilde{\beta},\tilde{\gamma})$.
Choose such a pair 
$(a,d)$ and let  
$\mathbf B = \mathbf B(a,d)$.
Let
$\beta^*$ be the congruence on 
$\mathbf C_3$ whose blocks are 
$\{ 0\}$ and $\{ 1,2 \}$ and let 
$ \gamma^*$ be the congruence on 
$\mathbf C_3$ whose blocks are 
$\{ 0,1 \}$ and  $\{ 2 \}$.
Let $\beta$, $\gamma$ and $\alpha$ 
be the congruences on $\mathbf B$
induced, respectively, by 
$ \beta ^* \times  \gamma  ^* \times 1 \times \tilde{\beta} $,
$ \gamma  ^* \times  \beta  ^* \times 1 \times \tilde{\gamma} $ 
and  $ 1 \times 1 \times 0 \times \tilde{\alpha}$. 
Notice  a difference with respect to the proof of Theorem \ref{thmbak}:
here the first two components of $\beta$ 
are distinct,
and the same for $\gamma$.  
Consider the following elements of $B$:  
 \begin{equation*}
 c_0 =(2,0,1, a),   \qquad 
 c_1 =(1,1,0, b), \qquad    c_{2} = (0,2,1,d),
 \end{equation*}
of types, respectively, I, IV and III, and 
where $b$ is an element witnessing
$(a,d) \in
 \tilde{\beta} \circ  \tilde{\gamma} $, 
thus 
$(c_0, c_2) \in \alpha ( \beta \circ \gamma )$. 

If, by contradiction,
$(c_0,c_2)$ belongs to 
$\alpha (\gamma \circ \beta )   \circ \chi ( \alpha, \beta, \gamma ) 
\circ \alpha (\gamma \circ \beta )$, 
then there are elements $g,h \in B$ such that   
 $(c_0, g) \in \alpha( \gamma    \circ   \beta  )$,
$ (g,h) \in    \chi ( \alpha, \beta, \gamma ) $
and  $(h,c_2) \in \alpha( \gamma    \circ   \beta )$.
Thus $c_0 \mathrel \alpha   g $  and 
$ c_0 \mathrel { \gamma }   g_1 \mathrel { \beta}
 g $, for some
$g_1 $ in $ B$.   By $\gamma$-equivalence,
the first component of $g_1$ is $2$
and, by $\beta$-equivalence, the first component
of $g$ is either $1$ or  $2$.     
By $\alpha$-equivalence, the third component
of $g$ is  $1$ and since its first component is not $0$,
we get that $g$ has type I,
thus its fourth component is $a$.  
Symmetrically, the fourth component of 
$h$ is $d$.  From 
$ (g,h) \in    \chi ( \alpha, \beta, \gamma ) $ we get 
$(a,d) \in \chi (\tilde{\alpha}, \tilde{\beta},\tilde{\gamma})$,
 a contradiction. 
\end{proof}

\begin{proposition}  \labbel{thmbakbiscor}
Under the assumptions 
and the notation from Theorem   \ref{thmbakbis}(ii)
and its proof
 (in particular, Construction 
\ref{bakbis} is in charge), 
let $\mathbf B'$ be the subalgebra
of $\mathbf B$ 
 generated by 
$c_0$, $c_1$ and $c_2$.
Then the identity
 $\alpha( \beta \circ  \gamma  )
\subseteq 
\gamma \circ \alpha \beta   \circ \chi ( \alpha, \beta, \gamma ) 
\circ \alpha \gamma \circ  \beta $
 fails in $\mathbf B'$.
 \end{proposition}

\begin{proof} 
We first show that $c_0$
is the only element in $\mathbf B'$
having $2$ as the first component.
Any element of 
$B'$ has the form 
$t(c_0, c_1, c_2)$,
for some ternary term $t$.
Suppose by contradiction that in $B'$ there is some element 
$c^* \neq c_0$ such that the first component of 
$c^*$ is $2$. Thus
$c^* =t(c_0, c_1, c_2)$, for some term 
$t$. Choose  
$c^*$ and $t$ in such a way that $t$ has minimal complexity,
hence   
$t(x,y,z) = t_i (r_1(x,y,z) , r_2(x,y,z) , r_3(x,y,z) )$,
for some $i < n$ and ternary terms
$r_1$, $r_2$ and $r_3$ such that    
each of  $r_1(c_0, c_1, c_2)$,
$r_2(c_0, c_1, c_2)$ and 
$r_3(c_0, c_1, c_2)$ is either equal to $c_0$, or 
 has the first component different from $2$.

If $2 \leq i \leq n-2$,
then $t_i(x,y,z) =xz$
on the first three components.   
Since $2$ is the first component 
of $c^* =t(c_0, c_1, c_2)=
t_i(r_1(c_0, c_1, c_2),
r_2(c_0, c_1, c_2),
r_3(c_0, c_1, c_2))$,
then   $2$ is the first component 
of both $r_1(c_0, c_1, c_2) $ and 
$r_3(c_0, c_1, c_2)$.
By minimality of $t$,
then $c_0 =r_1(c_0, c_1, c_2) $ 
and $c_0  = r_3(c_0, c_1, c_2)$, hence 
$c^*$ and  $c_0$ coincide on the first three components.
But then $c^*$, being in $B$,     
must have type I,
so $c^*$ and  $c_0$ coincide also on the
fourth component, hence $c^*= c_0$.

If $i=1 $,
then $t_1(x,y,z) =x(y+z)$
on the first three components.   
Since $2$ is the first component 
of $c^* =t(c_0, c_1, c_2)=
t_1(r_1(c_0, c_1, c_2),
r_2(c_0, c_1, c_2),
r_3(c_0, c_1, c_2))$,
then   $2$ is the first component 
of $r_1(c_0, c_1, c_2) $ and
of at least one between 
$r_2(c_0, c_1, c_2)$ and
$r_3(c_0, c_1, c_2)$.
Again by minimality of $t$,
we  get that  
$c_0 =r_1(c_0, c_1, c_2) $ and
either  
$c_0 =r_2(c_0, c_1, c_2)$ or
$c_0 =r_3(c_0, c_1, c_2)$.
In both cases,
$c^*$ and  $c_0$ coincide on the first three components and,
 arguing as above, $c^*= c_0$.
The case $i=n-1$ is similar. 

In each case we  get
$c^*= c_0$,
 a contradiction, thus
$c_0$
is the only element in $\mathbf B'$
having $2$ as the first component.

Suppose that the assumptions in Condition
 \ref{thmbakbis}(ii) hold and 
let $\alpha$, $\beta$ and $\gamma$ be the congruences
induced on $\mathbf B'$ by the congruences with the same name in 
the proof of   
Theorem \ref{thmbakbis}.
If  $g $ is an element of $  B'$ such that 
$(c_0, g) \in \gamma \circ \alpha \beta $,
then   
$c_0 \mathrel {  \gamma } g_1 \mathrel { \alpha \beta }  g $,
for some $g_1 \in B'$, but
$\gamma$-equivalence implies that the first component of 
$g_1$ is $2$, thus $c_0=g_1$, by the above paragraphs.     
By $\beta$-equivalence, the first component of $g$
is not $0$ and, by $\alpha$-equivalence,    
the third component of $g$ is $1$, hence 
$g$ gas type I. Now the final argument in the    
proof of Theorem \ref{thmbakbis} applies. 
\end{proof}

\begin{definition} \labbel{bakerdef}
Baker \cite{baker}
introduced and studied 
 the variety generated by term-reducts of lattices
in which the only basic operation is $t_{\mathcal {B}}(x,y,z) = x(y+z)$.
We shall denote Baker's variety  by $\mathcal {B}$.

 Strictly speaking,
 Baker studied the varieties 
generated by reducts of a single lattice.
See Cornish \cite{Co}
for further details and for the relationships 
between Baker's variety and nearlattices. 
In any case, what we call $\mathcal {B}$ here 
can be obtained  by considering 
the variety generated by the reduct of the free
lattice over $ \omega$ generators, or just
the reduct of the free
lattice over $ 3$ generators, since the 
former is embeddable in the latter, by a well-known result by 
Whitman. See, e.~g., Freese, Je\v{z}ek,  Nation
\cite{FJN}.

The variety $\mathcal {B}^d$
   is defined like Baker's, but considering only reducts of distributive
lattices.
Notice that in the distributive case the correspondence
with nearlattices is exact, since 
in distributive lattices 
$ x(y+z)=xy+xz$
and (dual) nearlattices
admit an equivalent definition as the variety 
generated by reducts of lattices 
in which only the operation 
given by $xy+xz$ is considered. 
See, e.~g., Chajda, Hala\v{s},  K\"{u}hr \cite{CHK}
for further informations about nearlattices.
In particular, since we shall always deal with distributive 
lattices, our results about 
 $\mathcal {B}^d$ apply to the variety 
of distributive nearlattices,  provided we
 consider nearlattices as ternary algebras.
Of course, the above remark applies to most results from
\cite{B}, as well.
 \end{definition}   

Notice that, for each $n \geq 3$,   Baker's variety is term-equivalent to
the variety generated by the algebras $\mathbf L^r$
from Construction \ref{bak}.  
Indeed, if  $t_i(x,y,z)=xz$ in \ref{bak}, then $t_i$  
can be expressed as $t_i(x,y,z) = t_{\mathcal {B}}(x,z,z) = x(z+z)$.   
However, the exact types of algebras will be highly
relevant in our arguments.

Congruence identities valid in Baker's variety have been intensively studied
in \cite{B}. Complementing the
results from \cite{B}, we shall now see that the arguments
from the proof of Theorems   \ref{thmbak}
and \ref{thmbakbis} show  the failure of 
still another congruence identity in $\mathcal {B}$.

\begin{proposition} \labbel{baker+}
Neither $\mathcal {B}$, nor $\mathcal {B}^d$,
nor the variety of nearlattices $\mathcal {NL}$ 
are $5$-reversed-modular.

Moreover, the  congruence identities 
$\alpha ( \beta \circ   \gamma  )
\subseteq  
\alpha (   \gamma \circ \beta   ) \circ 
\alpha (   \gamma \circ \beta  )$ 
and $\alpha ( \beta \circ   \gamma  )
\subseteq  
   \gamma \circ \alpha  \beta    \circ 
\alpha   \gamma \circ \beta  $ 
fail
 in $\mathcal {B}$,  $\mathcal {B}^d$
and $\mathcal {NL}$.

In particular, neither $\mathcal {B}$, nor $\mathcal {B}^d$
nor $\mathcal {NL}$ are
$4$-alvin.  
 \end{proposition}

  \begin{proof}
Let $n=3$. Consider only the first three components
in Constructions \ref{c},  \ref{bak}, \ref{bakbis}
and in the  proofs of Theorems \ref{thmB}, \ref{thmbak}
and  \ref{thmbakbis}.
Or, more formally, rather than reformulating everything,
take $\mathbf A_4$
and  $\mathbf D$ as  $1$-element algebras everywhere. 
By the comment shortly after Definition \ref{bakerdef},
all the mentioned constructions furnish algebras
which are 
 term-equivalent  
to algebras in $\mathcal {B}^d$, hence satisfying the same
congruence identities.

 Construct an algebra $\mathbf B$ and  elements 
$c_0, \dots, c_3$ 
as in the proof of 
Theorem \ref{thmbak}(ii),
either disregarding the fourth component,
or taking some fixed element
(the only element in $A_4$) at the fourth  place.
With the corresponding definitions of
$\alpha$, $\beta$ and $\gamma$
from the proof of 
\ref{thmbak}(ii),
in the present situation
we need no further assumption to get
$(c_0,c_3) \in \alpha ( \beta \circ \alpha \gamma \circ \beta )$.
Or, put in another way, since here
we are assuming $a=d$, we automatically 
get $ (a,d) \in \tilde{\alpha}
( \tilde{\beta} \circ \tilde{ \alpha }\tilde{\gamma} \circ \tilde{\beta} )$. 
We shall show that $\mathbf B$ is not
 $5$-reversed-modular.
If, by contradiction,   $\mathbf B$ is 
 $5$-reversed-modular, then
$(c_0,c_3) \in \alpha    \gamma \circ \alpha \beta  \circ \alpha \gamma \circ 
\alpha \beta  \circ \alpha  \gamma $,
a fortiori, 
$(c_0,c_3) \in \alpha (   \gamma \circ \beta  \circ \gamma ) \circ 
\alpha (   \gamma \circ \beta  \circ \gamma )$.
Thus there is some element  $g \in B$
such that 
 $(c_0,g) \in \alpha (   \gamma \circ \beta  \circ \gamma )$
and  $(g, c_3) \in \alpha (   \gamma \circ \beta  \circ \gamma )$.
The  proof of Theorem \ref{thmbak}
shows that the first component of
$g$ is not $0$ and that $g$ has type I.
The symmetric argument shows that $g$
has type III, a contradiction, since the first component of
any element of type III is $0$. 
Notice that here $g$ plays at the same time the role
of both $g$ and  $h$ from  the proof of  Theorem \ref{thmbak}.
The fact that 
  $\mathbf B^d$ is 
not $5$-reversed-modular
 is also a consequence of the
 last equation in \cite[Proposition 2.3]{B},
taking $n=3$ there.

The proof that the  identities
in the second statement
 fail
is obtained by a similar variation
on the proofs of
Theorem \ref{thmbakbis}(ii)
and of Proposition \ref{thmbakbiscor}. 
Another proof that 
$\alpha ( \beta \circ   \gamma  )
\subseteq  
\alpha (   \gamma \circ \beta   ) \circ 
\alpha (   \gamma \circ \beta  )$ 
fails in $\mathbf B^d$ follows from the case $n=2$ in 
the penultimate identity in \cite[Proposition 2.3]{B}. 

The final statement is then  immediate from the fact that
$\alpha \gamma \circ \alpha \beta \subseteq \alpha ( \gamma \circ \beta )$. 
 \end{proof}

\subsection*{J{\'o}nsson distributivity in the middle} 
Operations of Boolean algebras will be denoted
by juxtaposition, $+$  and $'$.
Let $\mathbf 2 = \{ 0,1 \} $
be the $2$-elements Boolean algebra with 
largest element $1$ and smallest element $0$.
Let $\mathbf 4 = \{ 0,1, 1', 2 \} $
be the $4$-elements Boolean algebra with 
largest element $2$ and smallest element $0$.
We have chosen such a labeling to maintain 
the analogy with the preceding subsections; thus, for example,
$\mathbf C_3 = \{ 0,1,2 \} $ is a sublattice of the
lattice-reduct of  $\mathbf 4$.
 
\begin{construction} \labbel{ba}
Fix some  natural number $n \geq 3$. 

For a Boolean algebra $\mathbf A$,  
let     $\mathbf A^r$ 
denote the  term-reduct
 with operations
\begin{multline*} 
t_1(x,y,z) = x(y'+z), \qquad
t_2(x,y,z) = xz, \qquad
t_3(x,y,z) = xz, \qquad  \dots, \qquad 
\\
\dots, \qquad t_{n-2}(x,y,z) = xz,
\qquad t_{n-1}(x,y,z) = z(y'+x).
\end{multline*}   

Let $\mathbf A_1=\mathbf A_2 = \mathbf 4^r$ 
and
$\mathbf A_3= \mathbf 2^r$.

Suppose that $\mathbf D$
is an algebra with ternary  operations  
$s_0, \dots, s_{n-2}$,
 relabel the operations
as $t_1=s_0, \dots, t _{n-1} = s_{n-2} $ and let
$\mathbf A_4$ be the resulting algebra.
Suppose that $\mathbf A_4$ satisfies the assumptions in
Construction \ref{c}. 
As in the preceding subsection, the algebras $\mathbf A_1$,
$\mathbf A_2$,
$\mathbf A_3$ and 
$\mathbf A_4$ satisfy the assumptions in
Construction \ref{c},
hence, by Theorem \ref{thmB},
for every choice of elements $a,d \in A_4$, 
 we have an algebra $\mathbf B = \mathbf B(a,d)$ 
 constructed as in \ref{c}(B). 
 \end{construction}   

Recall that an expression
is a term in the language
$ \{ \circ, \cap \} $. 

\begin{theorem} \labbel{thmba}
Let  the assumptions and the definitions in Construction 
\ref{ba} be in charge.
  \begin{enumerate}[(i)] 
   \item  
If $n$  is even and $s_0, \dots, s_{n-2}$ are J{\'o}nsson 
operations for $\mathbf D$, then, for every choice of  $a,d \in A_4$,
the algebra $\mathbf B$ is $n$-alvin.  

\item
 If there are congruences
$\tilde{\alpha}, \tilde{\beta}, \tilde{\gamma} $ of $  \mathbf A_4$
such that the identity 
$\tilde{\alpha}( \tilde{\beta} \circ \tilde{\alpha} \tilde{\gamma} \circ \tilde{\beta} )
\subseteq 
\tilde{\alpha} \tilde{\gamma} 
 \circ 
 \tilde{\alpha} \tilde{\beta} 
\circ {\stackrel{r}{\dots}}  $
 fails in $\mathbf A_4$,
for some $r$,  then there are $a,d \in A_4$ and 
congruences
$\alpha, \beta, \gamma $ of $\mathbf B = \mathbf B(a,d)$
such that the identity 
$\alpha( \beta \circ \alpha \gamma \circ \beta )
\subseteq
 \alpha \beta   \circ  \alpha \gamma      
\circ {\stackrel{r+2}{\dots}} $ 
fails in $\mathbf B$.  
\end{enumerate}
Moreover, for every expression $\chi$
and every choice of $\delta= \beta $ or $ \delta = \gamma $
and of $\varepsilon = \beta $ or $\varepsilon = \gamma $,
the following hold.
  \begin{enumerate}[(i)]      
\item[(iii)]
 If there are congruences
$\tilde{\alpha}, \tilde{\beta}, \tilde{\gamma} $ of $  \mathbf A_4$
such that the identity 
$\tilde{\alpha}( \tilde{\beta} \circ \tilde{\alpha} \tilde{\gamma} \circ \tilde{\beta} )
\subseteq 
\chi (\tilde{\alpha}, \tilde{\beta},  \tilde{\gamma}) $
 fails in $\mathbf A_4$,
  then there are $a,d \in A_4$ and 
congruences
$\alpha, \beta, \gamma $ of $\mathbf B = \mathbf B(a,d)$
such that the identity 
$\alpha( \beta \circ \alpha \gamma \circ \beta )
\subseteq
\alpha \delta  \circ 
\chi ( \alpha , \beta , \gamma)  
\circ \alpha \varepsilon  $ 
fails in $\mathbf B$.  

\item[(iv)]
 If there are congruences
$\tilde{\alpha}, \tilde{\beta}, \tilde{\gamma} $ of $  \mathbf A_4$
such that the identity 
$\tilde{\alpha}( \tilde{\beta} \circ \tilde{\gamma}  )
\subseteq 
\chi (\tilde{\alpha}, \tilde{\beta},  \tilde{\gamma}) $
 fails in $\mathbf A_4$,
  then there are $a,d \in A_4$ and 
congruences
$\alpha, \beta, \gamma $ of $\mathbf B = \mathbf B(a,d)$
such that the identity 
$\alpha( \beta \circ  \gamma )
\subseteq
\alpha \delta  \circ 
\chi ( \alpha , \beta , \gamma)  
\circ \alpha \varepsilon  $ 
fails in $\mathbf B$.  
 \end{enumerate} 
 \end{theorem}  

\begin{proof} 
Clause (i) is proved 
as in  Lemma \ref{lembak}(i).
The algebras  $\mathbf 4^r$ and 
$\mathbf 2^r$ are clearly  $n$-alvin,
since $n$ is even.
$ \mathbf A_4$ is  $n$-alvin, too,
since the indices are shifted by $1$.
Notice that if $s_0, \dots, s_{n-2}$ are J{\'o}nsson 
operations for $\mathbf D$, then
Conditions \eqref{t1}, \eqref{tn} and  \eqref{b}
in Construction \ref{c} are satisfied by $\mathbf A_4$. 

As usual by now, (ii) is a special cases of (iii).
Take $\chi ( \alpha , \beta , \gamma)   =
\alpha \gamma 
 \circ 
\alpha \beta  
\circ {\stackrel{r}{\dots}}  $,
$\delta = \beta $ and $\varepsilon = \gamma $
if $r $ is even, $\varepsilon = \beta  $
if $r $ is odd.

In order to prove (iii), suppose
 that $\tilde{\alpha}$, $\tilde{\beta}$ and $\tilde{\gamma}$ 
are congruences on $\mathbf A_4$  
and $a,d $ are elements of $  A_4$ such that 
$(a,d) \in  
\tilde{\alpha}( \tilde{\beta} \circ \tilde{\alpha} \tilde{\gamma} \circ \tilde{\beta} )$
and $ (a,d) \not \in 
\chi (\tilde{\alpha}, \tilde{\beta},  \tilde{\gamma}) $.
Let $\beta^*$ be the congruence 
on $\mathbf 4$ whose blocks are
$\{ 1,2 \}$  and $\{ 0, 1' \}$. 
Let $ \gamma ^*$ be the congruence
on $\mathbf 4$ whose blocks are
$\{ 0,1 \}$  and $\{  1', 2 \}$. 
Since $\beta$ and $\gamma$ 
are congruences on the Boolean algebra $\mathbf 4$,
they are also congruences on the reduct 
$\mathbf 4^r$.
 Let $\beta$, $\gamma$ and $\alpha$ 
be the 
 congruences on $\mathbf B = \mathbf B(a,d)$ induced, respectively, by
$ \beta ^* \times  \beta ^* \times 1 \times \tilde{\beta} $, 
$ \gamma  ^* \times  \gamma  ^* \times 1 \times \tilde{\gamma} $ 
and  $ 1 \times 1 \times 0 \times \tilde{\alpha}$. 
Since 
$(a,d) \in  \tilde{\alpha}( \tilde{\beta} \circ \tilde{\alpha} \tilde{\gamma} \circ \tilde{\beta} )$,
then $a \mathrel {\tilde{\alpha}}  d $, and  there are $b,c \in A_4$ 
such that $a \mathrel { \tilde{\beta}} b \mathrel { \tilde{\alpha} \tilde{\gamma}  } c 
\mathrel { \tilde{\beta} } d $.
Consider the following elements of $B$:  
 \begin{equation*}
 c_0 =(2,0,1, a),   \quad 
 c_1 =(1,0,0, b), \quad 
   c_{2} = (0,1,0,c),
 \quad c_{3} = (0,2,1, d).
\end{equation*} 
As in the proof of Theorem \ref{thmbak}, $c_0$ has type  I,
$c_1$ and $c_2$  have type  IV
and $c_3$ has type  III,
thus they belong to $B$.
Moreover, 
$c_0 \mathrel \alpha  c_3$
and  
$c_0 \mathrel \beta c_1 \mathrel { \alpha \gamma  } c_2
\mathrel { \beta }   c_3$, hence
$(c_0,c_3) \in  \alpha( \beta \circ \alpha \gamma \circ \beta)$.

Whatever the choice of $\delta$ and $\varepsilon$, 
assume by contradiction that
$\alpha( \beta \circ \alpha \gamma \circ \beta )
\subseteq
\alpha \delta  \circ  
\chi ( \alpha , \beta , \gamma)  
\circ \alpha \varepsilon   $, thus
 $(c_0,c_3) \in
\alpha \delta  \circ 
\chi ( \alpha , \beta , \gamma)  
\circ \alpha \varepsilon    $, hence
$c_0 \mathrel { \alpha \delta } g $,
$(g,h ) \in \chi ( \alpha , \beta , \gamma)  $
and $h \mathrel {  \alpha \varepsilon } c_3 $,
for certain elements $g, h$ of $\mathbf B$.    

Since
the first component of $c_0$
is $2$ and
 $c_0 \mathrel { \delta } g $,  then, whatever  the choice
of $\delta$, be it $  \beta $
or $ \gamma $, the first component 
of $g$ is not $0$. By $\alpha$-connection
of $c_0$ and $g$,
the third component of $g$ is $1$, hence 
$g$ has type I, so the fourth component of $g$ is $a$.
Symmetrically, the fourth component of $h$ is $d$.
Since 
$(g,h ) \in \chi ( \alpha , \beta , \gamma)  $,
it follows that  
$(a,d) $ is in $  \chi (\tilde{\alpha}, \tilde{\beta},  \tilde{\gamma}) $,
a contradiction.

 In order to prove (iv),
we use an argument resembling the proof of   Theorem \ref{thmbakbis}.
Suppose that $ (a,d) \in \tilde{\alpha}( \tilde{\beta} \circ \tilde{\gamma}  )$
and $(a,d) \not \in \chi (\tilde{\alpha}, \tilde{\beta},  \tilde{\gamma}) $.
As above, let $\beta^*$ be the congruence
on $\mathbf 4^r$ whose blocks are
$\{ 1,2 \}$  and $\{ 0, 1' \}$
and let $ \gamma ^*$ be the congruence
on $\mathbf 4^r$ whose blocks are
$\{ 0,1 \}$  and $\{  1', 2 \}$. 
 In this case, let $\beta$, $\gamma$ and $\alpha$ 
be the 
 congruences on $\mathbf B = \mathbf B(a,d)$ induced, respectively, by
$ \beta ^* \times  \gamma  ^* \times 1 \times \tilde{\beta} $, 
$ \gamma  ^* \times  \beta   ^* \times 1 \times \tilde{\gamma} $ 
and  $ 1 \times 1 \times 0 \times \tilde{\alpha}$. 
If $b$ is such that 
$ a \mathrel { \tilde{\beta} } b 
\mathrel {  \tilde{\gamma} } d$,
consider the following elements of $B$:  
 \begin{equation*}
 c_0 =(2,0,1, a),   \qquad 
 c_1 =(1,1,0, b), \qquad    c_{2} = (0,2,1,d),
 \end{equation*}
thus $(c_0, c_2) \in \alpha ( \beta \circ \gamma )$. 
If
$ (c_0, c_2) \in 
\alpha \delta  \circ 
\chi ( \alpha , \beta , \gamma)  
\circ \alpha \varepsilon   $, this relation is witnessed by 
appropriate elements 
$g$ and  $h$ and,  arguing as in (iii), 
the fourth components of $g$ and  $h$ are, respectively 
$a$ and  $d$.
 But then 
$(a,d) \in \chi (\tilde{\alpha}, \tilde{\beta},  \tilde{\gamma}) $,
a contradiction.
\end{proof}

\begin{remark} \labbel{ifina}  
In the notation from the proof
of Theorem \ref{thmba},
both in case (iii) and in case (iv), 
if we let 
$e_1=(1,0,1,a)$,
$e_1^* = (1',0,1,a)$,  
we see that 
$\{ c_0, e_1\}$
is an $ \alpha \beta$-block in $B$
and 
$\{ c_0, e_1^*\}$
is an $ \alpha \gamma$-block in $B$.
This might be useful in different situations.
\end{remark}

\section{Day's Theorem is optimal for $n$ even} \labbel{opti}

\begin{theorem} \labbel{buh}
Suppose that  $n \geq 2$ and $n$ is even.

(i)  There is a locally finite
$n$-distributive variety which is not $2n{-}1$-reversed-modular,
in particular, not $2n{-}2$-modular.

(ii) There is a locally finite
$n$-alvin variety which is not $2n{-}3$-modular.
 \end{theorem}

 \begin{proof} 
If some variety $\mathcal {V}$  is not $2n{-}1$-reversed-modular,
then $\mathcal {V}$ is not $2n{-}2$-\brfrt modular, by Proposition \ref{modeq}.

The proof of the hard parts
of the theorem goes by simultaneous induction on $n$. 
We first consider the base cases.

The variety of lattices is 
$2$-distributive and not $3$-reversed-modular.
Indeed, under the equivalence given by  Remark \ref{conidmod},
 $3$-reversed-modularity reads
$\alpha( \beta \circ \alpha \gamma  \circ \beta )
\subseteq \alpha \gamma \circ \alpha \beta \circ \alpha \gamma $ and this identity
implies $3$-permutability: just   
take $\alpha=1$, the largest congruence. The variety of lattices is not 
$3$-permutable, hence it is not $3$-reversed-modular.
The above arguments apply to the variety of distributive lattices, as well,
and the variety of distributive lattices  is locally finite.

The variety of Boolean algebras is
locally finite, $2$-alvin and not $1$-modular.
Notice that a $1$-modular variety is a trivial variety.
Thus the basis of the induction is true.  
Notice that, in place of lattices and of Boolean algebras 
we can just consider the varieties of their term-reducts
with just a majority or a Pixley operation.

Suppose that $n \geq 4$ and that
the theorem is true for $n-2$.
By the inductive hypothesis
and   Remark \ref{conidmod}, 
there exist an $n {-} 2$-alvin variety $\mathcal {V}$ 
and an algebra $\mathbf D \in \mathcal {V}$ 
with congruences $\alpha$, $\beta$ and $\gamma$ 
such that the congruence  identity 
$\alpha ( \beta \circ \alpha \gamma  \circ \beta )
\subseteq
\alpha \beta \circ \alpha \gamma 
\circ {\stackrel{2n-7}{\dots}} \circ   
\alpha \beta $ fails in $\mathbf D$. 
 Since $\mathbf D $ belongs to an $n{-}2$-alvin variety,
$\mathbf D $ has $n-2$ alvin terms. 

It is no loss of generality 
to assume that these terms are actually operations of
$\mathbf D $. We can also assume that $\mathbf D $
has no other operation, since $\alpha$, $\beta$ and $\gamma$ 
remain congruences on the reduct; moreover,
intersection and composition do not depend on the algebraic structure
of the algebra under consideration.
Thus the  identity 
$\alpha ( \beta \circ \alpha \gamma  \circ \beta )
\subseteq
\alpha \beta \circ \alpha \gamma 
\circ {\stackrel{2n-7}{\dots}} \circ   
\alpha \beta $ fails in $\mathbf D$, even if we consider
$\mathbf D$ as an algebra with only the alvin operations.
Otherwise, as we mentioned, the basis of the theorem
can be proved for algebras having only alvin or J{\'o}nsson 
operations and it is easy to check that in the induction
we are going to perform we construct algebras with 
 such operations only. This fact will appear
evident in the course of the proof
of Theorem \ref{sumup} below. Whatever the argument, we can suppose that
$\mathbf D$ has only alvin operations.

Apply Construction \ref{bak} to the algebra $\mathbf D$. 
By Lemma \ref{lembak}(i) and
Theorem \ref{thmbak}(i) 
with $r=2n-7$,
there is an $n$-distributive algebra $\mathbf B$    
(which henceforth generates an $n$-distributive variety)
in which $2n{-}1$-reversed-modularity fails. 

In the parallel situation, again by the inductive hypothesis
and  Remark \ref{conidmod}, 
there exist an $n {-} 2$-distributive variety $\mathcal {V}$ 
and an algebra $\mathbf D \in \mathcal {V}$ 
such that the congruence  identity 
$\alpha ( \beta \circ \alpha \gamma  \circ \beta )
\subseteq
\alpha \gamma  \circ \alpha \beta 
\circ {\stackrel{2n-5}{\dots}} \circ   
\alpha \gamma  $ fails in $\mathbf D$. 
 Arguing as above, we can suppose that
$\mathbf D$ has only the J{\'o}nsson  operations.
Apply Construction \ref{ba} to the algebra $\mathbf D$. 
By Theorem \ref{thmba}(i)(ii) with $r=2n-5$,
there is an $n$-alvin algebra $\mathbf B$    
(which henceforth generates an $n$-alvin variety)
in which $2n{-}3$-modularity fails. 

The induction step is thus complete. In order to conclude
the proof of  the theorem it is enough to show that the above varieties
can be taken to be locally finite. First notice that we have used distributive
lattices in all of our constructions, and the variety of distributive lattices is
locally finite. The variety of Boolean algebras is locally finite, as well.
By induction, if $\mathbf D$ belongs to some locally finite variety,
then $\mathbf A_4$, too, belongs to some locally finite variety, 
since  $\mathbf A_4$ is term-equivalent to $\mathbf D$.
By the above remarks, at each induction step, $\mathbf B$ can 
be taken to belong to the join of two locally finite varieties,
hence to a locally finite variety.
\end{proof}  

 A somewhat  simpler description of  varieties 
furnishing a  proof of Theorem \ref{buh}
shall be presented in Section \ref{moreex}.
However, as remarked in the introduction, 
the proof seems to necessarily rely on 
the methods in the present and in the former section.

The next lemma applies not
only to $n$-alvin varieties, but also to varieties 
which satisfy the weaker form of the $n$-alvin condition in   
which the  identities 
$x=t_1(x,y,x)$ and  $t_{n-1}(x,y,x) =x$
are not  assumed.
We shall state a reformulation of  this observation 
 in Proposition \ref{gummopt}(i) below. 

\begin{lemma} \labbel{dayoptlem}
  \begin{enumerate}[(a)]    
\item  
 If $n \geq 4$ and $n$ is even, then 
every $n$-alvin variety is $2n{-}3$-reversed-modular.

Actually, the result applies to a condition weaker
than $n$-alvin: it is not necessary to assume 
the ``outer'' equations 
$x=t_1(x,y,x)$ and  $t_{n-1}(x,y,x) =x$.

\item
If $n \geq 2$, then 
every $n$-alvin variety is $2n{-}2$-modular.
In particular, if $n$ is odd, then every
 $n$-distributive variety is $2n{-}2$-modular.
 \end{enumerate}  
\end{lemma} 

\begin{proof}
Part (a) is a special case
of \cite[Proposition 6.4]{ntcm} with $n-2$ in place of $n$ there.
Corollary  \ref{propmixmod}(ii)(c) below provides 
a more general result. See Remark \ref{rcm}. Still another proof, along the lines
of Day's argument, is obtained by performing the trick  
in the proof of   \cite[Theorem 1 (3) $\rightarrow $   (1)]{LTT} ``at both ends''.  
We report the details below for the reader's
convenience.

Given alvin  terms $t_0, \dots, t_{n}$, we obtain the following 
terms $u_0, \dots, u_{2n-3}$ satisfying 
the reversed form of the conditions in
Definition \ref{daydef}. The terms $u_0, \dots, u_{2n-3}$ below
 are considered as $4$-ary terms depending on the variables
$x,y,z,w$ in that order. The term $u_0$ is constantly
$x$ and the term $u_{2n-3}$ is constantly
$w$. The remaining terms are 
defined in  the following table, where 
  we omit commas for lack of space. 
\begin{align*}   
  u_1&{\hspace {1 pt}=\hspace {1 pt}}t_1(x y z)   &  u_2&{\hspace {1 pt}=\hspace {1 pt}}t_2(x y w) &   u_3&{\hspace {1 pt}=\hspace {1 pt}}t_2(x z w) &     
   u_4&{\hspace {1 pt}=\hspace {1 pt}}t_3(x z w)
 \\   
 u_5&{\hspace {1 pt}=\hspace {1 pt}}t_3(x y w) &
   u_6&{\hspace {1 pt}=\hspace {1 pt}}t_4(x y w)  & 
   u_7&{\hspace {1 pt}=\hspace {1 pt}}t_4(x z w) & & \dots 
\\
  u_{4i{+}1}&{\hspace {1 pt}=\hspace {1 pt}}t_{2i{+}1}(x y w)   
 & u_{4i{{+}}2}&{\hspace {1 pt}=\hspace {1 pt}}t_{2i{+}2}(x y w)    &
  u_{4i{{+}}3} &{\hspace {1 pt}=\hspace {1 pt}}t_{2i{+}2}(x z w)   &
  u_{4i{+}4}& {\hspace {1 pt}=\hspace {1 pt}}t_{2i{+}3}(x z w) 
\\
& \dots 
&   u_{2n{\hspace {-1 pt}-\hspace {-1 pt}}10}\hspace {-1 pt} 
&{\hspace {1 pt}=\hspace {1 pt}}  t_{n{-}4}(x y w)
&   u_{2n{-}9} &{\hspace {1 pt}=\hspace {1 pt}}  t_{n{-}4}(x z w)  &
  u_{2n{-}8} &{\hspace {1 pt}=\hspace {1 pt}}  t_{n{-}3}(x z w)  
\\
 u_{2n{-}7} &{\hspace {1 pt}=\hspace {1 pt}}  t_{n{-}3}(x y w) &
   u_{2n{-}6} &{\hspace {1 pt}=\hspace {1 pt}}  t_{n{-}2}(x y w)
&   u_{2n{-}5} &{\hspace {1 pt}=\hspace {1 pt}}  t_{n{-}2}(x z w)  &
  u_{2n{-}4} &{\hspace {1 pt}=\hspace {1 pt}}  t_{n{-}1}(y z w)  
 \end{align*}

Notice the different arguments of $t_{1}$
and of
$t_{n-1}$
with respect to the other
terms in the corresponding columns.
If $n=4$, we consider only
the first line, taking 
$ u_{4} =  t_{3}(y z w)$.

Notice that
the indices in the last two lines follow the same pattern
of the preceding lines, taking, respectively,  $i=  \frac{n-6}{2} $
and $i=  \frac{n-4}{2} $.
We can do this since $n$ is assumed to be even. 

The fact that  $u_0, \dots, u_{2n}$ satisfy the conditions in 
Definition \ref{daydef} with even and odd exchanged
is easy and is proved as in \cite[p.\ 172--173]{D}.
The only different computations are 
$u_{0} (x,y,y,w)=x=
 t_{0}(x,y,y)=  t_{1}(x,y,y)=
u_{1} (x,y,y,w) $ and 
$u_{1} (x,x,w,w) =
t_1(x,x,w)= t_2(x,x,w)= u_{2} (x,x,w,w)$.
Notice that, in order to perform the above
computations, it is fundamental to deal 
with the alvin and the reversed Day conditions!  
 Symmetrically, at the other end,
$ u_{2n-5} (x,x,w,w)=  t_{n-2}(x,w,w) =   t_{n-1}(x,w,w) 
=u_{2n-4} (x,x,w,w)$ 
and 
$u_{2n-4} (x,y,y,w)=  t_{n-1}(y,y, \allowbreak  w) = t_{n}(y,y,w)  = w
= u_{2n-3} (x,y,y,w)$. 
Notice  that in this case it is fundamental to have $n$ even!
Finally, notice that we have not used the equations
$x=t_1(x,y,x)$ and  $t_{n-1}(x,y,x) =x$
in the above computations.

Part (b) is proved in a similar way, with no 
``special trick'' at the final end.  
We get that every $n$-alvin variety is $2n{-}2$-reversed modular,
but this is equivalent to  $2n{-}2$-modular, by
Proposition \ref{modeq}. The last statement follows
from the fact that if $n$ is odd, then $n$-alvin and 
 $n$-distributive are equivalent conditions, by 
Remark \ref{conidcomm}(a). 
 \end{proof}

\begin{corollary} \labbel{dayopt} 
Suppose that $n \geq 2$ and $n$ is even.
  \begin{enumerate}[(i)]   
 \item  
 Every $n$-distributive variety is $2n{-}1$-modular.
\item
Every $n$-alvin variety is  $2n{-}2$-modular.
\item
Every 
$2$-alvin variety is $2$-reversed-modular.
If $n \geq 4$, then 
every $n$-alvin variety is $2n{-}3$-reversed-modular.
\item
All the above results are sharp: 
for every even $n \geq 2$ there are an
$n$-distributive variety which is not $2n{-}2$-modular
and an 
$n$-alvin variety which is not $2n{-}3$-modular,
in particular, by Proposition \ref{modeq}, not $2n{-}4$-reversed-modular.  
The variety of Boolean algebras is $2$-alvin 
and not  $1$-reversed-modular.  
 \end{enumerate} 
\end{corollary}

 \begin{proof}
As already mentioned, (i) is due to Day \cite{D}
and the assumption that $n$ is even is not necessary
in (i).  The proof is  slightly simpler than the proof of
Lemma \ref{dayoptlem}. 
This time the chain of terms is given by
\begin{align*} 
u_1 &=
 t_1(x, y, w),  & u_2 &= t_1(x, z, w),  
& u_3 &= t_2(x, z, w),  
& u_4 & = t_2(x, y, w), 
\\ 
 u_5 & = t_3(x, y, w), 
& u_6 & = t_3(x, z, w), 
&\dots
\end{align*}    
and there are no special variations on the outer edges.
The  proof of a fact  more general than (i)
 using different methods
 shall be presented 
in Corollary  \ref{propmixmod}(i).

(ii) is a special case of Lemma \ref{dayoptlem}(b). 

 (iii) As mentioned, $2$-alvin is arithmeticity. In particular, by 
distributivity (hence modularity) and 
permutability,
we get both $2$-modularity and $2$-reversed-modularity.
If $n \geq 4$, we get  that every 
$n$-alvin variety $\mathcal {V}$ is  $2n{-}3$-reversed-modular
from Lemma \ref{dayoptlem}(a), hence $\mathcal V$ is
$2n{-}2$-modular,
  by Proposition \ref{modeq}, or use directly \ref{dayoptlem}(b).

The nontrivial parts in (iv) are given by Theorem \ref{buh}.
 \end{proof}  

\begin{remark} \labbel{daystrong}
Day's proof of 
Theorem \ref{day} 
(recalled above in Corollary \ref{dayopt}(i))
actually provides terms satisfying the equations
\begin{equation}\labbel{3dist}    
x=u_k(x,y,z,x), \qquad \text{for all indices $k$.}
 \end{equation} 
The above equations are stronger than 
equations \eqref{d0}, 
and correspondingly Day's proof actually
shows that
if $n > 0$, then every
$n$-distributive variety
satisfies the congruence identity
\begin{equation}\labbel{str}     
\alpha (\beta \circ  \gamma \circ \beta )
\subseteq \alpha \beta \circ \alpha \gamma 
\circ {\stackrel{2n-1}{\dots}} \circ \alpha \beta   .      
  \end{equation}
In the same way it can be proved that
$n$-alvin varieties, as well, satisfy the identity \eqref{str}.  
Compare also Proposition \ref{propmixi}
below (take $i=2$, $S_0=S_2= \beta $ and  $S_1= \gamma $ there). 
Further elaborations and comments about
these generalizations can be found in 
\cite{jds, ntcm}. 

Notice that, on the other hand,
the terms $u_1$ and $u_{2n-4}$
constructed in the proof of Lemma \ref{dayoptlem}(a)
do not necessarily satisfy the equations in \eqref{3dist},
though the terms    $u_2, u_3, \dots, u_{2n-5}$
do satisfy \eqref{3dist}.
From the point of view
 of congruence identities the above observation shows 
that if $n \geq 4$ and $n$ is even, then
every $n$-alvin variety satisfies the congruence identity    
\begin{equation}\labbel{3mal}     
\alpha ( \beta \circ   \gamma \circ \beta  ) 
\subseteq 
 \alpha (\gamma  \circ   \beta)
 \circ  
(\alpha \gamma     \circ    \alpha \beta
   \circ  {\stackrel{2n-7}{\dots}}
   \circ    \alpha \gamma  )
\circ \alpha ( \beta    \circ    \gamma  ), 
  \end{equation}
and, as in Lemma \ref{dayoptlem}(a),
the ``outer''  equations  
$x=t_1(x,y,x)$ and  $t_{n-1}(x,y,x) =x$
are not necessary for the proof.
If we take $\alpha \gamma $ in 
place of $\gamma$ in identity \eqref{3mal},
we have  $\alpha ( \alpha \gamma  \circ   \beta) =
\alpha \gamma \circ \alpha \beta $ and
similarly on the other end,
hence we get 
$\alpha \gamma     \circ    \alpha \beta
   \circ  {\stackrel{2n-3}{\dots}}$ 
on the right-hand side,
thus
 \eqref{3mal} is stronger than Lemma \ref{dayoptlem}(a),
via Remark \ref{conidmod}.

However, we do not know whether there is a common improvement
of \eqref{str} and \eqref{3mal} holding in every 
$n$-alvin variety.   
See, e.~g., Problem \ref{prob}(d) below. 
 \end{remark}

The arguments in the proof of Theorem \ref{buh},
together with Theorems \ref{thmbakbis}(ii), \ref{thmba}
and Proposition \ref{thmbakbiscor},
allow us to present other congruence identities which
are not always  satisfied in 
$n$-distributive and $n$-alvin varieties.
Let $  {\stackrel{ \ell }{\dots}} \circ \gamma  \circ  \beta$ 
denote
$  \gamma  \circ  \beta \circ  {\stackrel{ \ell }{\dots}}  \circ  \beta$ 
if $\ell$ is even and 
$  \beta  \circ  \gamma  \circ  {\stackrel{ \ell }{\dots}}  \circ  \beta$ 
if $\ell$ is odd.
If $R$ is a binary relation and $k$ is a natural number, 
let $R^k = R \circ R \circ {\stackrel{k}{\dots}} \circ   R$. 

\begin{theorem} \labbel{buhbis}
Suppose that  $n \geq 2$, $n$ is even
and let $\ell= \frac{n}{2} $. 

  \begin{enumerate}[(i)]
    \item    
 There is a locally finite
$n$-distributive variety in which the following congruence identities 
fail:
\begin{align} \labbel{agob}    
\alpha ( \beta \circ \gamma ) &\subseteq ( \alpha ( \gamma \circ \beta ))^\ell,
 \\
\labbel{goab}   
\alpha ( \beta \circ \gamma ) &\subseteq
 (  \gamma \circ \alpha \beta \circ {\stackrel{ \ell  }{\dots}} ) \circ
 (     {\stackrel{ \ell }{\dots}} \circ  \alpha \gamma  \circ  \beta  ). 
\end{align}

\item
 If $n \geq 4$, then there is a locally finite
$n$-alvin variety in which the following congruence identities 
fail:
\begin{align} \labbel{abog}    
\alpha ( \beta \circ \gamma ) & \subseteq
\alpha \beta  \circ  ( \alpha ( \gamma \circ \beta ))^{\ell-1} \circ \alpha \gamma,
 \\
\labbel{abogo}   
\alpha ( \beta \circ \gamma ) &\subseteq
(\alpha \beta \circ \gamma \circ \alpha \beta \circ {\stackrel{ \ell }{\dots}} ) \circ
 (     {\stackrel{ \ell }{\dots}} \circ 
 \alpha \gamma  \circ  \beta  \circ  \alpha \gamma ). 
\end{align}
\end{enumerate}
 \end{theorem}   

\begin{proof}
The case $n=2$ in (i) is witnessed by the variety of 
distributive lattices.
Recall that, by convention, 
$    \gamma \circ   \alpha \beta \circ {\stackrel{1}{\dots}}= \gamma $.
 
The case  $n=4$ in (i) is witnessed by Baker's variety $\mathcal {B}$,
as proved in Proposition \ref{baker+}. To get an example which
is locally finite, consider $\mathcal {B}^d$, instead. 

The rest of the proof proceeds by simultaneous induction
as in the proof of Theorem \ref{buh}.
Notice that here we necessarily skip the case
$n=2$ in (ii).
This is the reason why we need consider 
the case $n=4$ in (i) in the basis of the induction. 

Suppose that $n \geq 4$ and that (i) holds for $n-2$.
Thus
there is an $n{-}2$-distributive variety 
in which, say, 
$\alpha ( \beta \circ \gamma ) \subseteq ( \alpha ( \gamma \circ \beta ))^{\ell-1}$
fails, as witnessed by some algebra $\mathbf D$.
By taking $\delta= \beta $, $\varepsilon= \gamma $ and 
   $\chi=( \alpha ( \gamma \circ \beta ))^{\ell-1}$
in Theorem \ref{thmba}(iv) and using  Theorem \ref{thmba}(i)
and the arguments in the proof of Theorem \ref{buh},
we get an $n$-alvin algebra in which 
$\alpha ( \beta \circ \gamma ) \subseteq
\alpha \beta  \circ  ( \alpha ( \gamma \circ \beta ))^{\ell-1} \circ \alpha \gamma $
fails. Thus (ii) holds for $n$.

Suppose that $n \geq 6$ and that (ii) holds for $n-2$.
Thus 
there is an $n {-} 2$-alvin variety 
$\mathcal V$ in which 
$\alpha ( \beta \circ \gamma ) \subseteq
\alpha \beta  \circ  ( \alpha ( \gamma \circ \beta ))^{\ell-2} \circ \alpha \gamma $
fails.
Taking $ \chi = \alpha \beta  \circ  ( \alpha ( \gamma \circ \beta ))^{\ell-2}
 \circ \alpha \gamma$ in Theorem \ref{thmbakbis}(ii),
by Lemma \ref{lembak}(i) and  the usual arguments,
we get  an $n$-distributive algebra 
in which 
$\alpha ( \beta \circ \gamma ) \subseteq
\alpha ( \gamma \circ \beta ) \circ 
  \alpha \beta  \circ  ( \alpha ( \gamma \circ \beta ))^{\ell-2} \circ \alpha \gamma
\circ \alpha ( \gamma \circ \beta ) $ fails.
Thus the identity \eqref{agob}  fails for $\ell$, since, for congruences,  
$\alpha ( \gamma \circ \beta ) \circ 
  \alpha \beta  = \alpha ( \gamma \circ \beta ) $ and
$\alpha \gamma
\circ \alpha ( \gamma \circ \beta ) =
 \alpha ( \gamma \circ \beta ) $.
On the other hand, if
\eqref{abogo} fails in $\mathcal V$,
we get  an $n$-distributive algebra 
in which 
\eqref{goab} fails using   
 Proposition \ref{thmbakbiscor}.
\end{proof}
 
Since $ \alpha \gamma  \circ \alpha \beta   \subseteq
\alpha ( \gamma  \circ \beta )$,
then the variety constructed in Theorem \ref{buhbis}(i)
is $n$-distributive and not $n$-alvin.
Similarly,  the variety constructed in Theorem \ref{buhbis}(ii)
is $n$-alvin and not $n$-distributive.
Thus we get another proof of some results
from \cite{FV}. We shall 
discuss this aspect in more detail in Section \ref{moreex},
where we shall present many other related results. 
See, in particular, Corollaries \ref{corsumupdist} 
and \ref{corsumupgumm}.

\begin{remark} \labbel{conidabog}
As in Remarks \ref{conid} and \ref{conidmod},
within a variety the identities 
\eqref{agob} - \eqref{abogo}
are equivalent to the existence of certain terms. 
For example, identity  \eqref{agob} (resp., \eqref{abog}) 
is equivalent to a weaker form 
of the alvin (J{\'o}nsson) condition in which the 
equations $t_h(x,y,x)=x$ are assumed 
only for even (odd) $h$.  

Similar ``defective'' sequences of terms shall be considered in 
Sections \ref{gummdef}, \ref{mixed} and \ref{moreex}.
See Definitions \ref{defgumm}, \ref{mix} and \ref{deflev}.  
\end{remark}

\section{Optimal bounds for varieties with directed terms} \labbel{dirsec} 

The assumption that $n$ is even is not necessary in the following theorem.
Recall that our counting conventions
are different from \cite{adjt}, as far as 
directed J{\'o}nsson terms are concerned. Cf.\ Remark \ref{count}.    

\begin{theorem} \labbel{dir}
(i) For every  $n \geq 2$, 
there is a locally finite  $n$-directed-distributive variety
which is not $2n{-}1$-reversed-modular,
hence, by Proposition \ref{modeq},  not  $2n{-}2$-modular.

(ii)
In the other direction, every
 $n$-directed-distributive variety
is  $2n{-}1$-modular.
Actually, every
 $n$-directed-distributive variety
satisfies $\alpha( \beta \circ \gamma \circ \beta )
\subseteq \alpha \beta \circ \alpha \gamma 
\circ {\stackrel{2n-1}{\dots}}  \circ \alpha \beta   $.
 \end{theorem}

\begin{proof}
(i) We first consider  
the cases $n=2$ and  $n=3$. 

A counterexample in the case $n=2$
is given by the variety of distributive lattices.
Indeed, a ternary majority term $t_1$ provides a sequence 
$t_0,t_1,t_2$ of directed J{\'o}nsson terms,
where $t_0 $ and $ t_2$ are projections.
As we have remarked in the proof of
Theorem \ref{buh},
the variety of distributive  lattices is not 
 $3$-reversed-modular.

To deal with the case $n=3$,
consider Baker's variety $\mathcal {B}$
recalled in Definition \ref{bakerdef}. 
As noticed in \cite[p.\ 11]{jds},
Baker's variety has a sequence $d_0, d_1,d_2,d_3$
of directed J{\'o}nsson terms (including the two projections),
that is, Baker's variety is $3$-directed-distributive in the present
terminology.   In fact, directed J{\'o}nsson terms 
for  $\mathcal {B}$ are given by the two projections
together with the terms $t_1$ and $t_2$  from Construction 
\ref{bak} in the case $n=3$.

By Proposition \ref{baker+},
Baker's variety
is not  $5$-reversed-modular.
 Thus the example of Baker's variety  takes care of the case $n=3$ in (i).
Baker's variety is not locally finite; however, all the above arguments
work in case we consider $\mathcal {B}^d$,
the variety defined like Baker's, but considering only reducts of distributive
lattices.  The variety $\mathcal {B}^d$ is indeed locally finite.

The rest of the proof  of (i) proceeds by induction on $n$.
Suppose that $n \geq 4$ and that  the theorem holds for $n-2$.
By the inductive hypothesis,
there is  an $n{-}2$-directed-distributive variety $\mathcal {W}$ 
in which $2n{-}5$-reversed-modularity fails.    
By Remark \ref{conidmod}, there is 
some algebra $\mathbf D \in \mathcal {W}$
such that the congruence identity
$\alpha( \beta \circ \alpha \gamma \circ \beta )
\subseteq  \alpha \gamma \circ \alpha \beta 
\circ {\stackrel{2n-5}{\dots}} \circ \alpha \gamma    $ fails 
 in  $\mathbf D $. Arguing as in the proof
of Theorem \ref{buh}, we can assume that  
$\mathbf D $ has  only directed J{\'o}nsson  operations.

Performing Construction \ref{bak} using such a $\mathbf D$,
we obtain an $n$-directed distributive algebra $\mathbf B$,
by Lemma \ref{lembak}(ii). 
By assumption, $\mathbf A_4$ has   congruences 
$\tilde{\alpha}, \tilde{\beta}, \tilde{\gamma} $ 
such that 
$\tilde{\alpha}( \tilde{\beta} \circ \tilde{\alpha} \tilde{\gamma} \circ \tilde{\beta} )
\subseteq 
\tilde{\alpha} \tilde{\gamma} 
\circ
\tilde{\alpha} \tilde{\beta} 
 \circ {\stackrel{2n-5}{\dots}} \circ    
 \tilde{\alpha} \tilde{\gamma}  $ fails,
a fortiori,
$\tilde{\alpha}( \tilde{\beta} \circ \tilde{\alpha} \tilde{\gamma} \circ \tilde{\beta} )
\subseteq 
\tilde{\alpha} \tilde{\beta}
\circ 
\tilde{\alpha} \tilde{\gamma} 
 \circ {\stackrel{2n-7}{\dots}} \circ    
 \tilde{\alpha} \tilde{\beta} $ fails.
By taking $r= 2n-7$ in Theorem \ref{thmbak}(i),
we get that $\mathbf B$ generates an
$n$-directed-distributive 
 variety   in which
$2n{-}1$-reversed-modularity fails,
again by Remark \ref{conidmod}.
The arguments in the proof of Theorem \ref{buh}(i)
show that $\mathbf B$ can be taken to belong to a locally 
finite variety.   

(ii) The first statement is the special case 
$n=k$, $\ell=3$, $T= \alpha \gamma $  in the last displayed identity
in \cite[Proposition 3.1]{jds}. 
The stronger statement is obtained by
taking $T=  \gamma $, instead. 
Otherwise, (ii) can be proved in a way similar to Day's
Theorem (see the proofs of Lemma \ref{dayoptlem} 
and of Corollary \ref{dayopt}(i)), using the terms 
\begin{equation*} 
u_1 {\hspace {1 pt}=\hspace {1 pt}} t_1(x y w),  
\quad u_2 {\hspace {1 pt}=\hspace {1 pt}} t_1(x z w),  
\quad u_3{\hspace {1 pt}=\hspace {1 pt}} t_2(x y w),
  \quad u_4 {\hspace {1 pt}=\hspace {1 pt}} t_2(x z w),  
\quad u_5 {\hspace {1 pt}=\hspace {1 pt}} t_3(x y w), \dots
\end{equation*}    
Still another proof can be obtained from
Corollary \ref{propmixmod}(i) below 
in the case of the first statement and 
from Proposition \ref{propmixi} below
(taking $i=2$, $S_0=S_2= \beta $ and   $S_1= \gamma  $)
in the case of the second statement.    
\end{proof} 

\begin{theorem} \labbel{dirthm}
(i) If  $n \geq 2$ and $n$ is even,
then there is a locally finite $n$-distributive not
 $n{-}1$-directed-distributive variety.

(ii) 
For every $n \geq 2$, there is a locally finite  $n$-directed-distributive variety 
which is not
$2n{-}2$-alvin, hence
not $2n{-}3$-distributive.

(iii) 
More generally, for every $n \geq 2$, 
there is a locally finite  $n$-directed-\brfrt distributive variety 
in which the congruence identity 
$\alpha( \beta \circ \gamma ) \subseteq ( \alpha ( \gamma \circ \beta ))^{n-1}$ 
fails.
 \end{theorem} 

\begin{proof}(i)  By Theorem \ref{buh}(i), for every even $n \geq 2$,
there is a locally finite 
$n$-distributive variety $\mathcal {V}$ which is
not $2n{-}2$-modular.
By  Theorem \ref{dir}(ii)
every
 $n{-}1$-directed-distributive variety
is  $2n{-}3$-modular, in particular, 
$2n{-}2$-modular.
Thus $\mathcal {V}$ is not 
$n{-}1$-directed-distributive.

The first part in (ii) follows from (iii)
and Remark \ref{conid}, since 
$\alpha \gamma  \circ \alpha \beta  \subseteq \alpha ( \gamma \circ \beta )$.
By Remark \ref{conidcomm}(b),
if some variety $\mathcal {V}$ is not
$2n{-}2$-alvin, then $\mathcal {V}$ is
not $2n{-}3$-distributive.
Hence it is enough to prove (iii). 

(iii)
The identity $\alpha( \beta \circ \gamma ) \subseteq \alpha ( \gamma \circ \beta )$
fails in the  variety of distributive lattices, since otherwise, by taking $ \alpha =1$
(the largest congruence in the algebra under consideration),
we would get congruence permutability; however,  distributive lattices
 are not congruence permutable. 

By Proposition \ref{baker+},
the identity  
$\alpha( \beta \circ \gamma ) \subseteq ( \alpha ( \gamma \circ \beta ))^{2}$
fails in the variety   $\mathcal {B}^d$.
 As recalled in the proof of Theorem \ref{dir}, 
$\mathcal {B}^d$ is $3$-directed-distributive
 and locally finite.

So far, we have proved the cases $n=2$ and $n=3$ of (iii).  
The rest of the proof is by induction on $n$.
Suppose that $n \geq 4$ and that (iii) is true for $n-2$, thus there 
exists some $n{-}2$-directed-distributive variety $\mathcal {V}$
in which $\alpha( \beta \circ \gamma ) \subseteq 
( \alpha ( \gamma \circ \beta ))^{n-3}$  fails.
In particular,  there is some algebra 
$\mathbf D \in \mathcal {V}$ with elements 
$a,d \in D$ and congruences
$\tilde{\alpha}, \tilde{\beta}, \tilde{\gamma} $ of $  \mathbf D$
such that 
$(a,d) \in
\tilde{\alpha}( \tilde{\beta} \circ  \tilde{\gamma} )$
but
$ (a,d) \notin 
(\tilde{\alpha} ( \tilde{\gamma} \circ
 \tilde{\beta}))
^{n-3} $.
Arguing as in the proof of
Theorem \ref{buh}, we can assume that $\mathbf D$ has only
the  directed operations.
Applying Construction \ref{bakbis},
we get an $n$-directed-distributive algebra $\mathbf B$,
by Lemma \ref{lembak}(ii).
 Applying Theorem \ref{thmbakbis}(ii)
with $\chi= ( \alpha ( \gamma \circ \beta ))^{n-3}$, we get that   
the identity 
$ \alpha ( \beta \circ \gamma ) \subseteq ( \alpha ( \gamma \circ \beta ))^{n-1}$ 
fails in  $\mathbf B$.
Again, the arguments from the proof of Theorem \ref{buh}
show that we can get a counterexample belonging to 
a locally finite variety. 
\end{proof}  

\begin{remark} \labbel{dirmin} 
(a)
 In \cite[Observation 1.2]{adjt}
it is shown that, in the present terminology, 
 every $n$-directed-distributive variety 
is $2n{-}2$-distributive. Recall from 
Remark \ref{count}  that our counting convention
is slightly different in comparison with \cite{adjt}.
Theorem \ref{dirthm}(ii)
shows that the result is optimal.
In this respect, see also Theorems  \ref{propmix2}, \ref{sumup}
and  Remark \ref{useofbm} below. 

(b) In the other direction,
in \cite{adjt}
it is shown that 
every $n$-distributive variety 
is  $k(n)$-directed-distributive, for some 
$k(n)$. The $k(n)$ obtained from the proof in \cite{adjt} depends
only on $n$, not on the variety, but is quite large.   
On the other hand, 
the only inferior bound we know is
given by Theorem \ref{dirthm}(i),
namely, $k(n) \geq n$, for $n$ even.
Concerning small values of $n$,
it is obvious that a variety 
is $2$-distributive if and only if 
it is  $2$-directed-distributive.
A direct proof 
that every  $3$-distributive variety 
is  $3$-directed-distributive appears 
in \cite[p.\ 10]{ia}.
In the next proposition 
we prove the corresponding  result  for $n=4$.
It is likely that these results follow already 
from the arguments in  \cite{adjt}.

Notice that, on the other hand, by Theorem \ref{dirthm}(ii),
there  is a $3$-directed-\brfrt distributive not
$3$-distributive variety  and there is a
$4$-directed-distributive not
$5$-distributive variety.
 \end{remark}

\begin{proposition} \labbel{4dir}
Every  $4$-distributive variety 
is  $4$-directed-distributive.
 \end{proposition} 

 \begin{proof} 
From terms $t_0, \dots,  t_4$  satisfying J{\'o}nsson's equations
 we get directed J{\'o}nsson terms $s_0, \dots,  s_4$ as follows: 
\begin{align*}\labbel{4dirjon}
s_1(x,y,z) & = t_1(t_1(x,y,z),t_3(x,x,y),t_3(x,x,z)),
\\
s_2(x,y,z) & = t_2(t_2(x,z,z),t_2(x,y,z),t_2(x,x,z)),
\\
s_3(x,y,z) & = t_3(t_1(x,z,z),t_1(y,z,z),t_3(x,y,z)),
  \end{align*}    
taking, of course, $s_0$ and  $s_4$ to be the suitable projections. 
\end{proof}

Hence, so far, we cannot
exclude the possibility that, for every $n$,  
every $n$-distributive variety 
is  $n$-directed-distributive,
though this would be a quite astonishing result.

\section{Specular Maltsev conditions} \labbel{specm} 

The constructions in the 
previous sections   suggest the following definition.

\begin{definition} \labbel{specdef} 
Let $n$ be a natural  number. 
An algebra or a  variety is \emph{specular $n$-distributive}
(\emph{specular $n$-alvin},
 \emph{specular $n$-directed-distributive}) 
if it has a sequence 
$t_0, \dots,   t_n$ of terms
satisfying the J{\'o}nsson   
(alvin, directed J{\'o}nsson) equations, as well as
\begin{equation} \labbel{s}    \tag{S} 
t_{n-i}(z,y,x) = t_i(x,y,z), \qquad \text{ for $0 \leq i \leq \frac{n}{2}, $} 
  \end{equation}    
equivalently, 
for all indices $i$ with  
  $0 \leq i \leq n $.
\end{definition}   

\begin{remark} \labbel{specrmk}
In all the preceding arguments and in each case
the algebras and the varieties  we have constructed 
have terms (or can be chosen to have terms) $t_0, \dots, t_n$ which
satisfy 
the equations
\eqref{s}. 
In fact, the varieties providing the basis of the induction in Theorem \ref{buh}
can be chosen to be varieties with a specular ternary operation,
e.~g., the classical majority term
$t(x,y,z)= xy+xz+yz$ in the case of lattices,
and the Pixley term 
$t(x,y,z)= xy'+xz+y'z$ in the case of Boolean algebras.
Here $n=2$, hence the specularity condition
reads $t_1(x,y,z) = t_1(z,y,x)$. 
The algebras
$\mathbf L^r$ and  $\mathbf A^r$ from Constructions \ref{bak}  
and \ref{ba} are defined by means of specular terms, too.
Thus the induction step in the proof
of Theorem \ref{buh} provides varieties 
with terms satisfying \eqref{s}.
 
In the proofs of Theorems \ref{buhbis}, \ref{dir} and \ref{dirthm}, too, 
we use specular sequences of operations; in fact, 
the standard terms witnessing that Baker's
variety is $4$-distributive  are
$t_1(x,y,z)=x(y+z)$,
$t_2(x,y,z)=xz$ and
$t_3(x,y,z)=z(y+x)$
and the outer terms in the above sequence
witness that $\mathcal {B}$ is $3$-directed distributive.
Since all the algebras in our constructions have
specular terms and the constructions themselves
proceed in a specular way,
all the outcomes turn out to be specular.

We shall elaborate further on the above observations
in Section \ref{moreex}.
 \end{remark}

\begin{remark} \labbel{specnodd}
Notice that, in the case of the J{\'o}nsson and of the alvin conditions,
Definition \ref{specdef}, as it stands,  is interesting only for $n$ even. 
Indeed, if $n$ is odd, then equation \eqref{s}
and the J{\'o}nsson conditions imply
\begin{align*}
x=^{ \eqref{bas} \eqref{j3}}t_1(x,x,z) & = ^{\eqref{s}} 
t _{n-1}(z,x,x) = ^{\eqref{j3}} t _{n-2}(z,x,x)
= ^{\eqref{s}} 
t _{2}(x,x,z) = ^{\eqref{j3}} 
\\
 t _{3}(x,x,z) & = ^{\eqref{s}} t _{n-3}(z,x,x) = ^{\eqref{j3}} t _{n-4}(z,x,x) \dots
\\
\dots t _{n-2}(x,x,z)  & =^{\eqref{s}} t _{2}(z,x,x) = ^{\eqref{j3}} t _{1}(z,x,x)
= ^{\eqref{s}}    t _{n-1}(x,x,z) =^{ \eqref{bas} \eqref{j3}}z,
 \end{align*}      
where 
$=^{ \eqref{bas}}$ means that we are using some equation from
$\eqref{bas}$ and similarly for the other conditions. 
So in fact we are in a trivial variety. 
Recall that if $n$ is odd, then the J{\'o}nsson condition is 
equivalent to the alvin condition.
See Remark \ref{conidcomm}(a).

On the other hand, by Remark \ref{specrmk},
for every $n$ there is a nontrivial specular $n$-directed-distributive
variety. 

More generally, 
one can consider several kinds of specular ``mixed'' 
conditions. See Definition \ref{mix}.
For example, if $n$ is odd and $\ell= \frac{n-1}{2}  $,  we can consider terms   
satisfying \eqref{bas} and
  \begin{enumerate}[(a)]   
 \item   
 the J{\'o}nsson (alvin)
conditions for $h < \ell$,   
\item
the alvin (J{\'o}nsson)
conditions for $ \ell < h$, as well as
\item  
$t_\ell(x,z,z)=t_{\ell+1}(x,x,z)$ or, according to convenience,
 $t_\ell(x,x,z)=t_{\ell+1}(x,z,z)$. 
\end{enumerate}
In each case we get a condition which implies congruence distributivity
and is compatible with the specularity condition \eqref{s}.
Examples of varieties satisfying the above conditions
can be constructed using Remark \ref{spec=}(a). 
See Section \ref{mixed} for a general form
of such ``mixed'' conditions, some compatible and 
some not 
compatible with specularity. 
  \end{remark}

\begin{remark} \labbel{fullsymm}   
(a)
Remark \ref{specrmk} above shows that, for $n$ even,
there are many examples of specular $n$-distributive and 
of specular $n$-alvin varieties. For every $n$, there are many examples 
of  specular $n$-directed-distributive varieties.
Explicit examples shall be presented in Section \ref{moreex}.
 
In a parallel situation, 
Chicco \cite{chic} has studied specular conditions
connected with $n$-permutability, again showing that
the examples of specular varieties abound in that context, too. 
The above comments suggest that   specular Maltsev conditions
in general  deserve further study. 

For example, 
it is probably interesting to study \emph{specular sequences of Day 
terms},
namely, terms satisfying equations \eqref{d0} - \eqref{d3}
from Definition \ref{daydef}, as well as
$u_k(x,y,z,w)= u _{m-k}(w,z,y,x) $, for all $k  \leq m$.
We get \emph{specular reversed Day 
terms} if we exchange parity in \eqref{d2}.
In most cases, the 
Day or reversed Day terms
whose existence follows
from congruence distributivity
turn out to be specular in 
the  above sense, provided one starts with specular
distributive terms.
When $n$ is even, this is the case for Day's original construction
recalled in the proof of Corollary \ref{dayopt},
and for the terms constructed in the proof of Lemma \ref{dayoptlem}(a).
For arbitrary $n$, the proof of   Theorem \ref{dir}(ii) provides
$2n$ specular Day terms, too. 

(b) Notice that, when dealing with ternary 
terms witnessing congruence distributivity, 
and with the only exception we shall mention below,
 we necessarily deal with specularity, not with (full)
symmetry. 
An $m$-ary   term 
 $w$ is \emph{symmetric}  if it satisfies the equations
\begin{equation}\labbel{specnu}     
w(x_1, x_2, \dots) = w( x _{ \sigma (1)},  x _{ \sigma (2)}, \dots ),
 \end{equation}
for all permutations $\sigma$ of $m$.  
The mathematical literature about symmetric operations
is so vast that it cannot be reported here.
We just mention that recent results connected with universal algebra
can be found in \cite{CK} and that the near-unanimity terms
constructed in \cite{misharp} are symmetric.
Recent research deals also with
terms and operations satisfying  partial versions of symmetry.
See, e.~g., \cite{BKMMN}.
 The above list is not intended to be exhaustive; 
moreover, further
references can be found in the quoted works.
It is possible that there are connections among the 
above-mentioned studies about symmetrical
terms and the present notion of specular terms, but
this has still to be analyzed in detail.

We just mention that, when dealing with the conditions introduced
in Definition \ref{jondef}, 
 full symmetry can occur only in the case of a 
ternary majority term, that is, $2$-distributivity,
equivalently, $2$-directed-distributivity. 
In fact, as soon as some ternary term $t$ is symmetric and satisfies
$t(x,y,x)=x$, then $t$ must be a majority term.  

Another  set of equations which can be satisfied
by  a symmetric ternary term is the following \emph{minority condition}:
\begin{equation} \labbel{gr}   
x=t(x,y,y), \qquad x=t(y,x,y), \qquad x=t(y,y,x).
\end{equation}    
The above equations partially resemble 
the $2$-alvin condition, but are satisfied, for example, 
by the term $x+y+z$ in a group of exponent 2;  hence
the equations  \eqref{gr}  do not imply  congruence distributivity,
though they do imply congruence permutability.
Minority terms have been extensively studied in 
 Kazda,  Opr\v sal, Valeriote, Zhuk \cite{KOVZ}.
See also Remarks \ref{notalspec}(b) and \ref{defmin} 
for a few further comments.  
 
On the other hand, there are indeed 
full symmetric terms whose existence
implies congruence distributivity.
In \cite{misharp} a variety $\mathcal {N}_m$ 
has been constructed such that $\mathcal {N}_m$ has
an $m$-ary symmetrical near-unanimity term, is not $2m{-}4$-alvin,
hence not $2m{-}5$-distributive and is not $2m{-}3$-reversed-modular,
hence not $2m{-}4$-modular.
\end{remark}

We now notice that,
for even $n$, our constructions 
show that specularity does not influence distributivity levels. 
We shall prove stronger results in Section \ref{moreex};
however, we present the proof of the following corollary,
since it is particularly simple.

\begin{corollary} \labbel{corspec} 
If $n \geq 2$ and $n$  is even,
then there are  a  specular $n$-distributive locally 
finite variety 
and  a specular $n$-alvin locally finite variety 
which are not $n{-}1$-distributive.

If $n \geq 2$, 
then there is a  specular
$n$-directed-distributive locally finite variety 
which is not 
$n{-}1$-directed-distributive.
\end{corollary}

 \begin{proof}
By Remark \ref{specrmk},
the proof of  
Theorem \ref{buh} produces
a  specular $n$-distributive
variety $\mathcal {V}$ which is not $2n{-}2$-modular.
By Day's Theorem \ref{day}, 
 every $n{-}1$-\brfrt distributive variety is 
$2n{-}3$-modular, hence 
$2n{-}2$-modular. Thus $\mathcal {V}$ is not
$n{-}1$-distributive.
Again by Theorem \ref{buh} and 
Remark \ref{specrmk} we get a
specular $n$-alvin variety
which is not 
$2n{-}3$-modular
and the same argument as above applies.

If $n \geq 2$, then, by Theorem \ref{dir}(i) and 
Remark \ref{specrmk}, we have a
specular $n$-directed-distributive variety $\mathcal {W}$ 
which is not $2n{-}2$-modular.
However, by Theorem \ref{dir}(ii), 
every $n{-1}$-directed-distributive variety
is $2n{-}3$-modular, in particular,  $2n{-}2$-modular.
Hence $\mathcal {W}$ is not $n{-1}$-directed-distributive.
 \end{proof}

\begin{remark} \labbel{notalspec}   
(a)  For $n$ even, an  $n$-distributive variety is not necessarily
specular $n$-distributive. For example, 
just consider the $2$-distributive variety $\mathcal {V}$ 
with one ternary majority operation  $t=t_1$ satisfying no further
equation. We shall show that $\mathcal {V}$ is not
specular $2$-distributive, actually, $\mathcal V$ has no 
ternary specular term, except for the projection 
$p_2$ 
onto the second coordinate.

Every  term in $\mathcal {V}$ has a normal form, obtained
by applying the majority rule whenever possible.
Suppose by contradiction that 
there exists  a specular term
$s(x,y,z)$ in $\mathcal {V}$ distinct from 
$p_2$. 
Choose such an $s$  in normal form
 of minimal complexity.
Since $s$ is specular and distinct from  
$p_2$, then $s$ cannot be a variable.
hence it is  written  as
\begin{equation*}   
s(x,y,z) = t(u_1(x,y,z),u_2(x,y,z),  u_3(x,y,z) ),
 \end{equation*} 
for certain terms $u_1$, $u_2$ and $u_3$.
We  have
\begin{equation*}   
t(u_1(x,y,z),u_2(x,y,z),u_3(x,y,z) )=
t(u_1(z,y,x),  u_2(z,y,x),u_3(z,y,x) ),
\end{equation*}   
since $s$ is specular, 
hence 
\begin{equation*} 
u_1(x,y,z)=u_1(z,y,x), \quad
 u_2(x,y,z)=u_2(z,y,x), \quad
  u_3(x,y,z)=u_3(z,y,x),
\end{equation*}    
 since we are dealing with normal forms.

Since we have supposed that $s$ is specular,
distinct from the second projection and  of minimal complexity,
we have 
$u_1(x,y,z)=
 u_2(x,y,z)=
  u_3(x,y,z)=y$,
hence 
$s(x,y,z)=t(y,y,y)=y$, a contradiction.

(b) Similarly, 
the variety with a $2$-alvin  operation
satisfying
no further equation has no specular term
distinct from the second projection, in particular, 
no specular $2$-alvin  term.
The variety with a Maltsev operation (for permutability) satisfying
no further equation has no specular Maltsev term.
The variety with a minority operation
satisfying
no further equation has no specular minority term.

(c) It is likely that examples similar to (a) - (b) above can be worked out for every 
even $n >2$. On the other hand, we do not know examples of \emph{locally finite}
   $n$-distributive varieties which are not specular $n$-distributive.

(d) Suppose that some variety $\mathcal {V}$ has an idempotent term
$s$ such that $s(x,y)=s(y,x)$ holds.  
If $\mathcal {V}$ is $n$-directed-distributive
($n$ is even and $\mathcal {V}$ is $n$-distributive, $n$-alvin),
then  $\mathcal {V}$ is specular $n$-directed-distributive
(specular $n$-distributive, specular $n$-alvin).
Indeed, if $t_0, \dots, t_n$ is a sequence of terms witnessing the assumption,
then $t'_i(x,y,z) =  s(t_i(x,y,z), t_{n-i}(z,y,x))$
are terms witnessing the conclusion.
This is essentially an argument from \cite{chic}.   

(e) The same argument as in (a) above applies also to near-unanimity terms.
If $m \geq 3$
and  $\mathcal {V}$ is the variety 
with only one  $m$-ary operation $w$ satisfying exactly 
the near-unanimity identities, 
then, for every $j \geq 2$,  $\mathcal {V}$ has no  $j$-ary
term which is symmetric in the sense
of equation \eqref{specnu}. 
In fact, the argument is slightly simpler.
\end{remark}

\section{Gumm, directed Gumm and defective alvin terms} \labbel{gummdef} 
\subsection*{Gumm and defective terms}
Most results of the present paper apply
to Gumm terms, too.
This is quite surprising, since 
the existence of Gumm terms does not imply 
congruence distributivity; actually, a variety $\mathcal {V}$ 
has Gumm terms if and only if $\mathcal {V}$
is congruence modular.
More generally, our results apply to
the weaker notion of  defective Gumm terms;
as
 introduced in part (b) of  
 the following definition.
 
As we have briefly discussed in the introduction, it is  unusual---but
much more convenient for our purposes here---to
introduce Gumm 
terms 
as defective alvin terms.
Again, we refer to  \cite[Remark 4.2]{ntcm}
for a more complete discussion.
A more general study of ``defective''
conditions appears in
Kazda and Valeriote \cite{KV}.
See also Section \ref{mixed} below.  
Defective conditions in the present terminology 
 correspond to dashed lines in paths in
the  terminology from \cite{KV}.
See Remark \ref{lrrmk}(b) below for more details.

\begin{definition} \labbel{defgumm}    
(a)
We get a sequence of \emph{Gumm
terms} \cite{G1,G}
if in Definition \ref{jondef} the condition 
$x=t_1(x,y,x)$ from \eqref{bas} 
 is not assumed in the definition 
of alvin terms. More formally, for $\ell \leq n$,  it is convenient to 
 consider the following reduced set (B$^{\widehat \ell}$) of equations:
\begin{equation*}
\labbel{basi}    
\begin{aligned}    
  x&= t_0(x,y,z),  \qquad \qquad t_{n}(x,y,z)=z, 
 \\
  x  &=t_h(x,y, x), \quad 
\text{ for } 0 \leq h \leq n, \ h \neq \ell.
\end{aligned}
 \end{equation*}       
In the above situation we shall say that the sequence of terms
$t_0, \dots, t_n$ is \emph{defective}
at place $\ell$. 

Under the above notation,
a sequence of Gumm terms is a sequence
satisfying (B$^{\widehat 1})$,
as well as \eqref{al} from Definition \ref{jondef}.   

With the above
definition, if $t_0, t_1, t_2$ is a sequence of Gumm terms, then
$t_1$ is a Maltsev  term for permutability \cite{M}.   
Recall that we are not exactly assuming J{\'o}nsson Condition \eqref{j3},
but the alvin variant \eqref{al} 
 in which even and odd are exchanged. 
This is fundamental: see Remark \ref{gummrmk}(c) below. 
Notice also that many authors,
including Gumm himself,  
define Gumm terms in a slightly different fashion. See
Remark \ref{galt}(b) below.

(b)
If in the definition of Gumm terms we discard also the equation
$x=t_{n-1}(x,   y,  \allowbreak  x)$ we get a sequence of 
\emph{doubly defective alvin terms}, or 
\emph{defective Gumm terms} \cite{DKS,ntcm}.
More formally, and with the obvious extension of the above convention,
 a sequence of   defective Gumm terms
is a sequence satisfying 
 (B$^{\widehat 1, \widehat{n-1}})$
and \eqref{al}.
We shall soon see that the 
the parity of $n$ and the places at which the missing
equation(s) occur are highly relevant. See Remarks \ref{gummrmk}(b)(c)
and \ref{polin}.

(c)
As in Definition \ref{nd}, a variety or an algebra  is said 
to be \emph{$n$-Gumm} (\emph{defective $n$-Gumm})
  if it has a sequence
$t_0, \dots, t_n$ of
 Gumm (defective Gumm) terms.
\end{definition}

\begin{remark} \labbel{galt} 
(a)
 If $n$ is even, then we get a definition 
equivalent to \ref{defgumm}(a) if we ask that 
 (B$^{ \widehat{n-1}})$
and \eqref{al}
are satisfied. Indeed, the definitions are 
shown to be equivalent
by reversing both the order of terms and of variables.
Otherwise, use the first displayed line in Remark \ref{gummrmk}(a)
below, taking converses and exchanging $\beta$ and $\gamma$.

Notice that, on the contrary, when $n$ is odd 
the definitions 
are not equivalent, actually, the conjunction of 
 (B$^{ \widehat{n-1}})$
and \eqref{al} turns out to be a trivial condition, if 
$n$ is odd. This is proved arguing as in  Remark \ref{gummrmk}(c)
below.    

(b) On the other hand,  when $n$ is odd,
we get a definition equivalent 
to \ref{defgumm}(a) if we ask that 
 (B$^{ \widehat{n-1}})$
and \eqref{j3}  are satisfied.
Notice that here we are considering   \eqref{j3}  rather than \eqref{al}.

At first glance this equivalence might appear strange and counterintuitive.
However it is enough to observe that, \emph{when $n$ is odd}, if we reverse the order
of terms and of variables,  condition  \eqref{j3} transforms into \eqref{al} and,
obviously, (B$^{ \widehat{n-1}})$ transforms into  (B$^{ \widehat 1})$.
The above equivalence explains the reason why
the convention in the literature is not uniform.
Some authors define Gumm terms by
taking the ``permutability part'', namely, the  defective part,
at the end, while others take it at the beginning.

For our purposes, it is convenient to define 
Gumm terms by taking the ``permutability part''
 at the beginning, rather than
at the end. In this way, we have no need to shift
from the J{\'o}nsson and the alvin conditions, according to the parity of 
$n$. 
The definition we have adopted has also the advantage of 
providing a finer way of counting the number of terms:
compare \cite[p.\ 12]{jds}. 
To the best of our knowledge,
this formulation of the Gumm condition
is due to \cite{LTT,T}.

(c)
As we mentioned, $2$-alvin is equivalent to 
arithmeticity, that is, to the conjunction
of congruence distributivity and permutability.
Moreover, defective  $2$-alvin
is equivalent to congruence permutability.
The above observations suggest that the alvin conditions
share some aspects in common with congruence permutability.
This is indeed the case: we have discussed 
 this matter in more detail  in \cite[Remark 2.2]{ia}
for the case $n=3$,  
and in \cite[Remark 4.2]{ntcm} for the general case.
Another exploitation of this fact is presented here
in the proof of Theorem \ref{propmix2}(iii) below.
Compare also Remark \ref{spend} and \cite{LGA}.  
\end{remark}
 
The definition of directed Gumm
terms \cite{adjt} shall be recalled in Definition \ref{dirgumm}(a) below. 
The definition of two-headed directed Gumm
terms  shall be given in Definition \ref{dirgumm}(b) below. 
In order for the reader to appreciate the exact power of the 
notions involved,  we state the following 
theorem even if we have not yet given all the definitions; in fact,
we shall not actually need the theorem  in what follows.

\begin{theorem} \labbel{Gth}
\cite{DKS,G1,G,adjt}
For every variety $\mathcal {V}$, the following conditions are equivalent. \begin{enumerate}[(i)] 
  \item 
$\mathcal {V}$ is congruence modular.
\item
$\mathcal {V}$ has a sequence of Gumm terms.
\item
$\mathcal {V}$ has a sequence 
$t_1, \dots, t_n$ 
of  defective Gumm terms, for some even $n$.
\item
$\mathcal {V}$ has a sequence of directed Gumm terms.
\item
$\mathcal {V}$ has a sequence of two-headed directed Gumm terms.
 \end{enumerate} 
 \end{theorem} 

The equivalence of (i) and (ii)
is due to H.-P.\ Gumm \cite{G1,G}. 
 The equivalence of (i) and (iii)  appears 
in \cite{DKS}  under different notation and terminology. See
\cite[Proposition 6.4]{ntcm} for further details. 
In any case, (ii) $\Rightarrow $    (iii) is obvious, and
it follows from Lemma \ref{dayoptlem}(a)
or Corollary \ref{propmixmod}(ii)(c) below that 
(iii) implies congruence modularity
(the case $n=2$ is obvious). 
The equivalence of
(ii) and (iv) is proved in \cite{adjt}.
By adding one more term, a trivial projection
at the beginning, and shifting the indices,
it is obvious that (iv) implies (v). Compare
the proof of Theorem \ref{Dth}.
If $n=2$ in Definition \ref{dirgumm}(a),
just add the projection two times.
It follows from Corollary \ref{propmixmod}(ii)(c), to be proved below,
 that (v) implies congruence modularity.

\begin{proposition} \labbel{gummopt} 
(i) If $n \geq 4$ and $n$ is even, then
all $n$-Gumm varieties and all
 defective $n$-Gumm varieties  are $2n{-}3$-reversed-modular,
in particular, $2n{-}2$-modular.

(ii) The result is optimal: 
for every even $n \geq 2$ there is
a  defective $n$-Gumm locally finite variety,
in particular,  $n$-Gumm,
 which is not $2n{-}3$-modular.
\end{proposition}

\begin{proof}
An $n$-Gumm variety is, in particular,
 defective $n$-Gumm.
 The fact that if $n$ is even, then  a 
 defective $n$-Gumm variety is $2n{-}3$-reversed-modular
is simply a reformulation of 
the second statement in Lemma \ref{dayoptlem}(a). 

In order to prove (ii), recall that in Theorem \ref{buh}(ii) an $n$-alvin 
not  $2n{-}3$-modular variety $\mathcal {V}$ has been  constructed.
In particular, $\mathcal {V}$  is a defective
$n$-Gumm variety.
\end{proof}

\begin{remark} \labbel{gummrmk}    
(a) As in Remark \ref{conid}, within a variety,
 each condition on the left in the following table
is equivalent to the condition on the right.
\begin{equation*} 
\begin{array}{cl} 
 \text{$n$-Gumm}   
&
\alpha ( \beta { \hspace{1pt} \circ \hspace {1pt} } \gamma ) 
\subseteq 
 \alpha (\gamma  { \hspace{1pt} \circ \hspace {1pt} }   \beta)
{ \hspace{1pt} \circ \hspace {1pt} }
 (\alpha \gamma  { \hspace{1pt} \circ \hspace {1pt} }
 \alpha \beta { \hspace{1pt} \circ \hspace {1pt} }
 {\stackrel{n-2}{\dots}} ) 
\\ 
 \text{defective $n$-Gumm}   
& 
\alpha ( \beta { \hspace{1pt} \circ \hspace {1pt} } \gamma ) 
\subseteq 
 \alpha (\gamma  { \hspace{1pt} \circ \hspace {1pt} }   \beta)
{ \hspace{1pt} \circ \hspace {1pt} }
(\alpha \gamma  { \hspace{1pt} \circ \hspace {1pt} } \alpha \beta
  { \hspace{1pt} \circ \hspace {1pt} }  {\stackrel{n-4}{\dots}}
  { \hspace{1pt} \circ \hspace {1pt} } \alpha \beta )
\circ \alpha ( \gamma { \hspace{1pt} \circ \hspace {1pt} } \beta ),
\end{array}
\end{equation*}        
where in the first line we are assuming $n \geq 2$
and in the second line we are assuming 
$n $ even  and $n \geq 4$.  

(b) As another example of applications of congruence identities
one sees immediately, arguing as in  (a), that,  for $n$ odd and
$n > 1$,  being  defective $n$-Gumm
is a trivial condition.
Indeed, the condition is equivalent to the trivially true identities 
$\alpha ( \beta \circ \gamma ) 
\subseteq 
 \alpha ( \gamma \circ  \beta \circ \gamma )$, for $n=3$, 
and
$\alpha ( \beta \circ \gamma ) 
\subseteq 
 \alpha (\gamma  \circ   \beta)
\circ 
(\alpha \gamma  \circ \alpha \beta  \circ {\stackrel{n-4}{\dots}} \circ \alpha \gamma  )
\circ \alpha ( \beta \circ \gamma  ) 
$, for larger $n$'s.  
It requires a bit of ingenuity to see
that the conditions are trivial, when expressed in function of terms.
Take $t_{n-1}$ to be the projection onto the second coordinate  
and all the terms before $t_{n-1}$ as the projection onto the first coordinate.
Recall that here $n$
is odd and that  we are considering a defective alvin condition, namely, 
we are assuming \eqref{al} from Definition \ref{jondef}, rather than  \eqref{j3}. 

(c) Expanding on an observation
from \cite{adjt}, we get a trivial condition also
by considering defective J{\'o}nsson or defective directed J{\'o}nsson terms,
namely, discarding the equation $x=t_1(x,y,x)$ from each set of conditions.
The existence of defective J{\'o}nsson terms is equivalent to the trivial
congruence identity 
$\alpha( \beta \circ \gamma ) \subseteq \alpha ( \beta \circ \gamma ) \circ \alpha \beta 
\dots$. 
In both cases, 
 take $t_1$ as the projection onto the second
coordinate and all the terms after  $t_1$
 as the projection onto the third coordinate,
to show that the conditions are satisfied by every variety.
This is essentially a remark  on  \cite[p.\ 205]{adjt},
where it is used under the assumption that \emph{all} the equations
$  x  =t_h(x,y, x)$ 
($ 0 \leq h \leq n$)
 are discarded. However, the argument works
if we just discard only $x=t_1(x,y,x)$.   

An even more general fact is presented in \cite[Subsection 3.3.1]{KV}. 

(d) Anyway, there is a useful and interesting notion of directed
Gumm terms \cite{adjt}. The remarks in (c) above show
that directed Gumm terms cannot be plainly obtained as defective 
directed J{\'o}nsson terms, as we did for ``undirected'' Gumm terms.
However, notice that there is the possibility of considering various kinds of
\emph{mixed J{\'o}nsson terms}, in which, for each index,
we choose some condition among \eqref{m1}, \eqref{m3}, \eqref{m2}.   
Details shall be presented in  Definition \ref{mix}  below,
where a condition parallel to \eqref{m2} shall be also taken
into account.

Then directed Gumm terms, as introduced by
\cite{adjt}, are actually defective mixed J{\'o}nsson terms,
when  \eqref{m2}  is assumed for all indices, 
with the exceptions of  $n-2$, for which   \eqref{m3}
is assumed, and 
of $n-1$, for which   \eqref{m1}
is assumed.
The definition applies both  to the case 
  $n$ even and to the case $n$ odd.
See Definition \ref{dirgumm}(a) below for formal details
and Remarks \ref{mixrem} and \ref{lrrmk}(b) 
 for the connection with the idea of mixed J{\'o}nsson terms.
\end{remark}

\subsection*{Directed Gumm terms} 
\begin{definition} \labbel{dirgumm}
(a) \cite{adjt} If $n\geq 2$, a   sequence 
$t_0, t_1, \dots, t_{n-2}, q$ of ternary terms  
is a sequence of \emph{directed Gumm terms}
if the following equations are satisfied:
\begin{align}
 \labbel{dg0} \tag{DG0}
x&=t_h(x,y,x), 
\quad \phantom {_{+1}} \text{for $0 \leq h \leq n-2$},
\\
\labbel{dg1} \tag{DG1}
x&=t_0(x,y,z),
\\
\labbel{dg2} \tag{DG2}
 \quad t_h(x,z,z)&=t_{h+1}(x,x,z), 
  \quad \text{for $0 \leq h < n-2$},
\\
\labbel{dg3} \tag{DG3}
t_{n-2}(x,z,z)&= q(x,z,z), \quad   q(x,x,z)=z. 
\end{align}   

Notice that if $n=2$ in the above definition, then 
$q$ is a Maltsev term for permutability.
Thus, for $n=2$, the existence of directed Gumm terms  
is equivalent to the existence of  Gumm terms
(and equivalent to congruence permutability).
Notice the parallel situation with respect to 
  J{\'o}nsson terms and directed J{\'o}nsson 
terms, which give equivalent conditions in the case 
$n=2$, as we mentioned in Definition \ref{jondef}.

(b) If $n\geq 4$, a   sequence 
$p, t_2, \dots, t_{n-2}, q$ of ternary terms  
is a sequence of \emph{two-headed directed Gumm terms}
if the following equations are satisfied:
\begin{align}
 \labbel{thg0} \tag{THG0}
x&=t_h(x,y,x), 
\quad \phantom {_{+1}} \text{for $2 \leq h \leq n-2$},
\\
\labbel{thg1} \tag{THG1}
x&=p(x,z,z), \quad p(x,x,z)= t_2(x,x,z),
\\
\labbel{thg2} \tag{THG2}
 \quad t_h(x,z,z)&=t_{h+1}(x,x,z), 
  \quad \text{for $2 \leq h < n-2$},
\\
\labbel{thg3} \tag{THG3}
t_{n-2}(x,z,z)&= q(x,z,z), \quad   q(x,x,z)=z. 
\end{align}   

(c) If in (b) above we also require that the terms
$p$ and  $q$
satisfy the  equations 
$x= p(x,y,x) $ and   $x= q(x,y,x) $
we get a sequence of \emph{directed terms with two alvin heads}.
If an algebra or a variety has a sequence of such terms, 
we say that it is \emph{$n$-directed with alvin heads}. 
Of course, in the situation described here in (c), the terms 
$p$ and  $q$ can be safely relabeled as 
$t_1$ and $t_{n-1}$. 

Notice that being $4$-directed with
alvin heads is the same as being $4$-alvin.    
\end{definition}

The indexing of terms in the above definitions has been 
chosen in order to match with
the more general notions we shall introduce in
Definition \ref{mix}, and to get corresponding results,
as far as modularity levels are concerned. 

In the formalism from \cite{KV}, two-headed directed Gumm terms
correspond to a pattern path 
with forward solid edges everywhere, except for two
dashed backwards edges  on the outer ends.
The case of  directed terms with alvin heads
is similar, but in this case all the edges are solid.
More details shall be given in Remark \ref{lrrmk}(b).

\begin{theorem} \labbel{thm2hg}
Suppose that $n \geq 4$.
  \begin{enumerate}[(i)]   
 \item   
If some variety $\mathcal {V}$ has two-headed directed Gumm terms   
$p, t_2, \dots, t_{n-2}, q$,
then $\mathcal {V}$ is 
$2n{-}3$-reversed-modular,
hence $2n{-}2$-modular. 
In particular, this applies to any variety which is 
$n$-directed with alvin heads.
 
\item
 There is some locally finite variety $\mathcal {V}$ 
 which has  two-headed directed Gumm terms   
$p, t_2, \dots, t_{n-2}, q$,
but  is not $2n{-}3$-modular,
hence (i) is the best possible result.
A variety $\mathcal {V}$ as above can be chosen to be
$n$-directed with alvin heads.
 \end{enumerate}
\end{theorem}

 \begin{proof} 
The proof of (i) is obtained by merging the arguments 
in the proofs of Lemma \ref{dayoptlem} and of Theorem \ref{dir}(ii),
namely, considering a sequence of terms
 which behaves as  in the proof of Theorem \ref{dir}(ii)  on the inside
and which  on the outer edges is defined as in the proof of
 Lemma \ref{dayoptlem}.
In detail, define
\begin{equation*}
u_1=p(x,y,z), \quad
u_2=t_2(x,y,w), \quad
u_3=t_2(x,z,w), \quad
u_4=t_3(x,y,w), \ \dots
  \end{equation*}    
and symmetrically at the other end.

A more general result shall be proved in
Corollary \ref{propmixmod}(ii)(c) 
 below.

(ii) By Theorem \ref{dir}(i),
there is an $n{-}2$-directed-distributive variety $\mathcal {W}$ 
  which is not $2n{-}5$-reversed-modular. By Remark \ref{conidmod},
there is some algebra $\mathbf D \in \mathcal W$ such that 
the congruence identity 
$ \alpha ( \beta \circ \alpha  \gamma \circ \beta )
 \subseteq 
 \alpha \gamma  \circ  \alpha \beta  \circ {\stackrel{2n-5}{\dots}} 
\circ \alpha \gamma $ fails in $\mathbf D$.
As in the proof of Theorem \ref{buh},
we can suppose that  $\mathbf D$ has only
directed J{\'o}nsson operations.
If we perform Construction \ref{ba}
by considering such an algebra $\mathbf D$, 
then, arguing as in the proofs of Lemma \ref{lembak}(i)(ii)   
and Theorem \ref{thmba}(i),   
we get some algebra $\mathbf B$
belonging to a variety $\mathcal {V}$  with
two-headed directed Gumm terms   
$p, t_2, \dots, t_{n-2}, q$.
Actually, 
the  equations 
$x= p(x,y,x) $ and   $x= q(x,y,x) $ hold in $\mathcal {V}$, hence
we get a variety which is  $n$-directed with alvin heads.
By Theorem \ref{thmba}(ii),
the identity  
$ \alpha ( \beta \circ \alpha  \gamma \circ \beta )
 \subseteq 
 \alpha \beta   \circ  \alpha \gamma  \circ {\stackrel{2n-3}{\dots}} 
\circ \alpha \beta  $ fails in $\mathbf B$,
thus $\mathcal {V}$ is not 
$2n{-}3$-modular, again by  Remark \ref{conidmod}.
The arguments showing that $\mathcal {V}$ can be chosen
to be locally finite are from Theorem \ref{buh}.
\end{proof}

\section{Mixed J{\'o}nsson  terms and congruence modularity} \labbel{mixed} 
\subsection*{Mixed terms and a way to describe them} 
We need a companion equation to \eqref{m1}, \eqref{m3}, \eqref{m2}
from Definition \ref{jondef}. 
\begin{equation}
 \labbel{m4} \tag{M$\xtoz$} 
 t_{h}(x,x,z) =
t_{h+1}(x,z,z).
 \end{equation}    

As clearly explained in \cite{adjt,KV},
the symmetry between \eqref{m2} from Definition \ref{jondef}
 and  \eqref{m4} above
is only apparent. The two equations are
significantly different,  under the additional 
assumptions
  $x= t_0(x,y,z) $ and $  t_{n}(x,y,z)=z$.
Suppose for the rest of the paragraph
that we are not assuming  the  equations
on the second line of  \eqref{bas} but, for each index,
one of \eqref{m1}, \eqref{m3}, \eqref{m2}, \eqref{m4}
is satisfied.
If we assume \eqref{m2} so much as for  one index,
we get a trivial condition, as in Remark \ref{gummrmk}(c).
On the other hand, if we assume \eqref{m4}
for all indices, we get a condition equivalent to 
$n$-permutability \cite{HM}.  See \cite{adjt,KV}
for further comments and details. 

Remark \ref{gummrmk}(d) 
and Definition \ref{dirgumm} suggest the following definition. 

\begin{definition} \labbel{mix}
A sequence $t_0, \dots, t_n$
is a sequence of \emph{mixed J{\'o}nsson terms}
if the equations in \eqref{bas} from Definition \ref{jondef}
hold and, moreover, for each $h$ ($0 \leq h<n$),
at least one of the equations      
\eqref{m1}, \eqref{m3}, \eqref{m2}, \eqref{m4} is satisfied.

Any choice of some specific equation for each $h< n$
determines a \emph{mixed condition}.  

If in the above definition
 we only assume 
(B$^{\widehat \ell}$) instead of \eqref{bas}, we say that the sequence
is \emph{defective at place $\ell$}.
Compare Definition \ref{defgumm}.
Sequences defective at two or more places are defined similarly.
We shall see that sequences defective at one or both ``ends''
$1$ and $n-1$ are particularly interesting
and enjoy special properties. Compare also Lemma \ref{dayoptlem}
and Theorem \ref{thm2hg}(i). 
Sequences defective at ``internal places''
are less well-behaved. See Remark \ref{polin}
in conjunction with Remark \ref{conidabog}.
 \end{definition}   

In a different context and with different terminology
mixed and defective mixed conditions appeared 
in Kazda and Valeriote \cite{KV} as Maltsev conditions
associated to pattern paths.
See \cite[Subsection 3.2]{KV} and
Remark \ref{lrrmk}(b) below. 
Mixed conditions
 deserve a
more attentive study,
but, of course, the present paper is already long enough
and it is not possible
to include a detailed study of such conditions.
We just present a few remarks.
In particular, we shall use the notion of a mixed
condition in order to present alternative proofs,
in a uniform way,
for Theorems \ref{dayshr}(i),
\ref{dir}(ii), \ref{thm2hg}(i), Lemma \ref{dayoptlem}, 
Corollary \ref{dayopt}(i)-(iii),
Proposition \ref{gummopt}(i)
and for the statements in Remark \ref{daystrong}. 
In a sense, the main point in the present section is
just to introduce an appropriate notation and then we
repeat some classical proofs in a greater generality. 
However, let us remark that the notational issue is not that trivial.
As we are going to see soon, in 
order to provide a good way to generalize the classical methods,
we need to describe mixed conditions 
by means of the way ``variables are moved'', rather than
by listing the equations which are satisfied.
See, in particular, Definition \ref{lr}, Remarks \ref{lrr}, \ref{lrrmk}(a),
as well as Theorem \ref{propmix2}, together with its proof.   

As we mentioned, though  the terminology adopted here is different,
 there is a strong  correspondence with some parts of \cite{KV}. 
See Remark \ref{lrrmk}(b).

\begin{remark} \labbel{mixrem}   
Clearly the J{\'o}nsson, the alvin and the directed distributive
conditions are examples of 
mixed conditions
as introduced in Definition \ref{mix}.
In the case of J{\'o}nsson and alvin terms
we alternatively
 choose equations \eqref{m1} and \eqref{m3};  
if the starting equation is \eqref{m1} we get J{\'o}nsson 
terms, otherwise we get alvin terms. 
In the case of directed J{\'o}nsson
terms we always choose equation \eqref{m2}.
 
If we always choose equation \eqref{m4},
we get a sequence 
$t_1, \dots, t_{n-1}$ of Pixley  terms, in the terminology from
\cite[p.\ 205]{adjt}. See also \cite{au}.
To be consistent with the notation
in Definition \ref{jondef},
it is convenient to add the two trivial terms at the outer edges.
In detail, under our convention, a sequence 
 of \emph{Pixley  terms} is a sequence
$t_0, \dots, t_{n}$  of ternary terms such that 
the following equations are satisfied: the equations \eqref{bas} from
Definition \ref{jondef}, as well as \eqref{m4}, for $h=0, \dots, n-1$.
 
A less standard example
is the notion of directed terms with alvin heads, 
as introduced in Definition \ref{dirgumm}(c).
In this case
the first two equations are chosen to be 
\eqref{m3} and \eqref{m1}, 
the last two equations are chosen to be 
\eqref{m3} and \eqref{m1}, in that order, 
and all the remaining middle equations, if any,
are \eqref{m2}.

The above examples are nondefective.
All the various kinds of Gumm terms 
introduced in Section \ref{gummdef}
are examples of defective 
mixed conditions.
\end{remark}

\begin{proposition} \labbel{propmix1}
 If $n \geq 1$,
then every variety with mixed J{\'o}nsson terms $t_0, \dots, t_n$
is congruence distributive, actually (at least) $2n{-}2$-distributive. 
 \end{proposition}  

Proposition \ref{propmix1} 
 can be  proved using the arguments from \cite{adjt}.
We shall prove a more refined result in
Theorem \ref{propmix2} 
 below. See Remark \ref{bgp}.

Of course, according to the form of the mixed condition,
a variety as in Proposition \ref{propmix1}
might be $m$-distributive, 
for some $m < 2n-2$.
Indeed, $n$-distributivity itself is a special
case of a mixed condition, hence it might  already happen 
that  $m=n$. On the other hand, Theorem \ref{dirthm}(ii)  
shows that, in general,  $ 2n-2$
 is the best possible bound  in Proposition \ref{propmix1}.

We shall now prove a more explicit version of 
Proposition \ref{propmix1}. More importantly for our purposes,
and a generalization of Lemma \ref{dayoptlem}, 
we are going to show that there are cases in which defective 
conditions imply congruence modularity.
In order to accomplish this, we
 need a way to describe each specific mixed condition.
While, in the case of the more usual examples, it appears natural to list
the various kinds of equations which are satisfied,  it turns
out that, in the general case of an arbitrary 
mixed condition, it is more convenient to deal with
the variables which are moved relative
to each single  term. See, in particular,
Remark \ref{lrrmk}(a).
We first need a formal notational remark
about how to treat the initial and final equations.

\begin{remark} \labbel{mixremult}
Let $\mathcal {V}$ be a variety 
with mixed J{\'o}nsson terms as in Definition \ref{mix}. 
It might happen that, for some $h< n$, $\mathcal V$ satisfies
two or more equations among \eqref{m1}, 
\eqref{m3}, \eqref{m2}, \eqref{m4}.
In our definition of a \emph{mixed condition}
it is convenient to
 require that, for each $h$, exactly one of the above equations
is selected. 
The ``edge cases'' $h=0$ and $h=n-1$ are an exception.
As we shall see,
in the present section  it will be convenient to deal only with the terms
$t_1, \dots, t_{n-1}$. 
Hence, for $h=0$,  we shall not distinguish between 
\eqref{m1} and \eqref{m2}, which both entail
$x=t_1(x,x,z)$, and we shall not distinguish
between   \eqref{m3} and \eqref{m4}, 
 both entailing
$x=t_1(x,z,z)$.
A symmetrical consideration applies to 
$t_{n-1}$.
 \end{remark}

\begin{definition} \labbel{lr}
Suppose that $n \geq 2$ and that  $l$,  $r$ are functions 
from $\{ 1, \dots, n-1 \}$ to the  set 
$\{ x, z \}$ of variables. Each such pair of functions 
\emph{determines} a mixed    
condition in the sense of Definition \ref{mix},
 modulo the convention from Remark \ref{mixremult}.
The equations to be satisfied are  \eqref{bas} from Definition \ref{jondef} and  
\begin{equation}\labbel{lreq}    
\begin{aligned} 
x &= t_{1}(x,l(1),z),
\\ 
t_{h}(x,r(h),z) &=
t_{h+1}(x,l(h+1),z), \quad \text{ for $ 1 \leq  h < n-1$,  and} 
\\
t_{n-1} & (x,r(n-1),z) =z.
 \end{aligned}   
   \end{equation}
\end{definition}   

\begin{remark} \labbel{lrr}   
Clearly, every pair of functions 
$l$ and  $r$ as above actually determines a
mixed condition in the sense of
Definition \ref{mix} and Remark \ref{mixremult}. 
In passing, we notice that, by Remark \ref{mixremult},
 here it is more practical
to write, say, 
$x = t_{1}(x,l(1),z)$
in place of  something like
$t_0(x, r(0), z) = t_{1}(x,l(1),z)$, since
$x=t_0(x, \hy, \hy)$,
no matter the second and the third arguments in the
range of $t_0$, hence there is no use in specifying
 some value for $r(0)$.
In other words, we do not need the outer
(trivial) terms $t_0$ and $t_n$ when we present the definition 
of a mixed condition as determined
by certain functions\footnote{On the other hand,
as we mentioned in Remark \ref{count},
it is slightly more convenient to maintain  $t_0$ and  $t_n$
in the usual definitions of, say, J{\'o}nsson and  alvin 
terms, when the definitions 
are expressed by specifying the identities 
\eqref{m1} or \eqref{m3}  to be satisfied.}
  $l$ and  $r$.    

Conversely, every mixed condition requires that, for each $h$, 
 one equation
of the following form 
\begin{equation*} 
t_{h}(x,v_h,z) =t_{h+1}(x,w_{h+1},z)
\end{equation*}     
is satisfied, where
each one of $v_h$ and  $w_{h+1}$  is either $x$ or  $z$.
Hence if we set  
$r(h)=v_h$ and 
$l(h+1) = w_{h+1}$, for all appropriate values of $h$,  
the given mixed condition is determined
by $l$ and  $r$.  
Thus we get that  every mixed condition is determined by 
some pair $l$ and  $r$.
 \end{remark}

In particular, the J{\'o}nsson (alvin) condition is obtained
by setting $l(h)=x$ and $r(h)=z$,  for $h$ odd (even),
and $l(h)=z$ and $r(h)=x$, for $h$ even (odd).   
The directed J{\'o}nsson condition is obtained by 
putting $l(h)=x$ and $r(h)=z$, for every $h$.
Letting $l(h)=z$ and $r(h)=x$, for every $h$,
we obtain the generalized Pixley condition in the sense
of \cite{adjt}, as recalled in Remark \ref{mixrem}.
 If $n \geq 4$, $l(1) = l(n-1) = z$,
 $r(1) = r(n-1) = x$, and,
in all the other cases, $l(h)=x$ and $r(h)=z$, then
 we get directed terms with alvin heads, as in 
Definition \ref{dirgumm}(c), of course, suitably
relabeling the terms $p$ and $q$.   
See Remark \ref{lrrmk}(b) below  for a diagram 
(inspired by \cite{KV})
representing the situation.

Gumm and   defective Gumm
terms are defective cases of the alvin condition.
Directed Gumm terms are a defective
case of the mixed condition determined by the
positions     $l(n-1) = z$,
 $r(n-1) = x$, and $l(h)=x$ and $r(h)=z$,
for all the other $h$'s. 
 Two-headed directed Gumm terms
are   defective cases of what we have called 
directed terms with alvin heads.

The identities 
\eqref{agob} and 
\eqref{abog} in the statement of Theorem \ref{buhbis}
correspond to defective mixed conditions, too, via Remark \ref{conidabog}.
The conditions are defective versions of, respectively,
the alvin and the J{\'o}nsson conditions;
in this case the terms are defective at all odd (even) indices.
In the statement of Theorem \ref{buhbis}
$n$ is assumed to be even; 
if $n$ is odd, we 
set $\ell= \frac{n-1}{2} $
and consider variations of 
 \eqref{agob} and 
\eqref{abog}
in which an $ \alpha \gamma $ factor
is added or deleted
  on the right. 
See Definition \ref{deflev} below  for exact details.
The above conditions shall be called 
the \emph{switch} and  
the \emph{J-switch} condition.

 The next remark  shows that the  description
given by Definition \ref{lr} 
can be somewhat simplified.

\begin{remark} \labbel{lrrmk}
(a) The $l$-$r$-convention introduced in Definition \ref{lr} 
is particularly useful in order
to detect redundant conditions.
Indeed, if, under the assumptions
in Definition \ref{lr},  we have $l(h)=r(h)$, for some $h$, then
\begin{equation*}
t_{h-1}(x,r(h-1),z) =t_{h}(x,l(h),z) = t_{h}(x,r(h),z) =
t_{h+1}(x,l(h+1),z),
  \end{equation*}    
 by \eqref{lreq},   hence,
\begin{equation}\labbel{caaa}
    t_{h-1}(x,r(h-1),z)=t_{h+1}(x,l(h+1),z),
   \end{equation}
thus in this case  the term $t_h$ is redundant and can be discarded,
getting a shorter sequence of mixed J{\'o}nsson terms. 
Indeed, given terms
$t_0, \dots, t_{h-1},     t_{h+1},    \dots, t_n$
satisfying \eqref{caaa} and all the other appropriate equations, we get
terms  $t_0, \dots, t_{h-1}, t_h,  \allowbreak   t_{h+1},  \dots, t_n$
satisfying the equations \eqref{lreq} by setting $t_h(x,   y,z)=
 t_{h-1}(x,r(h-1),z) $.
Notice that, with this position,
$t_h$ does not depend on its second place.

In particular, it is no loss of generality to assume that
$l(h) \neq r(h)$, for every $h$. Under this assumption, a mixed condition
  is determined either by $l$ alone, or by  $r$ alone, since
 $l(h) $ and  $ r(h)$ may assume only two values, hence, if   $l(h) \neq r(h)$,
the value of $l(h) $ determines the value of $ r(h)$ and conversely.

(b) Following \cite{KV},  
still assuming that
$l(h) \neq r(h)$, for every $h$,
we can represent a  mixed condition by
a  path with directed edges.
The path $P$ goes from $0$ to $n-1$ 
and touches  all the intermediate natural numbers, in their order.
If $1 \leq i < n$, $l (i)= x$ and $ r(i) = z$,
then the edge in $P$ connecting 
$i-1$ and $i$ is directed from 
 $i-1$ to $i$ (\emph{forward directed}).
On the other hand, if 
$l (i)= z$ and $ r(i) = x$,
then the edge in $P$ connecting 
$i-1$ and $i$ is directed from 
 $i$ to $i-1$ (\emph{backward directed}).
With the above conventions,
the functions $l$ and  $r$
determine  the same 
Maltsev condition 
$M(P)$ as defined 
in \cite[Subsection 3.2]{KV},
provided all the edges  of $P$ 
are considered as solid,  modulo the relabeling
of the variables $y$ in \cite{KV} to  $z$ 
here. Notice that there is an extra term
in \cite[Subsection 3.2]{KV}, namely, we should have  taken
$n+1$ in place of $n$  here to get an exact correspondence.

One can describe defective mixed conditions in a similar way,
considering $P$ as an edge-labeled path,
where edges are labeled as either solid or dashed.
 If the condition is defective at some place $i$,
we consider the edge
connecting $i-1$ and $i$ as   
dashed, otherwise, as above, the edge is solid.
In this more general case, too, the (possibly defective)
mixed condition
turns out to be $M(P)$, as defined 
in \cite{KV}.

In the following table we represent
the conditions we study by means of 
paths. 
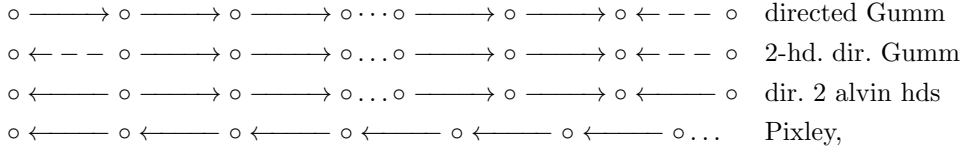
\begin{flalign*}    
&\circ  \xrightarrow {\hspace {25pt}} \circ
\xleftarrow {\hspace {25pt}} \circ
\xrightarrow {\hspace {25pt}} \circ
\xleftarrow {\hspace {25pt}} \circ
\xrightarrow {\hspace {25pt}} \circ 
\xleftarrow {\hspace {25pt}} \circ \dots
&& \text{J{\'o}nsson} 
\\
&\circ  \xrightarrow {\hspace {25pt}} \circ
\xdashleftarrow {\hspace {25pt}} \circ
\xrightarrow {\hspace {25pt}} \circ
\xdashleftarrow {\hspace {25pt}} \circ
\xrightarrow {\hspace {25pt}} \circ 
\xdashleftarrow {\hspace {25pt}} \circ \dots
&& \text{J-switch} 
\\
&\circ \xleftarrow {\hspace {25pt}} \circ
\xrightarrow {\hspace {25pt}} \circ
\xleftarrow {\hspace {25pt}} \circ
\xrightarrow {\hspace {25pt}} \circ 
\xleftarrow {\hspace {25pt}} \circ
 \xrightarrow {\hspace {25pt}} \circ \dots
&& \text{alvin}
\\
&\circ \xdashleftarrow {\hspace {25pt}} \circ
\xrightarrow {\hspace {25pt}} \circ
\xdashleftarrow {\hspace {25pt}} \circ
\xrightarrow {\hspace {25pt}} \circ 
\xdashleftarrow {\hspace {25pt}} \circ
 \xrightarrow {\hspace {25pt}} \circ \dots
&& \text{switch}
\\ 
&\circ \xdashleftarrow {\hspace {25pt}} \circ
\xrightarrow {\hspace {25pt}} \circ
\xleftarrow {\hspace {25pt}} \circ
\xrightarrow {\hspace {25pt}} \circ 
\xleftarrow {\hspace {25pt}} \circ
 \xrightarrow {\hspace {25pt}} \circ \dots
&& \text{Gumm} 
\\
&\circ \xdashleftarrow {\hspace {25pt}} \circ
\xrightarrow {\hspace {25pt}} \circ
\xleftarrow {\hspace {25pt}} \circ  
\xrightarrow {\hspace {25pt}} \circ \dotsc \circ 
\xrightarrow {\hspace {25pt}} \circ 
 \xdashleftarrow {\hspace {25pt}} \circ 
&& 
\text{defective Gumm}
\\
&\circ  \xrightarrow {\hspace {25pt}} \circ
\xrightarrow {\hspace {25pt}} \circ
\xrightarrow {\hspace {25pt}} \circ \dotsc \circ 
 \xrightarrow {\hspace {25pt}} \circ
\xrightarrow {\hspace {25pt}} \circ
\xrightarrow {\hspace {25pt}} \circ 
&& \text{directed J{\'o}nsson} 
\\
&\circ  \xrightarrow {\hspace {25pt}} \circ
\xrightarrow {\hspace {25pt}} \circ
\xrightarrow {\hspace {25pt}} \circ \dots \circ
 \xrightarrow {\hspace {25pt}} \circ 
\xrightarrow {\hspace {25pt}} \circ
\xdashleftarrow {\hspace {25pt}} \circ 
&& \text{directed Gumm} 
\\
&\circ  \xleftarrow {\hspace {25pt}} \circ
\xrightarrow {\hspace {25pt}} \circ
\xrightarrow {\hspace {25pt}} \circ
\dotsc \circ 
\xrightarrow {\hspace {25pt}} \circ
\xrightarrow {\hspace {25pt}} \circ
\xleftarrow {\hspace {25pt}}  \circ
&& \text{dir.\,  alvin heads} 
\\
&\circ  \xdashleftarrow {\hspace {25pt}} \circ
\xrightarrow {\hspace {25pt}} \circ
\xrightarrow {\hspace {25pt}} \circ
\dotsc \circ  
\xrightarrow {\hspace {25pt}} \circ
\xrightarrow {\hspace {25pt}} \circ
\xdashleftarrow {\hspace {25pt}}  \circ 
&& \text{$2$-hd.\ dir.\ Gumm} 
\\
&\circ \xleftarrow {\hspace {25pt}} \circ
\xleftarrow {\hspace {25pt}} \circ
\xleftarrow {\hspace {25pt}} \circ
\xleftarrow {\hspace {25pt}} \circ
\xleftarrow {\hspace {25pt}} \circ
\xleftarrow {\hspace {25pt}} \circ \dots
&& \text{Pixley,} 
\end{flalign*}
where in the line representing 
defective Gumm terms we assume that $n$ is even. 

We refer to \cite{KV}
for further details, diagrams, results and variations.     
 
(c) In (b) above it would probably be more natural, though
notationally clumsier, to assume that the vertices
of $P$ are labeled as $ \frac{1}{2} $, $1 + \frac{1}{2} $, \dots, $n - \frac{1}{2}$,
namely, the indexing of the vertices is shifted
by $\frac{1}{2}$. 
The above conditions turn out then to be more symmetric, for example,
with the above shifted indices, if    
$l (i)= x$ and $ r(i) = z$,
then the edge in $P$ connecting 
$i- \frac{1}{2}$ and $i+\frac{1}{2}$ is directed from 
$i- \frac{1}{2}$ to $i+\frac{1}{2}$; it is dashed if the condition is defective
at $i$, otherwise, it is solid. 
\end{remark}

\subsection*{Relation identities} 
Now we turn our attention to relation identities
which are consequences of mixed conditions.
The study of relation identities seems to be interesting for itself,
see  \cite{Fio,G,HM,CV,contol,cm2,ia,LGA,B,ntcm,T}
 and Remarks \ref{rmkkb} and \ref{equiv}. 
However, in this subsection we deal with 
relation identities only because they provide a quite easy and,
at least in our opinion, natural way to congruence identities.

We let $R$, $S$ and $T$ denote binary reflexive and admissible relations 
on some algebra.
We let $R ^\smallsmile $ denote  the \emph{converse} of $R$
and
 $ \overline{S \cup T} $ 
denote the smallest  admissible relation  
containing both $S$ and $T$.
In the following formulae, when
dealing with relation identities, juxtaposition denotes intersection
of binary relations.
Recall that in this context $0$ denotes the identical relation.  

\begin{theorem} \labbel{propmix2}
 Suppose that $n \geq 2$. 

  \begin{enumerate}[(i)]
    \item   
Every variety $\mathcal {V}$ with mixed J{\'o}nsson terms $t_0, \dots, t_n$
in the sense of Definition \ref{mix}
 satisfies some relation identity of the form
\begin{equation} \labbel{bm}  
\alpha (S \circ T) \subseteq  B_1 \circ B_2 \circ \dots \circ B_{n-1},
  \end{equation}    
where $S$ and $T$ are reflexive and admissible relations 
and each $B_h$ 
is either $ \alpha S  \circ \alpha T $ or 
$    \alpha  T ^\smallsmile \circ \alpha S ^\smallsmile$.
\item
In detail, if $\mathcal {V}$ satisfies some mixed condition determined
by $l$ and  $r$, as in Definition \ref{lr}, then  
the $B_h$'s in identity \eqref{bm}  can 
be taken as 
\begin{align*}
B_h&= \alpha S  \circ \alpha T,  
&& \text{ if $l(h) = x$ and $r(h) = z$}, 
\\
B_h&= \alpha  T ^\smallsmile \circ \alpha S ^\smallsmile, 
  && \text{ if $l(h) = z$ and $r(h) = x$}, 
\\
B_h&= 0 , 
  && \text{ if $l(h) = r(h)$}.
 \end{align*}    
\item
  \begin{enumerate}  
  \item    
 If, in addition, 
$l(1)=z$,
that is, $\mathcal {V}$ satisfies
$x=t_1(x,z,z)$,  
then $B_1$ in identity 
\eqref{bm} can be replaced by  
$ \alpha ( \overline{S ^\smallsmile \cup T})$,
and this holds also in case the sequence of terms 
is defective at $1$.  
  
\item
Symmetrically, if
$r(n-1)=x$, that is, 
$t_{n-1}(x,x,z)=z$ holds, 
then $B_{n-1}$ in identity 
\eqref{bm} can be replaced by  
$ \alpha ( \overline{S  \cup T ^\smallsmile}) $,
and this applies also  if the sequence of terms 
is defective at $n-1$.  

\item
If $n \geq 3$ and both the  
additional assumptions in (a) and (b) hold,
we can perform both replacements,
also  if the sequence  
is defective both at $1$ and at $n-1$.
  \end{enumerate} 
\end{enumerate}
\end{theorem}

\begin{remark} \labbel{bgp}   
Before giving the proof of Theorem \ref{propmix2},
we observe that Proposition \ref{propmix1}
is a consequence of 
\ref{propmix2}. 
If $n=1$ in Proposition \ref{propmix1},
then we are in a trivial variety 
and the conclusion vacuously holds.
If $n \geq 2$, take 
$S= \beta $ and $T= \gamma $
congruences in \ref{propmix2}(i).
From \eqref{bm} we get 
$\alpha ( \beta \circ \gamma ) \subseteq 
\alpha \delta _1 \circ \dots \circ \alpha \delta _{2n-2}$,
where each $\delta_i$ is either $\beta$ or $\gamma$. 
If $\delta_1 = \beta $, we are done, by   Remark \ref{conid}.
If $\delta_1 = \gamma  $ and, for some $i$,
$\delta_i = \delta _{i+1} $, the two factors 
$ \alpha \delta_i $ and $  \alpha  \delta _{i+1} $
absorb into one, hence we get 
(at least) $2n{-}3$-alvin,
hence $2n{-}2$-distributivity. 
In the remaining case we always have
$B_h= \alpha  \gamma  \circ \alpha \beta $, 
hence $l(h) = z$ and $r(h) = x$, for every $h$, by \ref{propmix2}(ii).
This means that the terms $t_1, \dots, t_{n-1}$
are (the nontrivial terms in a sequence of) Pixley terms;
see Remark \ref{mixrem}. A variety with   Pixley terms
$t_0, \dots, t_{n}$ is congruence $n$-permutable and distributive
\cite{adjt,au}, hence $n$-distributive, a fortiori,
$2n{-}2$-distributive, since $n \geq 2$.    
 \end{remark}

\begin{proof}[Proof of Theorem \ref{propmix2}]
 In many particular instances the proof
of \ref{propmix2} 
 is standard,
 using or modifying the arguments from \cite{JD}.
See, e.~g., 
\cite[Lemma 4.3]{ntcm},  \cite{G1,G,LTT,jds,ia,B,T} 
and the proofs of  \ref{dayoptlem}, \ref{dayopt}(i), \ref{daystrong},
\ref{dir}(ii), \ref{thm2hg}(i)      here
for similar examples.

To prove the theorem  in general, 
first notice that (i) is a special case of (ii), since
 if $B=0$,
then trivially $B \subseteq \alpha S  \circ \alpha T$.
In any case, 
 by Remark \ref{lrrmk}(a),
it is no loss of generality to assume that the case  $l(h) = r(h)$
 never occurs.

So let us prove (ii). Suppose that $\mathcal {V}$
 has mixed J{\'o}nsson terms $t_0, \dots, t_n$,
with equations determined by  $l$ and $r$. 
If $\mathbf A \in \mathcal V$,
$a,c \in A$ and 
$(a,c) \in \alpha ( S \circ T )$,
then $a \mathrel { \alpha } c $
and there is some $b \in A$ 
such that $ a \mathrel { S } b \mathrel { T } c $.
For  $h=1, \dots, n-1$, let 
$l^*(h) = a$ if  $l(h) = x$ and
$l^*(h) = c$ if  $l(h) = z$ and define
$r^*$ similarly. 
Consider the elements
\begin{equation*}
\begin{aligned}
e_0 &=  a= t_{1}(a,l^*(1),c),
\\
e_h&=t_h(a,r^*(h),c)=t_{h+1}(a,l^*(h+1),c),
\quad
\text{for 
$1 \leq h <  n-1$,}
\\
 e_{n-1} &=   t_{n-1}(a,r^*(n-1),c)= c,
 \end{aligned} 
  \end{equation*}
 where the equalities follow from the
equations \eqref{lreq}.

If $1 \leq h \leq  n-1$, $l(h) = x$ and $r(h) = z$,
then 
\begin{equation*} 
e_{h-1}{ \hspace{1pt} = \hspace {1pt} }t_{h}(a,l^*(h),c) 
{ \hspace{1pt} = \hspace {1pt} } t_{h}(a,a,c) \mathrel {S } 
\\
 t_{h}(a,b,c) \mathrel { T } 
t_{h}(a,c,c)  { \hspace{1pt} = \hspace {1pt} }t_{h}(a,r^*(h),c) { \hspace{1pt} = \hspace {1pt} } e_{h}.
\end{equation*}  
Moreover, $e_{h-1}= t_{h}(a,a,c) \mathrel { \alpha  } t_{h}(a,a,a)=a$,
and similarly $t_{h}(a,b,c) \mathrel \alpha  t_{h}(a,b,a) =a$
and   $  e_{h} = t_{h}(a,c,c)  \mathrel \alpha t_{h}(a,c,a)= a $,
thus $e_{h-1} \mathrel \alpha t_{h}(a,b,c) \mathrel \alpha e_{h}  $,
hence    $e_{h-1} \mathrel {\alpha S}
t_{h}(a,b,c) \mathrel {\alpha T} e_{h}  $,
from which
$ e_{h-1} \mathrel { B_{h}  } e_{h} $ follows.

Similarly, if $1 \leq h \leq  n-1$, $l(h) = z$ and $r(h) =x$,
then 
\begin{equation*} 
e_{h{-}1}{ \hspace{1pt} = \hspace {1pt} }t_{h}(a,l^*(h),c) 
{ \hspace{1pt} = \hspace {1pt} } t_{h}(a,c,c)
{  \hspace{1pt} \mathrel {T ^\smallsmile  } \hspace{1pt}   }
\\
 t_{h}(a,b,c) {  \hspace{1pt} \mathrel { S ^\smallsmile  } \hspace{1pt}  } 
t_{h}(a,a,c)  { \hspace{1pt} =
 \hspace {1pt} }t_{h}(a,r^*(h),c) { \hspace{1pt} = \hspace {1pt} } e_{h}.
\end{equation*}  
As in the previous paragraph,
 $e_{h-1} \mathrel \alpha t_{h}(a,b,c) \mathrel \alpha e_{h}  $,
hence    $e_{h-1} \mathrel {\alpha T ^\smallsmile }   
t_{h}(a,b,c)   \allowbreak  \mathrel {\alpha S ^\smallsmile } e_{h}  $,
from which we get
$ e_{h-1} \mathrel { B_{h}  } e_{h} $
in this case, as well.

Finally, if
$l(h) = r(h)$, then $l^*(h) = r^*(h)$, hence 
$e_{h-1}=t_{h}(a,l^*(h),c) 
=t_{h}(a,r^*(h),c) = e_{h}$
and we can take $B_{h}=0$.
 
In conclusion,  
$a=e_0 \mathrel {B_1} e_1 \dots  e_{n-2} 
\mathrel {B_{n-1}} e_{n-1}= c  $, 
hence
$(a,c) \in B_1 \circ B_2 \circ \dots \circ B_{n-1}$
and (ii) is proved.

(iii)(a) If $r(1)=z$, then by (ii) we can take $B=0$  
and we are done. Notice that no special equation
involving $t_h$   is needed
in the proof of the case $l(h)=r(h)$. 

Otherwise, $r(1)=x$. Suppose that
$a \mathrel { \alpha } c $
and  $ a \mathrel { S } b \mathrel { T } c $,
as in the proof of (ii)
above. Then, under the additional assumption, we have 
\begin{equation} \labbel{def}  
e_0=a=t_1(a,b,b) \mathrel { \overline{S ^\smallsmile \cup T }}
t_1(a,a,c)  =t_{1}(a,r^*(1),c) = e_{1}
 \end{equation}    
and this holds also when $t_1$ is defective, 
since the equation $x=t_1(x,y,x)$
has not been used in the proof of \eqref{def}.   
Furthermore, $e_0=a= t_1(a,a,a) \mathrel \alpha  t_1(a,a,c) = e_1$,
 so we do not need $x=t_1(x,y,x)$, either,
in order to prove the $\alpha$-relation.

(b) is proved in a symmetrical way.

(c) If $n \geq 3$, the two arguments 
at $1$ and $n-1$ do not interfere,
hence we can perform both replacements.

Notice that the argument applies when $n=3$,
in which case \emph{all} the terms 
can be taken as  defective!  
We thus get that a $3$-permutable variety 
satisfies  
$ \alpha (S \circ T) \subseteq 
\alpha ( \overline{S ^\smallsmile \cup T}) \circ 
\alpha ( \overline{S \cup T ^\smallsmile })$. 
This is a strong way to say that 
$3$-permutable varieties are congruence modular:
just take $S= \beta $ and $T= \alpha \gamma \circ \beta $
in the above identity. 
Indeed, we have got a direct proof that 
$3$-permutable varieties are $3$-reversed-modular.   
 \end{proof} 

\begin{remark} \labbel{aha}
If $n \geq 4$ and $\mathcal V$ is $n$-directed 
with alvin heads, then $\mathcal V$ is $2n{-}4$-alvin.   

This fact follows immediately from  \ref{propmix2}(ii)  taking $S= \beta $
and $T= \gamma $, thus 
\begin{align*}
\alpha ( \beta \circ \gamma ) 
&\subseteq 
(\alpha \gamma \circ \alpha  \beta) \circ
(\alpha \beta  \circ \alpha  \gamma ) \circ
(\alpha \beta  \circ \alpha  \gamma ) \circ
\dots \circ 
(\alpha  \beta  \circ \alpha  \gamma ) \circ
(\alpha \gamma \circ \alpha  \beta)
\\
& = 
\alpha \gamma \circ \alpha \beta \circ  {\stackrel{2n-4}{\dots}} \circ  \alpha \beta.   
  \end{align*}      
 \end{remark}

\begin{corollary} \labbel{propmixmod}
 Suppose that $n \geq 2$ and that the
variety $\mathcal {V}$ has mixed J{\'o}nsson terms $t_0, \dots, t_n$
in the sense of Definition \ref{mix}.
Then the following hold.
  \begin{enumerate}[(i)]
    \item   
$\mathcal {V}$ 
is $2n{-}1$-modular.
\item 
  \begin{enumerate}[(a)]    
\item 
If 
$l(1)=z$,
that is, $\mathcal {V}$ satisfies
$x=t_1(x,z,z)$,
then $\mathcal {V}$ is
$2n{-}2$-modular.
The result holds also if the mixed terms are defective at $1$.

\item
If $r(n-1)=x$, that is, 
$t_{n-1}(x,x,z)=z$ holds,
then $\mathcal {V}$ is
$2n{-}2$-modular.
The result holds also if the mixed terms are defective at $n-1$.

\item
If $n \geq 3$ and 
both the assumptions in (a) and (b)
above hold, then $\mathcal {V}$ is
$2n{-}3$-reversed-modular.
The result holds also in the case when
 the mixed terms are defective at $1$ and at $n-1$.
  \end{enumerate} 
\end{enumerate} 
 \end{corollary}

 \begin{proof} 
(i) By taking 
$S= \beta $ and $T= \alpha \gamma \circ \beta $
in identity \eqref{bm},
we get that both
$ \alpha S  \circ \alpha T $ and 
$    \alpha  T ^\smallsmile \circ \alpha S ^\smallsmile$
are equal to 
$ \alpha \beta \circ \alpha \gamma \circ \alpha \beta $,
since $ \alpha T = \alpha (\alpha \gamma \circ \beta)=
\alpha \gamma \circ \alpha \beta $,
because $\alpha$ is transitive.
Hence in \eqref{bm} 
we have $B_h=  \alpha \beta \circ \alpha \gamma \circ \alpha \beta$,
for every  $h$.
When we compute 
$B_1 \circ B_2 \circ \dots \circ B_{n-1}$
we have $n-2$ adjacent pairs 
of occurrences of 
$ \alpha \beta $, and each pair absorbs into one, since
$\alpha \beta $ is a congruence, hence transitive.  
Thus 
$B_1 \circ B_2 \circ \dots \circ B_{n-1} =
\alpha \beta \circ \alpha \gamma \circ {\stackrel{2n-1}{\dots}}
 \circ \alpha \beta   $,
hence
identity \eqref{bm} gives $2n{-}1$-modularity,
by Remark \ref{conidmod}.  

(ii)(a) Choose  $S$ and $T$ as above.
By Theorem  \ref{propmix2}(iii)(a)
we can take $B_1=  \alpha (\overline{ S ^\smallsmile \cup T})
  = \alpha T = \alpha \gamma \circ \alpha \beta $,
hence we can save an occurrence of $\alpha \beta $
at the beginning in the computation of 
$B_1 \circ B_2 \circ \dots \circ B_{n-1}$.
That is, we have $2n{-}2$-reversed-modularity,
which is equivalent to $2n{-}2$-modularity, by 
Proposition \ref{modeq}, since $2n-2$ is even.

(b) The symmetrical argument provides $2n{-}2$-modularity
directly.

(c)  If both the assumptions in 
(a) and (b) hold,
 we can save the occurrences
of $\alpha \beta $ both at the beginning and  at the end.   
Thus we get  $2n{-}3$-reversed-modularity.
In passing, notice that either Theorem \ref{buh}(ii)
or  Theorem \ref{thm2hg}(ii)
show that we \emph{cannot} get  $2n{-}3$-modularity. 

Notice that the argument in (c) does not work
if $n=2$, since in that case we cannot perform at the same time
both the replacements
allowed by Theorem \ref{propmix2}(iii)(c).
Indeed, in the present terminology, any  
arithmetical variety $\mathcal {V}$ has mixed J{\'o}nsson terms
$t_0, t_1, t_2$ with $l(1)=z$ and 
$r(1)=r(2-1)=x$, but if $n=2$, then 
$2n-3=1$, while   only trivial
varieties are  $1$-reversed-modular. 
\end{proof}

\subsection*{Additional  remarks on mixed terms} 
\begin{remark} \labbel{rcm}
Corollary \ref{propmixmod}(i) immediately implies Day's result that 
every $n$-\brfrt distributive variety is $2n{-}1$-modular.

Item (ii)(a) implies that every 
 $n$-alvin and  every $n$-Gumm variety
is $2n{-}2$-modular. This holds for every $n$
and,  as we mentioned,  is implicit
in \cite{LTT}. If $n$ is odd, we get 
that every $n$-distributive variety is
 $2n{-}2$-modular, by  Remark \ref{conidcomm}(a).
Otherwise, apply item (ii)(b) directly. 
In particular, we get another proof for Lemma \ref{dayoptlem}(b).

Item (ii)(c) implies that if $n \geq 4$ and $n$ is even, then     
every $n$-alvin (actually, every  defective $n$-Gumm,
in particular, every $n$-Gumm) variety
 is $2n{-}3$-reversed-modular.
This gives another proof of Lemma \ref{dayoptlem}(a). 
Item (ii)(c) also implies 
that every
 variety with 
a   sequence 
$p, t_2, \dots, t_{n-2}, q$ of two-headed directed Gumm terms,
as  introduced in Definition \ref{dirgumm}(b), is
$2n{-}3$-reversed-modular.

It is important to notice that, in all the above situations, 
in order get congruence modularity it is fundamental that the conditions
are defective at the outer places $1$ or $n-1$. 
In fact, mixed conditions defective at some ``internal''
place usually do not  imply congruence modularity.
See Remark \ref{polin}(a). 
 \end{remark}

\begin{remark} \labbel{spend}
\emph{What's so special at the ends?} 
We have seen that  special situations 
occur at the outer ``edges'' $t_1$ and  $t_{n-1}$.
Namely, if the equation 
$x=t_1(x,z,z)$ holds, 
then we get congruence modularity
even without assuming 
$x=t_1(x,y,x)$. Actually, we get 
$2n{-}2$-modularity rather 
than  $2n{-}1$-modularity,
thus we have the rather remarkable result that, in this special case,
we get a stronger conclusion using a weaker hypothesis!
As we hinted in Remark \ref{galt}(c),
this can be seen as a consequence of the fact that   
alvin-like conditions share some aspects in common with congruence
permutability.

It is essential to assume that $x=t_1(x,z,z)$.
If $x=t_1(x,x,z)$, instead, and $t_1$
is defective, then we get a trivial condition, by
Remark \ref{gummrmk}(c). 

While our arguments here depend crucially on the form
$x=t_1(x,z,z)$ of the first nontrivial equation, it should be mentioned
that  there are \emph{always} some special kinds of shortcuts 
which can be taken
``at the outer edges''.
 See \cite[Remark 17]{contol}.
\end{remark}

\begin{remark} \labbel{useofbm}
(a)
Of course, for each specific application of identity
 \eqref{bm} in Theorem \ref{propmix2},
 we could explicitly find out appropriate terms
which give a proof of the consequences 
under consideration. Compare the
 proofs of the classical 
Day's Theorem \cite[p.\ 172]{D},
reported here in Corollary \ref{dayopt}(i),
of \cite[Theorem 1(3) $\rightarrow $ (1)]{LTT},
of Lemma \ref{dayoptlem} and of Theorems \ref{dir}(ii),
\ref{thm2hg}(i).

 However, there are various reasons
suggesting that  identities like  \eqref{bm} are
particularly interesting and useful.

First, the original proof \cite{JD} that 
J{\'o}nsson terms imply congruence distributivity,
or, equivalently, that, within a variety,  the identities 
displayed in Remark \ref{conid} imply congruence distributivity,
essentially uses (the J{\'o}nsson-terms particular version of) 
identity \eqref{bm}. The point is that, in principle, 
in order to prove congruence distributivity,
it is not enough to find bounds for $\alpha( \beta \circ \gamma )$,
one needs bounds for   
$\alpha( \beta \circ \gamma  \circ {\stackrel{k}{\dots}}    )$
for arbitrarily large $k$. See \cite{jds} for further
elaborations  on this aspect.

Second, identity 
\eqref{bm} provides a uniform way 
to prove some quite disparate facts.
While, of course, once we have proved 
congruence distributivity, we surely have 
congruence modularity as a consequence, 
on the other hand, identity \eqref{bm} is useful in establishing 
the exact distributivity or modularity levels.
See  Remarks \ref{bgp},  \ref{rcm},
Corollary \ref{propmixmod}  and Theorem \ref{sumup}. 
Compare also some parallel results
in \cite{G1,G,LTT,jds,ia,ntcm,T}.

(b) As another example, 
if we argue in terms of identity \eqref{bm},
we generally get a clear explanation for the difference
in the possible distributivity levels of varieties 
with the same number of J{\'o}nsson and of directed J{\'o}nsson terms.
See Theorem \ref{dirthm}(ii)
or the table in Theorem \ref{sumup} below.
 
Indeed,
it is almost immediately clear from 
\eqref{bm} that the existence of 
J{\'o}nsson terms 
$t_0, \dots t_n$ 
implies the corresponding displayed
identity in Remark \ref{conid}.
Just take $S= \beta $
and $T= \gamma $;
then, due to Theorem  \ref{propmix2}(ii),
we get $n-2$ 
adjacent pairs of
congruences, either $\alpha \beta $
or $ \alpha \gamma $, so 
these congruences mutually absorb and 
we end up with a total
of $n$ factors.  
On the other hand, if we deal with directed J{\'o}nsson 
terms, then Theorem \ref{propmix2}(ii)
always gives $B_h= \alpha \beta \circ \alpha \gamma $,
so we get no adjacent pair of identical congruences and we are left
with  $2n-2$ factors.
In fact, in general, we can do no better,
as shown in Theorem \ref{dirthm}(ii).  

Of course, a proof is needed, since the above informal argument using 
identity \eqref{bm} is not a proof  and in principle we might find out 
different tricks leading to a better result.
In fact, this is the case, for example, 
for varieties with Pixley terms (cf. \cite[p.\ 205]{adjt} and Remark \ref{mixrem})
 which are congruence distributive and $n$-permutable,
hence $n$-distributive \cite{au}.
In this case we always have $B_h= \alpha \gamma  \circ \alpha \beta  $,
hence no pair of congruences absorb in \eqref{bm},
but we can get  $n$-distributivity  nesting terms.  
Apart from such exceptions, the argument based on 
identity \eqref{bm} seems generally a quite clear guide to 
intuition.

(c)
Identity \eqref{bm} appears to be generally a good guide to 
intuition also when dealing
with congruence modularity.
In this case, the best bounds for modularity levels 
of  varieties with the same number of 
J{\'o}nsson, alvin  and directed J{\'o}nsson terms terms are 
essentially always the same, at most
differing by $1$ or $2$, depending on  the form
of the identities ``at the outer edges''
$t_1$ and  $t_{n-1}$.   
The intuitive reason
for the above ``stationarity'' of the modularity levels 
 is that whenever we try to have
the left side 
$ \alpha (S \circ T)$ of \eqref{bm}
equal to 
$\alpha( \beta \circ \alpha \gamma \circ \beta ) $,
we always end up with
each $B_h$ having the form
$ \alpha \beta \circ \alpha \gamma \circ \alpha \beta $,
except possibly for the outer edges 
$B_1$ and   $B_{n-1}$.   
Hence in this case there is no sensible
difference between  the cases of, say,
J{\'o}nsson and directed J{\'o}nsson terms.
Again, the above intuitive argument
generally leads to the correct results, as we have showed in
Corollary \ref{dayopt}, Theorems \ref{dir}, \ref{thm2hg}
and  Proposition \ref{gummopt}.   
Notice that here we are dealing with the minimal number of mixed terms, not
with the  distributivity level, namely,
the minimal number of J{\'o}nsson terms.
In fact, while $n$-distributivity implies
$2n{-}1$-modularity, and generally this result cannot be improved, 
there are varieties in which the  distributivity and the modularity
levels differ only by $1$. Compare the results about 
the varieties $\mathcal V_n^c$, $\mathcal V_n^d$ and $\mathcal V_n^f$ 
 in  Theorem \ref{sumup} below.  
 \end{remark}

The following generalization of Theorem  \ref{propmix2} 
is proved in the same way. The generalization is used only
marginally 
in this paper.

\begin{proposition} \labbel{propmixi}
 If $n \geq 2 $, 
$  i \geq 1$ and $\mathcal {V}$ has mixed J{\'o}nsson terms $t_0, \dots, t_n$
satisfying a condition  determined
by $l$ and  $r$, then  
$\mathcal {V}$ satisfies
\begin{equation} \labbel{bmi}  
\alpha (S_0 \circ S_1 \circ \dots \circ S_i)
 \subseteq  B_1 \circ B_2 \circ \dots \circ B_{n-1},
  \end{equation}
where
\begin{align*}
B_h&= \alpha S_0  \circ \alpha S_1  \circ \dots \circ \alpha S_i, 
&& \text{ if $l(h) = x$ and $r(h) = z$}, 
\\
B_h&= \alpha  S_i ^\smallsmile \circ  \alpha S_{i-1} ^\smallsmile 
\circ \dots \circ  \alpha S_0 ^\smallsmile, 
  && \text{ if $l(h) = z$ and $r(h) = x$}, 
\\
B_h&= 0 , 
  && \text{ if $l(h) = r(h)$}.
 \end{align*}    
\end{proposition}  

Notice that item (iii)
from Theorem  \ref{propmix2}, as it stands,
cannot be immediately generalized 
 in the context of Proposition \ref{propmixi}.
In this connection, see however
 Lemma 4.3, Propositions 4.4 and 4.10
in \cite{ntcm}
and  Section 2 in \cite{jds}.

\begin{remark} \labbel{mixday}
It is also possible  to introduce  a notion of mixed Day terms.
Let a \emph{modular quadruplet} be anyone of the following
quadruplets of variables:
\begin{align*}     
(x,x,z,z) && (x,y,y,z) && (x,x,x,z) && (x,z,z,z).
  \end{align*}    
A sequence $u_0, \dots, u_m$ 
of $4$-ary terms 
is a sequence of \emph{mixed Day terms}
if the equations \eqref{d0}, \eqref{d1} and \eqref{d3}
from Definition \ref{daydef} are satisfied and, moreover, for each 
$k$ with $0 \leq k < m$, at least one of the following equations
is satisfied:
\begin{equation} \labbel{dmx}   
u_k(x_1,x_2,x_3,x_4) = u_{k+1}(y_1,y_2,y_3,y_4),
\end{equation}        
where both $(x_1,x_2,x_3,x_4)$ and $(y_1,y_2,y_3,y_4)$
are modular quadruplets. 

Of course, it is redundant to include
$ (x,y,y,z) $ in the set of modular quadruplets, since 
we can take $y=x$ or  $y=z$, still getting an equation
which is satisfied. However, it is convenient to maintain
$ (x,y,y,z) $ in order to have Day's conditions as a special case.
Moreover, 
if  some equation in \eqref{dmx} involves  $ (x,y,y,z)$,
then we generally get a smaller modularity level. 
As we are going to show,
the existence of mixed Day terms implies congruence modularity.
This
extends \cite[Corollary 3.5]{DKS}.
\end{remark}   

\begin{proposition} \labbel{propmixday} 
 Every variety $\mathcal {V}$ with 
a sequence of mixed Day terms
is congruence modular. 
\end{proposition}  

\begin{proof} 
If,
in some algebra in $\mathcal {V}$,
 $(a,d) \in  \alpha (\beta \circ \alpha \gamma \circ \beta  ) $
with $a \mathrel { \beta } b \mathrel { \alpha \gamma } c \mathrel { \beta } d  $
and   $(a_1,a_2,a_3,a_4)$ is (the interpretation of)
a modular quadruplet  (under the assignment $x \mapsto a$, 
$y \mapsto b$, $z \mapsto d$),
then it is easy to see that, in any case,
$u_k(a_1,a_2,a_3,a_4) \mathrel { \alpha \beta {+} \alpha \gamma } 
u_k(a,b,c,d)$.
Indeed, 
$u_k(a,a,d,d) \mathrel { \beta } u_k(a,b,c,d) $
and 
$u_k(a,a,d,d) \mathrel { \alpha } u_k(a,a,a,a)= 
a= u_k(a,b,b,a)  \mathrel { \alpha } u_k(a,b,c,d)$,
by \eqref{d0}, so that 
$u_k(a,a,d,d) \mathrel { \alpha \beta } u_k(a,b,c,d) $.
On the other hand, 
 $u_k(a,d,d,d) \mathrel { \beta  } u_k(a,c,c,d)
  \mathrel { \alpha \gamma  } u_k(a,b,c,d)$
and
$u_k(a,d,d,d) \mathrel { \alpha   } u_k(a,d,d,a)
=a= u_k(a,c,c,a)
  \mathrel { \alpha  } u_k(a,c,c,d)$, again by \eqref{d0}, so that
 $u_k(a,d,d,d) \mathrel { \alpha \beta \circ  \alpha \gamma  } u_k(a,b,c,d)$.
The remaining case is similar.

By \eqref{d1},  \eqref{d3}
and the mixed Day equations,
we then get   $a  = u_k(a,b,c,d)
 \mathrel { \alpha \beta {+} \alpha \gamma }
u_m(a,b,c,d) = d$. 
\end{proof}

\begin{problem} \labbel{per}  
 Perhaps it is interesting to study 
other kinds of mixed conditions involving $4$-ary$, 5$-ary
terms, or even  terms of larger arity. 
\end{problem}

\begin{remark} \labbel{defmin}
It is probably interesting to study 
(possibly,
defective)
``minority mixed conditions'',
 namely, conditions obtained 
from Definitions \ref{jondef}, \ref{mix}, etc. 
by replacing
some or all the occurrences of the equations
$  x  =t_h(x,y, x)$ 
by the equation $y  =t_h(x,y, x)$ 
in Condition \eqref{bas}.
Compare Remark \ref{fullsymm}(b). 
In general, we get new conditions.
For example, consider the following set
of equations:
\begin{equation}\labbel{minor}    
\begin{aligned} 
x &= t_{1}(x,z,z),
\\ 
t_{h}(x,x,z) &=
t_{h+1}(x,z,z), \quad \text{ for $ 1 \leq  h < n-1$,} 
\\
t_{n-1} (x, x,z) &=z,
\\
y &= t_{h}(x,y,x) , \quad \quad  \text{ for $ 1 \leq  h \leq n-1$.}
 \end{aligned}   
   \end{equation}
It is easy to see that, for $n \geq 2$,  an abelian group
$G$ has terms $t_1, \dots, t_{n-1}$
satisfying
equations \eqref{minor} 
if and only if  $G$ 
has exponent dividing  $n$. 
\end{remark}

\section{Some more explicit descriptions and summing 
up everything} \labbel{moreex} 

\subsection*{Constructing some algebras and varieties} \labbel{constrav} 
The observations in Remark \ref{specrmk}
and an analysis of the proofs of Theorems \ref{buh}, \ref{dir}(i), \ref{dirthm}
and \ref{thm2hg}(ii)     
can be used in order to provide  a relatively simple description of varieties
furnishing the corresponding  counterexamples.
While the description of these varieties
is quite simple, the proofs that they indeed furnish the desired
counterexamples  rely heavily on the constructions and arguments
from Sections \ref{mainc} and \ref{opti}. 
In the present section we also make some additional
remarks
summing up exactly what our counterexamples show.

Recall that 
lattice operations are denoted by 
juxtaposition and $+$,
 that Boolean complement is denoted by $'$.

\begin{definition} \labbel{simpldef1}
Let $ n \geq 2$ be a natural number
and let $\ell = \frac{n}{2} $ if $n$ is even,
$\ell = \frac{n-1}{2} $ if $n$ is odd.

For every lattice $\mathbf L$ and $0 < i < \frac{n}{2}    $, let 
 $\mathbf L ^{i, n} $ be the algebra with 
base set $L$ and with 
ternary 
operations
$t_1, \dots,t_ {n-1}$ defined as follows
\begin{align*} 
t_h(x,y,z) &=  x,    && \text{ if $0 < h <i$,}
\\ 
t_h(x,y,z) &=  x (y+z),    && \text{ if $h = i$,}
\\ 
t_h(x,y,z) &=  xz,    && \text{ if $i < h < n-i $}
\\ 
t_h(x,y,z) &=  z (y+x),    && \text{ if $h =n-  i$,}
\\ 
t_h(x,y,z) &=  z,    && \text{ if $n-i < h < n$.}
\intertext{Notice that if $n$ is odd and $i=\frac{n-1}{2}$,
then $i$ and  $n-i$ are consecutive integers,
hence 
 the case in the middle does not occur   in the above list of equations.   
Similarly, if $i=1$, then the cases in the first
and last lines do not occur. 
\endgraf 
If 
 $n$ is even, let  $\mathbf L ^{ \ell, n} $ be the algebra with  
the following operations $t_1, \dots,t_ {n-1}$:} 
t_h(x,y,z) &=  x,    && \text{ if $0 < h < \ell$,}
\\ 
t_h(x,y,z) &=  x y+xz+yz,    && \text{ if $h = \ell $.}
\\
t_h(x,y,z) &=  z,    && \text{ if $ \ell < h <  n$.}
\intertext{\indent 
For every Boolean algebra  $\mathbf A$ and $0 < i < \frac{n}{2} $, let 
 $\mathbf A ^{i, n} $ be the algebra with ternary 
operations
$t_1, \dots,t_ {n-1}$ defined as follows.}
t_h(x,y,z) &=  x,    && \text{ if $0 < h <i$,}
\\ 
t_h(x,y,z) &=  x (y'+z),    && \text{ if $h = i$,}
\\ 
t_h(x,y,z) &=  xz,    && \text{ if $i < h < n-i $,}
\\ 
t_h(x,y,z) &=  z (y'+x),    && \text{ if $h =n-  i$,}
\\ 
t_h(x,y,z) &=  z,    && \text{ if $n-i < h <  n$,}
\intertext
{and, for $n$ even, let  $\mathbf A ^{ \ell, n} $ be the algebra with  
operations}
t_h(x,y,z) &=  x,    && \text{ if $0 <  h < \ell$,}
\\ 
t_h(x,y,z) &=  x y'+xz+y' z,    && \text{ if $h = \ell $,}
\\ 
t_h(x,y,z) &=  z,    && \text{ if $ \ell< h < n$.}
 \end{align*}   
 \end{definition}

Notice that all the algebras of the form, say, 
 $\mathbf L ^{i, n} $
as above are term-equivalent (with 
$\mathbf L$ fixed and  $n$, $i$ subject to
the condition   
$0<i < \frac{n}{2} $). A similar remark applies to 
various  classes of the above algebras.
The point is that we shall combine various
algebras of the above form in order to generate appropriate
varieties.
Some particular care is needed, since 
the exact labeling of the operations 
will turn out to be relevant.
Notice also that 
$\mathbf L ^{ \ell, n}$ 
and $\mathbf A ^{ \ell, n}$ 
are defined both in the case $n$  even
and in the case  $n$  odd. 

We now  introduce some families of varieties.

\begin{definition} \labbel{simpldef2}    
As in Definition \ref{simpldef1},  assume  $ n \geq 2$
and set $\ell = \frac{n}{2} $ if $n$ is even and
$\ell = \frac{n-1}{2} $ if $n$ is odd.
Notice that,  for each algebra introduced in 
\ref{simpldef1}, the second 
superscript determines the type of the algebra, hence the following definitions
are well-posed.
Recall that $\mathbf 2$ denotes the two-elements
Boolean algebra and let $\mathbf C = \mathbf C_2$
be the two-elements lattice.  

  \begin{enumerate}[(a)]
    \item 
Let $\mathcal {V}_n^a$ be 
the variety generated by the algebras
\begin{align*}
&  \mathbf C ^{1, n}, \quad 
\mathbf 2 ^{2, n}, \quad
  \mathbf C ^{3, n}, \quad 
 \dots, \quad \mathbf 2 ^{ \ell-1, n},  \quad
 \mathbf C ^{ \ell, n},  
&& \text{ if $\ell$ is odd,}
\\   
&  \mathbf C ^{1, n}, \quad 
\mathbf 2 ^{2, n}, \quad
  \mathbf C ^{3, n}, \quad 
 \dots, \quad \mathbf C ^{\ell-1, n},  \quad
 \mathbf 2 ^{ \ell, n},  
&& \text{ if $\ell$ is even.}
\end{align*}     
The above definition is intended in the sense that if, say,
$\ell=1$, then  $\mathcal {V}_n^a$
is generated by the algebra 
$\mathbf C ^{1, n}$.
Similar conventions apply to the definitions 
 below.

In particular, $\mathcal {V}_2^a$
is generated by $\mathbf C ^{1, 2}$, hence 
$\mathcal {V}_2^a$
is the term-reduct of the variety 
of distributive lattices when only the
majority term is taken into account.

   \item 
Let $\mathcal {V}_n^b$ be 
the variety generated by the algebras
\begin{align*}
&  \mathbf 2 ^{1, n}, \quad 
\mathbf C ^{2, n}, \quad
  \mathbf 2 ^{3, n}, \quad 
 \dots, \quad \mathbf C ^{\ell-1, n},  \quad
 \mathbf 2 ^{\ell, n},  
&& \text{ if $\ell$ is odd,}
\\   
&  \mathbf 2 ^{1, n}, \quad 
\mathbf C ^{2, n}, \quad
  \mathbf 2 ^{3, n}, \quad 
 \dots, \quad \mathbf 2 ^{\ell-1, n},  \quad
 \mathbf C^{ \ell, n},  
&& \text{ if $\ell$ is even.}
\end{align*}     

In particular, $\mathcal {V}_2^b$
is generated by $\mathbf A ^{1, 2}$, hence 
$\mathcal {V}_2^a$
is the term-reduct of the variety 
of Boolean algebras, with the term
$xz+xy'+y'z$.

    \item 
Let $\mathcal V_n^c$ be 
the variety generated by the algebras
\begin{equation*}
  \mathbf C ^{1, n}, \quad 
\mathbf C ^{2, n}, \quad
 \dots, \quad \mathbf C ^{\ell-1, n},  \quad
 \mathbf C ^{\ell, n}.
\end{equation*}     

In particular, $\mathcal {V}_2^c$
and $\mathcal {V}_2^a$ are the same variety.
Moreover, $\mathcal {V}_3^c$
is equal to $\mathcal {V}_3^a$,  being the variety 
generated by $\mathbf C ^{1, 3}$,
hence $\mathcal {V}_3^c$
is term-equivalent to the variety 
of distributive nearlattices, by
a remark  in Definition \ref{bakerdef}.

    \item 
If $n \geq 4$,
let $\mathcal V_n^d$ be 
the variety generated by the algebras
\begin{equation*}
 \mathbf 2 ^{1, n}, \quad 
\mathbf C ^{2, n}, \quad
\mathbf C ^{3, n}, \quad
 \dots, \quad \mathbf C ^{\ell-1, n},  \quad
 \mathbf C ^{\ell, n}.
\end{equation*}
     
\item
Let  $\mathcal V_n^e$
be the \emph{non-indexed product}
\cite{CV,N,Ta} of
$\mathcal V_n^a$ and $\mathcal V_n^b$.

We shall not 
need the exact definition of
the non-indexed product  of  two varieties;
we shall only use the result that
the non-indexed product  of  two varieties
 satisfies exactly
all the   Maltsev conditions satisfied by
both varieties.

\item
If $n \geq 3$, let   $\mathcal V_n^f$
be the non-indexed product
 of
$\mathcal V_n^c$ and $\mathcal V_{n+1}^d$.

\item
Let $\mathcal V_n^g$ be 
the variety generated by the algebras
\begin{equation*}
  \mathbf 2 ^{1, n}, \quad 
\mathbf 2 ^{2, n}, \quad
 \dots, \quad \mathbf 2 ^{\ell-1, n},  \quad
 \mathbf 2 ^{\ell, n}.
\end{equation*}     
\end{enumerate} 
\end{definition}   

\begin{remark} \labbel{simpldefrmk}   
Since both the variety of distributive lattices 
and the variety of Boolean algebras
are generated by their $2$-elements members, we get that if, say,
$\mathbf C ^{i,n} $  belongs to the set of generators of 
some  variety $\mathcal V$ as defined in
\ref{simpldef2} (a) - (d),  then,
for every distributive lattice $\mathbf L$, 
the algebra $\mathbf L ^{i,n} $, 
with the same superscripts, belongs to $\mathcal {V}$.
A similar observation applies to Boolean algebras.
In other words, we could have defined 
$\mathcal V_n^a$ - $\mathcal V_n^g$
by considering a larger set of generators,
namely, all the algebras of the form 
$\mathbf L ^{i,n} $
and $\mathbf A ^{i,n} $,
for the corresponding values of the indices and  letting  
$\mathbf L$ and $\mathbf A$
vary among all distributive lattices 
and 
all Boolean algebras.
The definitions make sense and all the results hold 
even if we let $\mathbf L$ 
vary among all lattices, except that in this  case
the varieties are not necessarily  locally finite.
If we let $\mathbf L$
be any lattice in place of $\mathbf C = \mathbf C_2$ 
in Definition \ref{simpldef2}, we shall call 
the corresponding varieties 
the \emph{extended} $\mathcal V_n^a$ - $\mathcal V_n^f$.
\end{remark} 

\begin{remark} \labbel{+term}
With the above definitions, if $n$ is even, then the operations 
$t_1, \dots, \allowbreak  t_ {n-1} $,
together with the projections
$t_0$ and $t_n$,  
are J{\'o}nsson terms in the cases
of $ \mathbf L ^{i, n}$, 
for $i$ odd,  and  of 
$\mathbf A ^{i, n}$,
for $i$ even, possibly with $i=\ell$.
Hence if $n$ is even, then  $\mathcal {V}_n^a$
is $n$-distributive.   
Similarly, if $n$ is even, the operations 
$t_1, \dots, t_ {n-1} $ 
provide alvin terms in the cases
of $ \mathbf L ^{i, n}$, 
for $i$ even,  and  of 
$\mathbf A ^{i, n}$,
for $i$ odd. 
Hence if $n$ is even, then  $\mathcal {V}_n^b$
is $n$-alvin.   
For every $n$ and  $i$, in the case of 
$ \mathbf L ^{i, n}$, possibly $i=\ell$, the operations 
provide directed J{\'o}nsson terms,  thus, for every $n \geq 2$, 
$\mathcal {V}_n^c$
is $n$-directed-distributive.
Similarly, if $n \geq 4$, then  
$\mathcal {V}_n^d$ is
$n$-directed with alvin heads. 
\end{remark}   

\begin{remark} \labbel{+terma}
Under the conventions introduced in 
Definition \ref{simpldef1}, we have 
$n = 2 \ell $, if $n$ is even and
$n = 2 \ell + 1$, if $n$ is odd.
In each case, the operations
introduced in Definition \ref{simpldef1} 
satisfy
\begin{equation}
\begin{aligned} \labbel{ali}  
t_ {2\ell-i}(x,y,z) &= t_i(z,y,x),  && \text{ $n$ even,  } i =  1, \dots, \ell-1,
\\
t_ {2\ell+1-i}(x,y,z) &= t_i(z,y,x),  && \text{ $n$ odd, \  } i = 1, \dots, \ell.
\end{aligned}      
  \end{equation}    
Hence, adding the two projections as usual, we get a specular sequence
(namely, a sequence satisfying
\eqref{s}  in Definition  \ref{specdef})
of mixed J{\'o}nsson terms
(in the sense of  Definition \ref{mix}). 
Notice that if $n$ even, then 
$t_ {\ell}(x,y,z) = t_\ell (z,y,x) $ in each case.  
 
In view of \eqref{ali}, we could have introduced 
the algebras $ \mathbf L ^{i, n}$, $ \mathbf L ^{ \ell, n}$, 
$\mathbf A ^{i, n}$, $\mathbf A ^{ \ell, n}$
and the varieties 
 $\mathcal V_n^a$ -  $\mathcal V_n^d$
and $\mathcal V_n^g$ by just defining 
the operations $t_ {1},  \dots, t_ {\ell} $ and
then
considering 
 $ t_ {\ell+1},  \dots, t_ {n-1} $
 as defined terms.
For all practical purposes the two possible
approaches are equivalent.
For the sake of uniformity, here it is notationally convenient
to consider $t_ {1},  \dots, t_ {n-1} $
to be  operations.
Everything we shall prove will hold also 
for the term-equivalent algebras and varieties defined by considering 
only the operations $t_ {1},  \dots, t_ {\ell} $. 
\end{remark} 

\begin{definition} \labbel{simpldef3} 
It will be convenient to introduce a special notation 
for algebras and varieties defined as in 
\ref{simpldef1} and \ref{simpldef2}  when also the 
two trivial projections $t_0$ and $t_n$ are considered as operations.
Of course, this is an unessential expansion 
and, moreover, $t_0$ and $t_n$
can be introduced anyway as terms. 
However, as already mentioned, 
it is important for our purposes to keep track 
of the exact number of operations.

If $\mathbf L$ is a  lattice and  $\mathbf A$ 
is a Boolean algebra, let  $\mathbf L ^{i, n,+}$, 
$\mathbf L ^{\ell, n,+}$,
$\mathbf A ^{i, n,+}$ and  $\mathbf A ^{\ell, n,+}$
be constructed as in Definition \ref{simpldef1}, but adding also the two 
ternary operations 
$t_0$ and $t_n$
defined by 
$t_0(x,y,z) =x$ and $t_n(x,y,z)=z$.
We let $\mathcal V_n^{a,+},
\dots, \mathcal V_n^{d,+}$ be the varieties defined
correspondingly, as in Definition \ref{simpldef2}.
\end{definition}

\subsection*{Computing exact levels} \labbel{complev} 
\begin{definition} \labbel{deflev} 
If $\mathcal {V}$ is a congruence distributive variety,
 the \emph{distributive  level} of $\mathcal {V}$ 
is the smallest natural number $n$  such that $\mathcal {V}$ 
is $n$-distributive, namely,
the smallest $n$ such that $\mathcal {V}$ 
has J{\'o}nsson terms $t_0, \dots, t_n$.
The \emph{alvin, modular}, etc., \emph{levels} are defined in a similar way. 

In the case of two-headed terms it is necessary to  explicitly specify the convention.
If some variety $\mathcal {V}$ has two-headed directed Gumm terms
(directed terms with alvin heads)
$p, t_2, \dots, t_{n-2}, q$ in the sense of Definitions \ref{dirgumm}(b)(c),
we say that $\mathcal {V}$ is two-headed $n$-directed Gumm
($n$-directed  with alvin heads).
The counting convention is motivated by 
a remark  in
Definition \ref{dirgumm}(c) and,
more generally, by the definitions and the results from Section \ref{mixed}.
The \emph{two-headed directed Gumm level}
 of a congruence modular variety $\mathcal {V}$ 
is the smallest $n$  such that $\mathcal {V}$ 
is  two-headed $n$-directed Gumm.
The \emph{$n$-directed with alvin heads level}
is defined correspondingly. 

Again, let $\ell = \frac{n}{2} $ if $n$ is even and
$\ell = \frac{n-1}{2} $ if $n$ is odd.
Recall that $R^ \ell$ denotes 
$ R \circ R \circ  {\stackrel{\ell }{\dots}} \circ R    $.  
The \emph{switch level} of some
variety $\mathcal {V}$ is the smallest $n$
(if such an $n$ exists) such that 
either $n$ is even and $\mathcal {V}$ satisfies
 the congruence identity 
$ \alpha (\beta \circ \gamma )  \subseteq ( \alpha ( \gamma \circ \beta )  ) ^{ \ell } $,
or
 $n$ is odd and $\mathcal {V}$ satisfies
 the congruence identity 
$ \alpha (\beta \circ \gamma )  \subseteq ( \alpha ( \gamma \circ \beta )  ) ^{ \ell } 
\circ \alpha \gamma $.
 Notice that every congruence distributive 
variety has a switch level,
since it has an alvin level; 
however, there are 
varieties with a switch level
and which are not congruence distributive.
See Remark \ref{polin}(a). 

The \emph{J-switch level} of some
variety $\mathcal {V}$ is the smallest $n$
(if such an $n$ exists) such that 
either 
 $n \geq 2$ is even and $\mathcal {V}$ satisfies
 the congruence identity 
$ \alpha (\beta \circ \gamma )  \subseteq
\alpha \beta \circ ( \alpha ( \gamma \circ \beta )  ) ^{ \ell -1} 
\circ \alpha \gamma $, or
$n \geq 3$ is odd and $\mathcal {V}$ satisfies
 the congruence identity 
$ \alpha (\beta \circ \gamma )  \subseteq
\alpha \beta \circ  ( \alpha ( \gamma \circ \beta )  ) ^{ \ell } $.
By definiteness, a trivial variety is considered
to have all levels equal to $0$.
By Remark \ref{conidabog},
for every $n$, 
the statement 
that some variety 
has switch level  (J-switch level) 
$\leq n$ is equivalent to the existence of terms
satisfying certain identities, namely, 
it is a strong Maltsev condition. 

In passing,  Proposition  \ref{thmbakbiscor} 
and Theorem \ref{buhbis}
suggest that it is interesting to study
the levels determined by identities like
$\alpha (\beta \circ \gamma )  \subseteq 
\gamma \circ  \alpha \beta \circ \gamma \circ 
\alpha \beta \circ  {\stackrel{n}{\dots}} $
or, say,  
$\alpha (\beta \circ \gamma \circ \beta )
  \subseteq 
\alpha \beta \circ \gamma \circ 
\alpha \beta \circ \gamma \circ {\stackrel{n}{\dots}}$\ 
We shall postpone the study of such levels.
Some related results appear in \cite{nuodist}. 
Notice
that 
the congruence identity 
$\alpha (\beta \circ \gamma )  \subseteq 
\gamma \circ  \alpha \beta \circ \gamma \circ 
\alpha \beta \circ  {\stackrel{n}{\dots}} $
does not
imply congruence distributivity, see 
Hobby and McKenzie \cite[Theorem 9.11]{HMK} and
Kearnes and Kiss \cite[Theorem 8.14]{KK}.
 \end{definition}

The following theorem essentially sums up
all the results of the present paper and adds some more.

\begin{theorem} \labbel{sumup}
Under the above Definitions \ref{simpldef2} and \ref{deflev},
the following table describes the levels of the varieties 
 $\mathcal V_n^a$ -  $\mathcal V_n^f$,
where $n \geq 2$ is always assumed
and in the starred entries 
$n \geq 4$ is assumed.

\vfill
\break

\emph{\begin{equation*}   
\begin{array}{|c|c|c|c|c|c|c|}
\hline&&&&&&
\\[-11pt]
\text{ \qquad \qquad \qquad variety  } & \mathcal V_n^a
 & \mathcal V_n^b 
& \mathcal V_n^c& \mathcal V_n^d & \mathcal V_n^e & \mathcal V_n^f
\\[2pt]
\text{ level } & \text{$n$ even}& \text{$n$ even} & &n \geq 4&
\text{$n$ even} & n \geq 3
\\[2pt]
\hline 
&&&&&&
\\[-8pt]
\text{distributive, J-switch} & n &n{+}1^*& 2n{-}2 & 2n{-}3 &\ n{+}1^* & 2n{-}1
\\[2pt]
\text{alvin, Gumm, switch}
 &n{+}1 & n & 2n{-}1 & 2n{-}4  & n{+}1 & 2n{-}1
\\[2pt] 
\text{modular} & 2n{-}1  & 2n{-}2
 & 2n{-}1 & 2n{-}2 & 2n{-}1 &2n 
\\[2pt]
\text{reversed modular} & 2n & 2n{-}3^* & 2n & 2n{-}3 & 2n &2n 
\\[4pt]
\begin{gathered}  
\text{directed distributive},
\\[-4pt]
\text{mixed J{\'o}nsson} 
\end{gathered}
 &  n &  n & n &   n & n
&n{{+}} 1
\\[7pt]
\begin{gathered}  
   \text{$2${-}headed dir.\  Gumm},
\\[-4pt]
\text{dir.\ with alvin heads}
\end{gathered}
 &  n{{+}}2 &  n ^* & n{+}2 & n &  n{+}2 & n{+}2 
\\[8pt]
\hline
\end{array}   
\end{equation*}
}  
\end{theorem} 

\begin{proof}
\emph{Preliminary observations.} (i) 
 Assume that $n$ is even.  Following the proof of
Theorem \ref{buh}, we see that  $\mathcal V_n^a$
is not $2n{-}1$-reversed-modular
and not $2n{-}2$-modular.
Correspondingly, 
 $\mathcal V_n^b$
is   not $2n{-}3$-modular.
Indeed, the examples constructed in the proof 
of Theorem \ref{buh} can be taken to be members of 
 $\mathcal V_n^a$ and $\mathcal V_n^b$.
This is checked by induction on $n$. 
One base case in the proof of Theorem \ref{buh} 
can be taken to be the variety $\mathcal V_2^a$,
namely, the variety 
generated by  $\mathbf C ^{1,2} $.
In fact, the proof of \ref{buh} uses the observation that lattices
are not $3$-permutable; then $\mathcal V_2^a$
is not $3$-permutable, either, being a term-reduct of the variety
of distributive lattices.
Moreover,
  the existence of a majority term is enough to prove
$2$-distributivity.  
Similarly, the other base case can be taken to be
$\mathcal V_2^b$, the variety generated 
by  $\mathbf 2 ^{1,2} $, noticing
that $\mathcal V_2^b$ is $2$-alvin and nontrivial.

Now suppose that $n \geq 4$, $n$ is even, $\mathbf D$ belongs to 
$\mathcal V_{n-2}^b$ and witnesses  that condition (ii)
in Theorem \ref{buh} fails with $n-2$ in place of $n$,
that is, $\mathbf  D$ is not $2n{-}7$-modular. 
It is no loss of generality to assume that 
$\mathbf  D$ has also the operations 
$t_0$ and  $t_n$, the projections
onto the first and the third coordinate. Namely, we can assume that 
$\mathbf D$ belongs to 
$\mathcal V_{n-2}^{b,+}$, as
introduced in Definition \ref{simpldef3}.
By Remark \ref{+term}, 
$\mathbf D$ has alvin operations
$s_0, \dots,s_{n-2}$.
Recall that  $\mathcal V_{n-2}^{b,+}$ is generated by the algebras
$ \mathbf 2 ^{1, n-2,+^{\phantom{|}}}, 
\mathbf C ^{2, n-2,+}, 
  \mathbf 2 ^{3, n-2,+}, 
 \dots$ from Definitions \ref{simpldef1} and \ref{simpldef3}.
By Birkhoff's Theorem, $\mathbf D$ can be constructed
from  $ \mathbf 2 ^{1, n-2,+^{\phantom{|}}}, 
\mathbf C ^{2, n-2,+}, 
 \mathbf 2 ^{3, n-2,+}, \dots $
 by means of  the usual 
operators of taking products, subalgebras and
homomorphic images.
Now recall the definitions, notation and procedures in 
Construction 
\ref{bak}.
When we pass from $\mathbf D$
to $\mathbf A_4$ there, we shift 
all the indices of the operations by $1$.
Hence
$\mathbf A_4$  can be constructed using the corresponding 
``shifts'' of the generators, namely, 
$\mathbf A_4$  belongs to the variety generated
by 
$ \mathbf 2 ^{2, n}, 
\mathbf C ^{3, n}, 
  \mathbf 2 ^{4, n}, \dots $.
Notice that, say,
$ \mathbf 2 ^{1, n-2,+}$
and $ \mathbf 2 ^{2, n}$
have the same number of operations, since 
$ \mathbf 2 ^{1, n-2,+}$
is obtained from
$ \mathbf 2 ^{1, n-2}$
by adding the two trivial projections, hence the number
of operations is augmented by $2$. 
Moreover, say, the index $i$ at which we take
the equation $t_i(x,y,z) = x (y'+z)$ is shifted by $1$;
this justifies the shift by $1$ of the 
indices in the first upper position.  

On the other hand, each of  the algebras 
$\mathbf A_1$, $\mathbf A_2$ and $\mathbf A_3$
from  Construction 
\ref{bak} has the form 
$\mathbf L ^{1, n}$,
for an appropriate distributive lattice $\mathbf L$
(as we have mentioned at the beginning of Construction 
\ref{c}, the number $n$ is not explicitly indicated in the notations in
Section \ref{mainc}).
Hence $\mathbf A_1$, $\mathbf A_2$ and $\mathbf A_3$
belong to the variety generated by
$\mathbf C ^{1, n}$, 
by Remark \ref{simpldefrmk}. 
In conclusion, 
under the above assumptions, 
the algebra
$\mathbf E = \mathbf A_1 \times \mathbf A_2 \times 
\mathbf A_3 \times \mathbf A_4 $ 
from  Construction 
\ref{bak} belongs to
the variety generated by
$ \mathbf C ^{1, n}, \mathbf 2 ^{2, n}, 
\mathbf C ^{3, n}, 
  \mathbf 2 ^{4, n}, 
 \dots  $,
 namely, to
 $\mathcal V_n^a$. Thus also the substructure  
$\mathbf B=\mathbf B(a,d)$ of 
$ \mathbf A_1 \times \mathbf A_2 \times 
\mathbf A_3 \times \mathbf A_4 $ 
  belongs to  $\mathcal V_n^a$,
for appropriate $a,d \in A_4$.
By Theorem \ref{thmbak}(i) we have that  
$\mathcal V_n^a$ is   not $2n{-}1$-reversed-modular.

The parallel step of the induction is similar,
starting with an algebra $\mathbf D$
in  $\mathcal V_{n-2}^a$
witnessing the failure of 
$2n{-}5$-reversed-modularity.
Then use Construction 
\ref{ba}. In this case, each of
$\mathbf A_1$, $\mathbf A_2$ and $\mathbf A_3$
from  Construction 
\ref{ba} has the form 
$\mathbf A ^{1, n}$,
for some Boolean algebra $\mathbf A$
(notice that $n \geq 4$, hence $1< \ell $),
thus the corresponding 
$\mathbf B$ belongs to  $\mathcal V_n^b$.
 Then, by Theorem \ref{thmba}(ii), we get that  
 $\mathcal V_n^b$
is  not $2n{-}3$-modular.

(ii)
As another observation, we notice that the Gumm level
lies between the alvin and the switch levels. Compare
Remarks \ref{conid} and \ref{gummrmk}. Henceforth,
whenever we show that the  the alvin and the switch levels
are the same, we get the same value for the Gumm level.

(iii)
Still a general observation, holding for
 $\mathcal V_n^a$ -  $\mathcal V_n^d$.
 If we define
$s_h(x,y,z) = t_h(x,t_h(x,y,z),z)$, then
we get a sequence of terms satisfying the
directed J{\'o}nsson condition, hence in each case
the directed distributive level is $\leq n$.  

We now give the specific details of the proof  for 
each variety.

(a) In  order to complete the first column, notice that
$\mathcal V_n^a$
is $n$-distributive (in particular, has
J-switch level $\leq n$), arguing as in (i) above, 
or, more directly, 
 by  Remark \ref{+term}.
Thus   $\mathcal V_n^a$  is 
$2n{-}1$-modular by Day's Theorem \ref{day},
hence $2n$-reversed-modular by Proposition \ref{modeq}. 
Were  $\mathcal V_n^a$ $n{-}1$-distributive, it would be 
 $2n{-}3$-modular by Day's Theorem, contradicting 
what we have proved in (i).
Similarly,
$\mathcal V_n^a$ has mixed J{\'o}nsson level $> n-1$,
by Proposition \ref{propmixmod}(i).
Then, obviously,
$\mathcal V_n^a$ has mixed J{\'o}nsson level
$n$, being $n$-distributive.     
Were  $\mathcal V_n^a$ $n$-alvin, it would be
$2n{-}3$-reversed-modular, by Lemma \ref{dayoptlem}(a),
again contradicting (i). The case $n=2$ is not covered by
the above argument, but if $n=2$, then 
 $2$-alvin  implies congruence permutability;
however, $\mathcal V_2^a$ is not congruence permutable.
In passing, notice that the same arguments show that 
 $\mathcal V_n^a$ is not  defective
$n$-Gumm, by Corollary \ref{propmixmod}(ii)(c).
Another way to see that 
 $\mathcal V_n^a$ is not $n$-alvin is to observe that,
arguing as in the preliminary observation (i), 
 the counterexamples to \eqref{agob}  constructed in Theorem \ref{buhbis}
can be taken to belong to  $\mathcal V_n^a$ (possibly, a subvariety).
This also shows that the switch level of  $\mathcal V_n^a$
is $>n$. It follows that
the J-switch level of  $\mathcal V_n^a$
is $>n-1$,
since trivially
$ \alpha \beta \circ   ( \alpha ( \gamma \circ \beta ))^{\ell-1} 
\subseteq
( \alpha ( \gamma \circ \beta ))^{\ell} $.

Moreover, 
 $\mathcal V_n^a$ is $n{+}1$-alvin 
(in particular, 
has switch level $\leq n+1$) by 
Remark \ref{conidcomm}(b).
The variety  $\mathcal V_n^a$ is not 
$n{-}1$-directed-distributive by Theorem \ref{dir}(ii);
it is
not two-headed $n{+}1$-directed Gumm,
in particular, not $n{+}1$-directed with alvin heads,  since this
would imply $2n{-}1$-reversed-modularity,  
by Theorem \ref{thm2hg}(i). When $n=2$,
just formally notice that the  levels under consideration 
are defined only
for numbers $\geq 4$. 
On the other hand, $\mathcal V_n^a$
is $n$-directed distributive by the preliminary observation (iii);
hence it is $n+2$-directed with alvin heads
(in particular, two-headed $n{+}2$-directed Gumm)
 by adding the trivial projections
at the outer edges.  

We have proved that all the values 
in the first column are correct.
All the other places in the table
are filled using similar arguments,
we now proceed with the details.

(b) 
The relevant result for  $\mathcal V_n^b$ is
Corollary \ref{dayopt}(ii)-(iv). 
In the preliminary observation (i) we have showed that
$\mathcal V_n^b$
is  not $2n{-}3$-modular,
in particular, it is  not $2n{-}4$-reversed-modular.
Arguing in the same way,
we see that the counterexample  in 
Theorem \ref{buhbis}(ii) can be taken to be a member of
 $\mathcal V_n^b$.
Thus, for $n \geq 4$,  $\mathcal V_n^b$
is $n$-alvin,  has 
 J-switch level 
$> n$,
a fortiori,  $\mathcal V_n^b$ has switch level 
$> n-1$.
Moreover, since
$ \alpha \beta \circ \alpha \gamma \circ {\stackrel{n}{\dots}} \circ 
\alpha \gamma  
\subseteq
  \alpha \beta  \circ  ( \alpha ( \gamma \circ \beta ))^{\ell-1} \circ \alpha \gamma$,
we get that, for $n \geq 4$,   $\mathcal V_n^b$ is not 
$n$-distributive,    
otherwise this would contradict \eqref{abog}
in  Theorem \ref{buhbis};
in particular,  $\mathcal V_n^b$
is not $n{-}1$-alvin,
 by
Remark \ref{conidcomm}(b).
For $n=2$,
 $\mathcal V_2^b$ is a non trivial arithmetical 
variety, hence all the levels of $\mathcal V_2^b$
are equal to $2$, except for the 
two-headed directed Gumm and
the directed with alvin heads levels
which are defined only for $n \geq 4$.
By setting $s_1(x,y,z) = t_1(x,y,z)$, $s_{n-1}(x,y,z) = t_{n-1}(x,y,z)$ and
$s_h(x,y,z) = t_h(x,t_h(x,y,z),z)$, for $1<h<n-1$ we get terms witnessing
the directed with alvin heads level. 
All the rest is similar to (a).

(c) Concerning $\mathcal V_n^c$,
notice that, arguing as in the above preliminary observation (i),
 the counterexamples both in Theorem \ref{dir}(i) 
and in Theorem \ref{dirthm}(ii)(iii) 
can be taken to be  members of $\mathcal V_n^c$.
Thus $\mathcal V_n^c$ is $n$-directed-distributive,  
 not $2n{-}1$-reversed-modular,
not $2n{-}2$-alvin, has switch level
$> 2n{-}2$ and
J-switch level $> 2n{-}3$.
Moreover, $\mathcal V_n^c$ is $2n{-}2$-distributive
by \cite[Observation 1.2]{adjt} or
Proposition \ref{propmix1}. 
In particular, 
$\mathcal V_n^c$ has J-switch level
$\leq 2n{-}2$, hence 
switch level $\leq 2n{-}1$.
The variety $\mathcal V_n^c$  is $2n{-}1$-modular
by Corollary \ref{propmixmod}(i). 
Were $\mathcal V_n^c$ 
 two-headed $n{+}1$-directed Gumm, or 
$n{+}1$-directed with alvin heads, it would be
$2n{-}1$-reversed-modular by Corollary \ref{propmixmod}(ii)(c),
 a contradiction.
On the other hand, as we mentioned in (a), 
every $n$-directed-distributive variety
is  trivially $n{+}2$-directed with alvin heads,
in particular, two-headed $n{+}2$-directed Gumm:
just take trivial projections ``at the heads''. 

(d)  
 A large part of the fourth column follows from 
Theorem \ref{thm2hg}. 
Arguing as in the preliminary observation (i),
we see that we can take $\mathcal V_n^d$
to be a counterexample as constructed in the proof of 
Theorem \ref{thm2hg}(ii), in particular, 
$\mathcal V_n^d$ is $n$-directed with alvin heads,
$2n{-}3$-reversed modular and not $2n{-}3$-modular.   
 As far as the first two lines in the table are concerned,
apply  Construction \ref{ba} to some appropriate
algebra $\mathbf D$ in  
   $\mathcal V_{n-2}^{c,+}$. Since  $\mathcal V_{n-2}^{c,+}$
is not $2n{-}7$-distributive
and not  $2n{-}6$-alvin, we get that $\mathcal V_n^d$
is not $2n{-}5$-alvin
and not  $2n{-}4$-distributive, again by 
the  preliminary observation (i) and
Theorem \ref{thmba}(iv).
Hence $\mathcal V_n^d$ is not $n{-}1$-mixed
J{\'o}nsson, since otherwise it would be
$2n{-}4$-distributive, by Proposition \ref{propmix1}.
By (c), 
$\mathcal V_{n-2}^{c,+}$
has 
switch level $2n{-}5$, hence, by
Theorem \ref{thmba}(iv), as above
we get that 
the J-switch level of $\mathcal V_n^d$
is $> 2n{-}4$, hence the 
switch level is $> 2n{-}5$.
Finally, $\mathcal V_n^d$ is 
$2n{-}4$-alvin (hence $2n{-}3$-distributive)
by  Remark \ref{aha}. 

(e)
It is almost obvious 
that, except for the mixed level, each entry in the fifth column
is the maximum of the corresponding entries
in the first two columns.
To prove the result formally, recall that, as we mentioned,
 the non-indexed product of  two varieties $\mathcal W$
and $\mathcal W'$ satisfies exactly
the same Maltsev conditions satisfied both by 
$\mathcal W$
and $\mathcal W'$ \cite{CV,N,Ta}.
Each assertion that some level 
(except for the mixed level)
of a variety 
is $\leq k$
is a strong Maltsev condition, thus 
we get the levels for $\mathcal V_n^e$.

The mixed J{\'o}nsson level is an exception,
since it is a disjunction
of Maltsev conditions:
the equations to be satisfied vary
and are not fixed in advance.
However, by Corollary  \ref{propmixmod}(i),
 the mixed level of
$\mathcal V_n^e$ is $> n-1$,
since $\mathcal V_n^e$ is not 
$2n{-}3$-modular.  
Obviously,  the mixed level of
$\mathcal V_n^e$ is $ \leq  n$,
since $\mathcal V_n^e$ is $n$-directed distributive. 

(f) The levels of $\mathcal V_n^f$
are computed in the same way as (e).
 \end{proof}  

\begin{proposition} \labbel{vg}
If $n \geq 2$, then the variety $\mathcal V_n^g$ 
is $n$-distributive, $n$-alvin, $n$-modular,
$n$-reversed-modular    and $n$-permutable. It is
neither  $n{-}1$-distributive,
nor $n{-}1$-alvin, nor $n{-}1$-modular, 
 nor $n{-}1$-reversed-modular, nor $n{-}1$-permutable. 
 \end{proposition} 

 \begin{proof} 
It is trivial to see that 
$\mathcal V_n^g$ has Pixley terms
$t_0, \dots, t_n$.
The definition of Pixley terms has been 
recalled in Remark \ref{mixrem}.
Thus $\mathcal V_n^g$ is both 
$n$-permutable and  congruence distributive
\cite{adjt,au}. All the other conditions
in the first sentence follow immediately by Remarks \ref{conid}
and \ref{conidmod}.

If $n=2$, then all the conditions in the second sentence fail 
for $\mathcal V_2^g$, since if $n=2$
the conditions are satisfied only in trivial varieties. 
The variety $\mathcal V_3^g$
is a term-reduct of the variety $\mathcal {I}$ of implication algebras.
See, e.~g., \cite[p.\ 17]{ia}. 
Mitschke \cite{Mi1} shows that 
  $\mathcal {I}$  is neither
congruence permutable, nor $2$-distributive.
 All the other conditions in the second sentence,
for $n=3$,  imply 
congruence permutability, hence they fail 
in $\mathcal {I}$, a fortiori, they fail 
in the reduct  $\mathcal V_3^g$.
Then, by induction and arguing as in the proof
of Theorem \ref{sumup},
we get that, for every $n \geq 2$,   $\mathcal V_n^g$
is neither $n{-}1$-modular, 
 nor $n{-}1$-reversed-modular, by using 
Theorem \ref{thmba}(iii).
We get that  $\mathcal V_n^g$
is neither $n{-}1$-distributive, 
 nor $n{-}1$-alvin, by using 
Theorem \ref{thmba}(iv).
 All the above conditions would
be provable from  $n{-}1$-permutability
(together with congruence distributivity),
hence $\mathcal V_n^g$ is not $n{-}1$-permutable, either.
\end{proof}

\begin{problem} \labbel{fill}
Determine 
the directed distributive and the mixed levels
of $\mathcal V_n^g$.
 \end{problem}  

\begin{remark} \labbel{sumupspec}
(a) 
If some variety $\mathcal {V}$ 
has  specular J{\'o}nsson, alvin,
etc.\ terms, 
  we can define the corresponding \emph{specular  levels}.
It seems that there is no  meaningful definition  
of specular levels for the Gumm and  
the directed Gumm conditions. 
 
As we shall explicitly  point out below in some special cases,
the table in the statement of Theorem \ref{sumup}
generally provides  the  specular  levels of the varieties
under consideration.
Essentially, the values in the table
give the specular levels for odd values of the indices in the case of the modular 
and reversed modular levels, for indices of arbitrary
parity in the case of  directed 
levels and
for even values of the indices
in the remaining meaningful cases.  Cf.\ Remark \ref{specnodd}.
Generally, for other values of the indices, the specular level is
given by the next natural number.

(b) The table in the statement of Theorem \ref{sumup}
provides also the levels for the extended varieties 
$\mathcal V_n^a$ - $\mathcal V_n^f$ in the sense of Remark \ref{simpldefrmk}.
Indeed,  we have not used lattice distributivity in order
to prove, say, $n$-distributivity, $n$-directed distributivity etc.  
 On the other hand, the original (not extended) varieties
are subvarieties of the corresponding extended varieties,
hence the levels of the latter are $\geq$ than the levels of the former.
\end{remark}   

\subsection*{Some consequences} \labbel{sc} 
It is well-known that, for every $n \geq 2$,
there is an $n$-distributive ($n$-modular)
 not $n{-}1$-distributive (not $n{-}1$-modular)
variety; see  \cite{D,Fi,FV,JD,CV,Ke,csmc,Le}, among others. 
As we mentioned, 
$n$-distributive and $n$-alvin
are equivalent if $n$ is odd;
moreover, $2$-alvin implies $2$-distributive
and there is a  $2$-distributive not $2$-alvin variety.
 Freese and Valeriote \cite{FV}    
showed, among many other things,  that no more nontrivial relation
holds about the two notions, namely
they
showed that, for every even $n \geq 4$, 
there is an $n$-distributive  ($n$-alvin)
variety which is not $n$-alvin ($n$-distributive).
The present paper provides another proof
 of the above results; we shall then obtain  analogue
results for modularity and reversed modularity.
Our constructions might share some aspects in common
with the above-mentioned works;
we have not fully checked this.
Other papers which might contain constructions
bearing some resemblance with the present ones are
\cite{CCV,Cz}.

\begin{corollary} \labbel{corsumupdist}
\cite{Fi,FV,Ke,Le} 
(i)
For every even $n \geq 2$, there is an $n$-distributive 
not $n$-alvin variety.   

(ii)  
For every even $n \geq 4$, there is an $n$-alvin 
not $n$-distributive variety.   

(iii)
For every  $n \geq 2$, there is a variety 
which is both  $n$-distributive 
and $n$-alvin, but neither $n{-}1$-distributive
nor $n{-}1$-alvin.
  
All the above varieties can be taken to be locally finite.
For $n$  even, all the above varieties
can be taken to satisfy the specular
conditions from Definition \ref{specdef}.
 \end{corollary} 

\begin{proof}
(i) is given by
$\mathcal V_{n}^a$ 
or by $\mathcal V_{\ell}^c$
with $\ell = 1+ \frac{n}{2} $. 
(ii) is given by  $\mathcal V_{n}^b$
or  by $\mathcal V_{\ell}^d$
with $\ell = 2+ \frac{n}{2} $. 
Notice that every $2$-alvin variety is congruence permutable, hence 
$2$-distributive, so that the assumption $n \geq 4$ in (ii) is necessary. 

(iii) The result appears
on \cite[p.\ 71]{FV}.
We can also use Proposition \ref{vg} and  the variety $\mathcal V_n^g$.
It is not clear whether the two counterexamples are 
really distinct.
For $n \geq 5$ and 
$n$ odd the varieties    $\mathcal V_{n-1}^e$
and  $\mathcal V_\ell^f$
for $\ell= \frac{n+1}{2} $ 
furnish other counterexamples.

The varieties $\mathcal V_{n}^a$
and   $\mathcal V_{n}^b$ satisfy the specular conditions by construction.
The variety $\mathcal V_{n}^g$,
too, is constructed in a specular way,
but the terms defining $\mathcal V_{n}^g$
are Pixley, not necessarily J{\'o}nsson or alvin.
When $n$ is even, in order to get a sequence of
 specular J{\'o}nsson (resp.\ alvin) terms for $\mathcal V_{n}^g$, 
let us define,
for $0 < h < n$, 
\begin{align*}
s_h(x,y,z) &= t_h(x,t_h(x,y,z),z), && \text{ for $h$ odd (resp.\ even), and }
\\ 
s_h(x,y,z) &= t_h(x,y,z), && \text{ for $h$ even (resp.\ odd).} 
\qedhere
 \end{align*} 
 \end{proof}  

\begin{corollary} \labbel{corsumupmod}
 \begin{enumerate}[(i)]    
\item  
For every odd $m \geq 3$, there is a locally finite $m$-modular
not $m$-reversed-modular variety.  
\item  
For every odd $m \geq 5$, there is a locally finite  $m$-reversed-modular
not $m$-modular variety.
\item  
 For every $m \geq 2$, there is a locally finite  $m$-modular
$m$-reversed-modular
variety which is neither 
 $m{-}1$-modular, nor $m{-}1$-reversed-modular.   
 \end{enumerate} 
The varieties in (i) and (ii)
can be taken to satisfy the specular
Day conditions introduced in 
Remark \ref{fullsymm}(a).
\end{corollary}

 \begin{proof} 
(i) is witnessed by $\mathcal V_n^c$
with $n= \frac{m+1}{2} $ 
and (ii) is witnessed by $\mathcal V_n^d$
with $n= \frac{m+3}{2} $.
Notice that every 
$3$-reversed-modular variety
is $3$-permutable,  
hence  $3$-modular.
This shows that the assumption $m \geq 5$
is necessary in (ii). 

(iii) follows from Proposition \ref{vg}. 
When $m$ is even and $m \geq 6$,
  $\mathcal V_n^f$ furnishes another counterexample.

The sequences of the operations in the algebras generating 
$\mathcal V_n^c$ and $\mathcal V_n^d$
 are specular. If we apply the proofs 
of Theorems \ref{dir}(ii) and \ref{thm2hg}(i)
we get specular Day (reversed Day) terms.
See Remark \ref{fullsymm}(a).
\end{proof}

\begin{corollary} \labbel{corsumupgumm}
For every  $n \geq 2$, there is an 
$n$-Gumm 
($n$-directed-distributive)
not $n{-}1$-Gumm 
(not $n{-}1$-directed-distributive)
locally finite variety.   

For every  $n \geq 4$, there is a 
two-headed $n$-directed Gumm 
($n$-directed with alvin heads) locally finite variety
which is 
not two-headed $n{-}1$-directed Gumm 
(not $n{-}1$-directed with alvin heads).
\end{corollary}

\begin{proof}
The counterexamples are given by 
$\mathcal V_n^b$ and $\mathcal V_n^a$,
according to the parity of $n$, 
in the Gumm case, by $\mathcal V_n^c$
in the directed distributive case and by  $\mathcal V_n^d$
in the remaining cases.
\end{proof}

Recall that the variety $\mathcal {NL}$ of \emph{nearlattices}
 is  the term-reduct of the  variety 
of lattices obtained by considering the term 
$t_{\mathcal {NL}}(x,y,z)=xy+xz$.
Many authors use the dual term,
but for our purposes the two choices are equivalent.
The variety  $\mathcal B^d$  
of \emph{distributive nearlattices}
 is defined similarly, considering 
only reducts of distributive lattices.
See  Definition \ref{bakerdef} 
and \cite{CHK,Co} for further details.
Since, as we mentioned after its definition, $\mathcal V_3^c$
is term-equivalent to 
$\mathcal B^d$,
then the two varieties share the same levels.
Arguing as in Remark \ref{sumupspec}(b),
it is easy to see that  $\mathcal {NL}$
shares the same levels, too.

\begin{corollary} \labbel{nl} 
The levels of the varieties of nearlattices and of distributive nearlattices
are computed by taking $n=3$ in the column relative to
 $\mathcal V_n^c$ in the table in Theorem \ref{sumup}.
\end{corollary} 

\begin{proof} 
Since $\mathcal V_3^c$
and $\mathcal B^d$
are term-equivalent, they share the same levels;
moreover, the levels of 
$\mathcal {NL}$ are not strictly lower
than the levels of  $\mathcal B^d$,
since $\mathcal B^d$ is a subvariety of $\mathcal {NL}$.
On the positive side, as well-known,
the terms $t_1=t_{\mathcal {NL}}(x,y,z)=xy+xz$,
$t_2=t_{\mathcal {NL}}(x,z,z)=xz$
and $t_3=t_{\mathcal {NL}}(z,y,x)$
witness that $\mathcal {NL}$ is $4$-distributive.
Moreover, $s_1=t_1$ and $s_2=t_3$ witness that
$\mathcal {NL}$ is $3$-directed-distributive. 
All the rest follows from Remark \ref{conidcomm}(b)
and Theorem \ref{dir}(ii).
 
Of course, we do  not need the full power of
Theorem \ref{sumup} to get the present corollary;
essentially, the only main result we need is Proposition \ref{baker+}.   
\end{proof}

\begin{remark} \labbel{mnip}    
Of course, it is interesting to take  non-indexed products
of other pairs, triplets, etc.\ of varieties from Definition \ref{simpldef2},
 possibly, together with other known varieties.
 In any case, each level
(except possibly for the mixed J{\'o}nsson level)
 of such a  product
is the maximum of the levels of the factors, as explained in the
proof of \ref{sumup}. We leave the computations to the interested
reader. 
\end{remark}

\begin{remark} \labbel{spec=} 
(a)   
Definition \ref{simpldef2} 
can be obviously modified in order to 
construct varieties 
satisfying  arbitrary mixed conditions which are also \emph{specular},
in the sense of  Definition \ref{specdef}.
Here we require only specularity, we are not assuming that $l(i)$
from Definition \ref{lr} satisfies any rule prescribed in advance.
Thus we are not assuming that,  say, $l(i)$  has a constant value,
as in the directed or in the Pixley conditions,
or an alternating value, as in the J{\'o}nsson 
and the alvin conditions.

To construct a variety satisfying
some specular mixed condition, just merge the algebras $\mathbf C ^{i, n} $
and $\mathbf 2 ^{i, n} $, for fixed $n$, 
using different and appropriate patterns, in comparison with
the varieties constructed in Definition \ref{simpldef2}. 
Use Remark \ref{lrrmk}(a).

On the other hand, in the case of mixed conditions
which are not specular, 
the algebras from Remark \ref{alternat}
below can be used.
Of course, the simple fact that we can construct
some variety satisfying a desired Maltsev condition $\mathfrak M$ 
does not necessarily  entail that such a variety is 
sufficiently ``generic'' to furnish a counterexample
to other Maltsev conditions which are not implied by
$\mathfrak M$.
In this respect, we have been rather lucky,
as far as the results presented in this paper are concerned,
since we have generally obtained
the best possible values for distributivity and 
modularity levels.
Concerning the general case, we expect
that the case 
of specular conditions is simpler
in comparison with the case of possibly non specular conditions. 

(b) 
 The above comment explains the reason why
presently we are not able to get optimal values for 
levels of varieties with  directed Gumm terms, since
this is not a specular condition. 
Of course, every  $n{-}1$-directed-distributive variety
is $n$-directed Gumm,
hence we get intervals for the best possible values.
For example, by  Corollary
 \ref{propmixmod}(ii)(b), every $n$-directed Gumm
variety is $2n {-}2$-modular. On the other hand,
  $\mathcal V_{n-1}^c$ is
 $n$-directed Gumm
and not $2n {-}4$-modular.

(c) The situation depicted in (b) might appear similar with respect to 
(undirected) Gumm terms, whose defining condition is not specular.
However,  the  Gumm condition is
a defective version of the  alvin condition,
which is indeed specular, for $n$ even.
Together with the study of switch levels,
 the above fact has permitted us to compute
exactly optimal Gumm levels.

With hindsight, this also follows from the main result in
\cite{LGA}, asserting that in a congruence distributive variety
the alvin and the Gumm levels coincide.

(d) Accordingly, in all the examples presented in this paper
we get varieties with the same alvin and Gumm levels,
 as exemplified in the second line in the table
in Theorem \ref{sumup}.
Since the alvin condition is equivalent to congruence distributivity,
while the Gumm condition is equivalent to 
congruence modularity,  there  are varieties 
with a Gumm level and for which the alvin level
is not even defined.
However, as we mentioned, and quite surprisingly,  
\emph{in a congruence distributive variety} 
 the two levels always coincide \cite{LGA}.
In this respect, compare also  Problem \ref{!}
 \end{remark}

\section{Further remarks} \labbel{fur}

\subsection*{Generalizations and problems} 
\begin{remark} \labbel{merge}   
(a) We can merge the methods of the present paper with 
\cite{B}, namely, we can perform 
constructions similar to \ref{bak} and
\ref{bakbis} by considering lattices $\mathbf C_k$
with larger indices, thus getting bounds
(or, better, failure of bounds)
for expressions of the form
$\alpha( \beta \circ \alpha \gamma 
\circ \alpha \beta \circ \dots \circ \alpha \gamma  \circ \beta )$ or
$\alpha( \beta \circ \alpha \gamma 
\circ \alpha \beta \circ \dots \circ \alpha \beta   \circ \gamma  )$.    
Let $A  \circ _{n} B $
be an abbreviation for 
$ A \circ B \circ {\stackrel{n}{\dots}}$

In detail, if $q \geq 2$,
 $ n \geq 2$, $n$ is even and
 $\ell = \frac{n}{2} $,
 then the following congruence identities
fail 
\begin{align} \tag{1} \labbel{1}
\alpha (\beta \ciro ( \alpha \gamma \ciro \alpha \beta 
\ciro {\stackrel{q{-}2}{\dots}} \ciro \alpha \beta    )
\ciro   \gamma ) 
&\subseteq
(\alpha ( \gamma \ciro \beta \ciro {\stackrel{q}{\dots}} 
\ciro  \beta  )) ^{\ell}
&&\text{in $\mathcal V_a^n$, $q$ even,}
\\
\tag{2}\labbel{2}
\alpha (\beta \ciro ( \alpha \gamma \ciro \alpha \beta 
\ciro {\stackrel{q{-}2}{\dots}} \ciro \alpha \beta    )
\ciro   \gamma ) 
&\subseteq
\alpha \beta \ciro (\alpha ( \gamma \ciro \beta \ciro {\stackrel{q}{\dots}} 
\ciro  \beta  )) ^{\ell{-}1} \ciro \alpha \gamma 
&& \text{in $\mathcal V_b^n$, $q$ even,}
\\
\tag{3} \labbel{3}
\alpha (\beta \ciro ( \alpha \gamma \ciro \alpha \beta 
\ciro {\stackrel{q{-}2}{\dots}} \ciro \alpha \gamma     )
\ciro   \beta ) 
&\subseteq
\alpha ( \gamma \ciro \beta \ciro {\stackrel{q}{\dots}} 
\ciro  \gamma ) {\hspace {1pt} \circ  _{n{-}1} \hspace {1pt}}
 \alpha \beta 
&& \text{in $\mathcal V_a^n$, $q $ odd,}
\\
\tag{4}\labbel{4}
\alpha (\beta \ciro ( \alpha \gamma \ciro \alpha \beta 
\ciro {\stackrel{q{-}2}{\dots}} \ciro \alpha \gamma     )
\ciro   \beta ) 
&\subseteq
\alpha \beta {\hspace {1pt} \circ _{n{-}1} \hspace {1pt}}  
\alpha ( \gamma \ciro \beta \ciro {\stackrel{q}{\dots}} 
\ciro  \gamma ) 
&& \text{in $\mathcal V_b^n$, $q $ odd,}
   \end{align}    
where in 
\eqref{4} we further assume
 $n \geq 4$.  
 In particular, the  identities \eqref{1}, \eqref{3}   generally
fail in an $n$-distributive variety and
the  identities \eqref{2}, \eqref{4}   generally
fail in an $n$-alvin variety.

 In order to witness the failure of \eqref{1} - \eqref{4},
first consider the
$\mathbf C_{q+1}$-analogues of 
Theorems \ref{thmbakbis}(ii), \ref{thmbak}(ii)
and Remark \ref{ifin}(b). 
 Under the corresponding
hypotheses,
 we have that if the identity 
$\tilde{\alpha}( \tilde{\beta} \circ \tilde{ \alpha }\tilde{\gamma} \circ 
\tilde{ \alpha }\tilde{ \beta } \circ \dots)
\subseteq 
\chi (\tilde{\alpha}, \tilde{\beta},\tilde{\gamma}) $
 fails in $\mathbf A_4$,
 then  the identity
 $\alpha( \beta \circ  \alpha \gamma \circ \alpha \beta \circ \dots  )
\subseteq 
\alpha (\gamma \circ \beta \circ \gamma \circ \dots )
   \circ \chi ( \alpha, \beta, \gamma ) 
\circ \alpha (\gamma \circ \beta \circ \gamma \circ \dots )$
 fails in $\mathbf B$. 
Then argue as in the proofs of  
Theorems \ref{buh} and \ref{buhbis}, 
using a generalized version of
Theorem 
\ref{thmba}(iv)(iii)
(there is no need to modify Construction \ref{ba}).
The arguments in the proof of Theorem \ref{sumup}
show that the identities actually fail in the
mentioned varieties. 

Notice that the case $q=2$ in \eqref{1} - \eqref{2}
gives the equations in Theorem \ref{buhbis} and that 
the case $q=3$ in \eqref{3} - \eqref{4} implies Theorem \ref{buh}. 

Moreover, if  $n, q \geq 2$, then
the  congruence identity
\begin{equation}
\tag{5}\labbel{5}
\alpha (\beta \circ ( \alpha \gamma \circ \alpha \beta 
\circ {\stackrel{q{-}2}{\dots}} \circ \alpha \beta ^\bullet   )
\circ   \gamma ^ \bullet) 
\subseteq
(\alpha ( \gamma \circ \beta \circ {\stackrel{q}{\dots}} 
\circ  \beta ^ \bullet  )) ^{n-1}   
  \end{equation}    
fails in $\mathcal V_c^n$,
where 
$n$ is possibly odd,
$ \beta ^ \bullet = \beta $,
$ \gamma ^ \bullet = \gamma  $
if $q$ is even and 
$\beta^ \bullet = \gamma $,
$ \gamma ^ \bullet = \beta   $
if $q$ is odd.
Thus  identity \eqref{5} generally fails in an  
 $n$-directed-distributive variety. 
The failure of \eqref{5}
is obtained as in the proofs of Theorems \ref{dir}(i) 
and \ref{dirthm}(iii).  
Compare \cite{misharp}
for similar arguments and for a parallel result. 
The 
special case $q=2$ in \eqref{5}
is Theorem \ref{dirthm}(iii) and the case   
$q=3$ in \eqref{5} implies Theorem \ref{dir}(i)   

(b) The  remarks in (a) will be probably  useful
in order to check whether some results
from \cite{jds} are optimal.

In this respect, notice that
it follows from \cite{D} 
  that every
$n$-distributive variety 
satisfies the congruence identity
$ \alpha ( \beta \circ \gamma \circ \beta )
\subseteq \alpha \beta \circ \alpha \gamma 
\circ {\stackrel{2n-1}{\dots}}    $.
See
Remark \ref{daystrong}. The  assertion follows also from 
Theorem \ref{propmix2}(ii), by
taking $S= \beta $ and $T = \gamma $. 
See \cite{jds} for still another proof and for more general  results.  
Theorem \ref{buh} and, in particular, the variety $\mathcal V_n^a$
show that, for $n$ even, the result is the best possible;
actually, it cannot be improved
under the additional assumption  
that $\gamma \subseteq \alpha $.  
We do not know exactly what happens for $n$ odd;
see Problem \ref{prob} below. 
We do not know what happens if we consider longer 
chains of compositions on the left-hand side, namely, which 
are the best bounds for
$ \alpha (  \beta \circ \gamma  \circ {\stackrel{k} {\dots}})$?
Notice that in (a) we have considered 
 expressions of the form
$\alpha( \beta \circ \alpha \gamma 
\circ \alpha \beta \circ \dots \circ \alpha \gamma  \circ \beta )$ or
$\alpha( \beta \circ \alpha \gamma 
\circ \alpha \beta \circ \dots \circ \alpha \beta   \circ \gamma  )$,
instead.   See again Remark \ref{daystrong}
for possible differences between the two cases.
\end{remark}

\begin{remark} \labbel{bestnu}
 A.\ Mitschke \cite{Mi}
proved that every variety $\mathcal {V}$ with a near-unanimity term 
is congruence distributive, actually, that a variety with an
$m{+}2$-ary near-unanimity term is $2m$-distributive.
In particular, any such variety 
is congruence modular and, by Day's Theorem,
$4m{-}1$-modular.
Actually, L.\ Sequeira \cite[Theorem 3.19]{S}
 showed that a variety with an
$m{+}2$-ary near-unanimity term is $2m{+}1$-modular.
See also \cite{B} for related results. 

It can be shown that Mitschke's and  
Sequeira's results
are optimal 
by combining various reducts of lattices using
 near-unanimity terms
of the form, say,
$\prod _{ |\{ i, j, k \} | =3}(x_i + x_j + x_k) $,
 more generally,
$\prod _{|I|=k} \sum _{i \in I} x_i$.
We have presented details in \cite{misharp}.
 \end{remark}

\begin{remark} \labbel{odd}
In the case  $n$   odd
we noticed in \cite[Proposition 6.1]{ntcm} that Day's Theorem can be improved
(at least) by $1$,
namely that if $n >1 $ and $n$ is odd, then 
every $n$-distributive variety
is $2n{-}2$-modular. 
As we mentioned, this fact
is implicit in \cite{LTT} and   can be also obtained 
as a consequence of 
 Corollary \ref{propmixmod}(ii)(b).

While we do not know what is 
the best possible result, 
Theorem \ref{daysh} implies 
that in the case $n$ odd Day's
Theorem can be improved at most by $2$. 
Indeed, it is trivial that every $n{-}1$-distributive variety  
is $n$-distributive.
Hence, if $n$ is odd, thus $n-1$ is even, 
then Theorem \ref{daysh} provides
an $n{-}1$-distributive variety  
(thus 
also $n$-distributive)
which is not 
$2n{-}4$-modular.
Alternatively, for odd $n \geq 5$,
$\mathcal {V}_{n-1}^e$
is $n$-distributive and not  $2n{-}4$-modular, by
Theorem \ref{sumup}.
 \end{remark}  

We are not claiming that the problems below are difficult;
apparently, they are not solved by  the present work.
In connection with problems
(a) and (b) below,
observe that 
many congruence and relation identities
valid in $3$-distributive varieties
have been described in \cite{jds,ia}. 
In connection with
(c), see Remark \ref{odd}.
In connection with
(d), notice that, by Lemma \ref{dayoptlem}(a)
or Corollary \ref{propmixmod}(ii)(c),
 $4$-alvin varieties
 satisfy
$\alpha ( \beta \circ \alpha \gamma \circ \beta )
\subseteq 
\alpha \gamma \circ \alpha \beta \circ 
\alpha \gamma \circ \alpha \beta \circ   
 \alpha \gamma $. 

\begin{problem} \labbel{prob}
(a) Do $3$-distributive varieties 
satisfy 
$\alpha ( \beta \circ \alpha \gamma \circ \beta )
\subseteq 
\alpha \beta \circ \alpha \gamma \circ \alpha \beta $?  

(b) Do $3$-distributive varieties 
satisfy 
$\alpha ( \beta \circ \gamma \circ \beta )
\subseteq 
\alpha \beta \circ \alpha \gamma \circ \alpha \beta $?  

(c) More generally, for $n$ odd, is   every 
$n$-distributive variety  $2n{-}3$-modular? 
For $n$ odd, does every 
$n$-distributive variety satisfy
$\alpha( \beta \circ  \gamma \circ \beta )\subseteq 
\alpha \beta \circ \alpha \gamma \circ {\stackrel{2n-3}{\dots}} 
\circ  \alpha \beta $?  
 
(d) Does every  $4$-alvin variety
satisfy
$\alpha ( \beta \circ  \gamma \circ \beta )
\subseteq 
\alpha \gamma \circ \alpha \beta 
\circ \alpha \gamma \circ \alpha \beta \circ \alpha \gamma  $?
More generally, if $n$ is even,  
 does every  $n$-alvin variety
satisfy
$\alpha( \beta \circ  \gamma \circ \beta )\subseteq 
\alpha \gamma  \circ \beta  \gamma \circ {\stackrel{2n-3}{\dots}} 
\circ  \alpha \gamma  $?
Compare
Remark \ref{daystrong}.

(e) Study the distributivity spectra, in the sense of \cite{jds},
of the varieties $\mathcal V_n^a$ - $\mathcal V_n^f$. 
In this connection, see Remark \ref{merge}.
\end{problem}

\begin{remark} \labbel{alternat} 
When dealing with J{\'o}nsson and alvin terms when $n$ is odd
(or, more generally, when dealing with mixed J{\'o}nsson terms, for 
arbitrary $n$)
it is probably useful to perform a construction similar 
to \ref{ba} but modifying the definition of $t_{n-1}$ as follows:
\begin{multline*} 
t_1(x,y,z) = x(y'+z), \qquad
t_2(x,y,z) = xz, \qquad
t_3(x,y,z) = xz, \qquad  \dots, \qquad 
\\
\dots, \qquad t_{n-2}(x,y,z) = xz,
\qquad t_{n-1}(x,y,z) = z(y+x).
\end{multline*}   
We do not know whether such  constructions
are sufficient in order to get optimal bounds
in the  cases of arbitrary (not necessarily specular)
mixed J{\'o}nsson conditions, too. In particular,
we do not know 
if these constructions are sufficient to solve 
some of the above problems. 
Compare Remark \ref{spec=}(a).
We expect that our constructions
should be somewhat further modified in order to get 
the best possible results.

In any case, for $n \geq 2$,
 $0 < i < j < n$
and every Boolean algebra $\mathbf A$,
 it is probably useful to consider  its term-reduct
$\mathbf A ^{i,j; n} $, the algebra with ternary 
operations
$t_1, \dots,t_ {n-1}$ defined as follows.
\begin{align*} 
t_h(x,y,z) &=  x,    && \text{ if $0 < h <  i$,}
\\ 
t_h(x,y,z) &=  x (y'+z),    && \text{ if $h = i$,}
\\ 
t_h(x,y,z) &=  xz,    && \text{ if $  i  < h < j $}
\\ 
t_h(x,y,z) &=  z (y+x),    && \text{ if $h =j$,}
\\ 
t_h(x,y,z) &=  z,    && \text{ if $j  < h <  n$.}
 \end{align*}   
 
As we hinted in 
Remark \ref{spec=}(a), using the algebras
$\mathbf A ^{i,j; n} $
we can construct more varieties 
in the same fashion as of Definition \ref{simpldef2}. 
It is probably interesting to study
varieties constructed in this way, too.
\end{remark}

\begin{remark} \labbel{2gen}
As it follows from Definition \ref{simpldef2}
and from Theorem \ref{sumup},
all the counterexamples 
in this paper can be taken to be varieties 
 generated by a finite set of two-elements algebras.

Is there some hidden more general fact behind this observation?
Are there implications
similar to the ones considered here and such that all the possible
counterexamples necessarily involve varieties which cannot 
be generated by 
two-elements algebras?
In this connection, notice that also the main counterexamples
from \cite[Section 3]{csmc} 
are varieties generated by 
two-elements algebras. 
 \end{remark}

\subsection*{Weakening some assumptions}
\begin{remark} \labbel{rmkka}
As we mentioned, the assumptions
in Construction \ref{c} are rather weak, hence it is possible
that further applications can be found.
It is also probably possible to modify the construction using
similar ideas. A promising (possibly difficult)  approach is trying to deal
with $4$-ary terms.  Moreover, as we are going to explain soon,
our results can be stated in a slightly more general form.
 \end{remark}

\begin{remark} \labbel{rmkkb}   
 We have made no essential use 
of the assumption that $\tilde \alpha$, $\tilde \beta $
and $\tilde \gamma $ 
are congruences in the proofs of Theorems \ref{thmbak},
\ref{thmbakbis} and \ref{thmba}.
Of course, we need the assumption that, say, 
$\tilde \alpha$ is a congruence, in order to get that $\alpha$
is a congruence.
Apart from this, the assumption that the relations at hand
are congruences is not used in the proofs.
Hence Theorems \ref{thmbak}, \ref{thmbakbis} and \ref{thmba}
hold even in the case when $\tilde \alpha$, $\tilde \beta $
and $\tilde \gamma $, or just some of them, are assumed to be, say,
tolerances or reflexive and admissible relations,  
provided the corresponding assumptions are made relative to 
$\alpha$, $\beta$  or $\gamma$. 

The above observation might be of interest, since there are 
deep problems involving relation identities in 
congruence distributive varieties. See, e.~g., \cite{ia,B}.
See also Remark \ref{equiv} below. 
No matter how interesting the subject, in 
the present work we have  limited ourselves
 to  congruences.
 \end{remark}

\begin{remark} \labbel{expres}
For practical purposes, we have defined an expression
$\chi$ 
to be a term in the language
$ \{ \circ, \cap \} $. See the convention introduced
right before Theorem \ref{thmbak}. 
Anyway, we have used only a very weak property of such terms.
The statements of Theorems \ref{thmbak}, \ref{thmbakbis}
and \ref{thmba}  
 apply to any expression in the following more 
extensive
sense. 

 An ($m$-ary)\emph{expression} $\chi( x , y , \dots )$ is a way of associating, 
for every algebra $\mathbf A$,   
a congruence $\chi( \alpha , \beta , \dots )$ of $\mathbf A$ 
to each $m$-uple $\alpha, \beta, \dots $ of congruences of $\mathbf A$.
We require that expressions satisfy the following property.

  \begin{enumerate}  
 \item [(E)] 
 Whenever $\mathbf A_1$, $\mathbf A_2$ are algebras of the same type,
$\mathbf B \subseteq \mathbf A_1 \times \mathbf A_2$, 
$\alpha_1, \beta_1, \dots $ are congruences of $\mathbf A_1$,
$\alpha_2, \beta_2, \dots $ are congruences of $\mathbf A_2$,
$\alpha, \beta, \dots $ are the congruences 
induced by $\alpha_1 \times \alpha _2, \beta_1 \times \beta _2, \dots $
on $\mathbf B$, $a=(a_1,a_2)$, $b=(b_1,b_2)$ 
are elements of $\mathbf B$ 
 and
$(a,b) \in \chi( \alpha , \beta , \dots )$,
then 
$(a_2,b_2) \in \chi( \alpha_2 , \beta_2 , \dots )$.
 \end{enumerate}  

Notice that Condition (E)  holds 
if the \emph{homomorphism} condition 
\begin{equation}\labbel{h} \tag{H}      
\varphi (\chi ( \alpha , \beta , \dots ) ) \subseteq 
\chi(\varphi( \alpha ), \varphi( \beta  ),   \dots )
 \end{equation}
 holds, 
for every algebra $\mathbf A$,
for every
$m$-uple  of congruences on  $\mathbf A$ and
for every morphism $\varphi$ with domain $\mathbf A$.
However, at first sight, the condition  \eqref{h} seems
slightly stronger than (E). 
Let us mention that 
condition \eqref{h} is useful also in different contexts.
See \cite[Section 3]{cm2}
and further references there.

As in Remark \ref{rmkkb},
we can replace the word ``congruence''
in the above definitions with ``reflexive and admissible relation''.
\end{remark}

\subsection*{From a wider perspective} 
\begin{problem} \labbel{!}
In this paper we have considered 
only one side of
the problem of the relationships
among the modularity and the distributivity levels
of congruence distributive varieties.
As mentioned at the beginning of the introduction,
it follows abstractly just from the theory of Maltsev 
conditions that, for every $n$, there is some
$m(n)$ such that every 
$n$-distributive variety is  $m(n)$-modular.   
We have evaluated the best possible value, namely, 
$m(n) = 2n-1$, in the case $n$ even, and showed that
if $n$ is odd, the only possibilities for the best value
are $m(n) = 2n-2$ or $m(n) = 2n-3$. See Remark \ref{odd}.

At first sight, the other direction looks completely
and vacuously  trivial,
since there are congruence modular varieties
(actually, congruence permutable, namely,
$2$-modular varieties)  which are not 
congruence distributive. 
However, the problem is completely 
nontrivial if we 
study the relationship among the modularity
and distributivity levels of some variety 
$\mathcal {V}$,
\emph{assuming that $\mathcal {V}$ is congruence distributive}.

As we mentioned, we have showed in \cite{LGA}  that 
in a congruence distributive variety 
the alvin level and the Gumm levels coincide.
In particular, we get from \cite{LTT}
that an $r$-modular congruence distributive  variety is 
 $r^2{-}r{+}2$-distributive. Let
DG($r$) denote the smallest $n$
such that every $r$-modular variety 
is $n$-Gumm, and let  
 DJ($r$) denote the smallest $n$
such that every $r$-modular
\emph{congruence distributive}  variety  
is $n$-distributive.
Notice that DG($r$) is defined for every $r$,
by  \cite{G1,G}, and  DJ($r$) is defined, as well, by the above comment.

It seems that
the exact evaluation
of DG($r$) 
and of 
 DJ($r$)
is one of the most important problems left open
by \cite{LGA} and the present work.
 \end{problem}

\begin{remark} \labbel{polin} 
(a)
Our main focus in the present  paper 
are congruence distributive and congruence modular 
varieties. However,  as we mentioned, some identities we have considered 
are strictly weaker  than congruence modularity.
Using Polin's
 variety \cite{Po,DF},
we shall check that the identities
\eqref{agob} and \eqref{abog} from Theorem \ref{buhbis}  
do not generally imply congruence modularity.
In particular, having a switch or a J-switch level 
(Definition \ref{deflev}) does not imply congruence modularity.

Consider the following
algebras of type 
$ (2,1,1) $,
the ``external algebra''  
  $\mathbf A_e = \langle \{ 0,1 \}, \cdot, \sigma , Id \rangle $
and the ``internal algebra'' 
 $\mathbf A_i = \langle \{ 0,1 \}, \cdot, 1 , \sigma  \rangle $,
where $\cdot$ is meet,  $\sigma$ is the only nontrivial permutation of 
$\{ 0,1 \}$ and $1$ is the function with constant value $1$. 
Both algebras, considered alone,  are
term equivalent with the
$2$-elements Boolean algebra,
but the variety $\mathcal P$ they generate together, 
though $4$-permutable, 
is not
even congruence modular \cite{chic,DF,Po}.
Let the unary operation symbols of 
  $\mathbf A_e $  and of  $\mathbf A_i$
be denoted by $^+$ and  $'$, in that order.
The ``external join'' $x+_e y = (x^+y^+)^+$
and 
the ``internal join'' $x+_i y = (x'y')'$    
provide $\mathbf A_e $,
respectively, $\mathbf A_i$
with the Boolean structure.

The terms 
\begin{equation}\labbel{poll}     
\begin{aligned} 
t_1(x,y,z) &= x(y+_e z),
\\
t_2(x,y,z) &= xz+_i xy' +_i zy'  \quad \text{ and } 
\\
t_3(x,y,z) &= z(y+_e x)
 \end{aligned}    
 \end{equation}
witness that
$\mathcal P$
has J-switch level $4$. Thus 
$\mathcal P$ satisfies 
identity \eqref{abog} for $\ell=2$.
This fact had been stated without  proof on
 \cite[p.\ 167]{au}. 
Notice that the sequence of terms given in \eqref{poll}
is specular in the sense of Definition \ref{specdef}. 

In particular, having 
J-switch level $\geq 4$ does not imply 
congruence modularity.
Moreover 
having 
switch level $\geq 5$ does not imply 
congruence modularity, either. Indeed, 
$\alpha \beta  \circ  ( \alpha ( \gamma \circ \beta ))  \circ \alpha \gamma
\subseteq 
( \alpha ( \gamma \circ \beta ))^2  \circ \alpha \gamma$. 
 
(b) On the other hand, having J-switch level $\leq 3$
is the same as being $3$-Gumm,
and having switch level $\leq 4$
is the same as being  defective 
$4$-Gumm. Compare Remark \ref{gummrmk},
taking converses in the case $3$-Gumm.  
 In particular, these conditions do imply 
congruence modularity, by Theorem \ref{Gth}. 

(c) In passing, we notice that, as well-known, the specular terms
$s_1(x,y,z) = x(y^+ +_e z)$,
$t_2(x,y,z)$ from \eqref{poll}
and 
$s_3(x,y,z) = z(y^+ +_e x)$ 
witness that $\mathcal P$
is $4$-permutable, by \cite{HM}.
However, we also have
$s_1(x,y,x)= s_3(x,y,x) = x$
and this implies that $\mathcal P$
satisfies also
\begin{equation} \labbel{perm}     
\alpha ( \beta \circ \gamma ) \circ \alpha ( \beta \circ \gamma ) 
\subseteq 
\alpha \gamma \circ \alpha ( \beta \circ \gamma )  \circ \alpha \beta .
  \end{equation}        
This identity is stronger than $4$-permutability: 
 just take $\alpha= 1$ to get back
$4$-permutability.

To check \eqref{perm}, assume that 
$ a \mathrel { \beta  } b \mathrel { \gamma  } c
 \mathrel { \beta  } d \mathrel { \gamma  } e $
and $a \mathrel { \alpha  } c \mathrel { \alpha  } e $.
Then
\begin{align*} 
a&= s_1(a,b,b) \mathrel { \alpha \gamma  } s_1(a,b,c)
\mathrel { \beta  } s_1(b,b,c)=t_2(b,c,c) \mathrel { \beta }
t_2(b,c,d)
\\
 & \mathrel { \gamma } t_2(c,c,d) =s_3(c,d,d) \mathrel { \gamma  } s_3(c,d,e)
\mathrel { \alpha  \beta } s_3(d,d,e)= e.   
\end{align*}  
\end{remark}

\begin{remark} \labbel{equiv}
Define an equivalence relation $\sim_{CL}$ between varieties as follows:
$\mathcal {V} \sim_{CL} \mathcal {W}$ 
if congruence lattices of algebras in 
$\mathcal {V}$ and $\mathcal {W}$ 
satisfy exactly the same identities.

In passing, notice that
the relation $\sim_{CL}$
arises by means of a Galois 
connection from the more frequently considered
relation $\models_{Con}$ between lattice identities.
The relation  $\varepsilon \models _{Con} \varepsilon' $ 
means that, for every variety $\mathcal {V}$,
if $\mathcal {V} \models _{Con} \varepsilon$,
then $\mathcal {V} \models _{Con} \varepsilon'$.
As usual, $\mathcal {V} \models _{Con} \varepsilon$
means that every congruence lattice of algebras in $\mathcal {V}$
satisfies $\varepsilon$.
See \cite{CV} for a survey about the notion. 

The relation $\sim_{CL}$ is rather rough,
for example, all nontrivial congruence distributive varieties 
form a single equivalence class.
As already implicit in
\cite{JD}, it is interesting to study a finer relation
$\sim_{Co}$ defined by
$\mathcal {V} \sim_{Co} \mathcal {W}$
if the set of congruence relations of algebras in 
$\mathcal {V}$ and $\mathcal W$ satisfy the same identities
expressed in the language with $+$, $\cdot$
and $\circ$ (for two relations
$R$ and $S$ it is convenient to interpret
$R+S$ as the transitive closure of $R \circ S$).  

Theorem \ref{sumup} and Proposition \ref{vg}
show that the varieties $\mathcal V_n^a $  -
$\mathcal V_n^g $ are pairwise not
 $\sim_{Co}$-equivalent,
say, for all even $n > 4$.
 Notice that some assumption on $n$ is necessary,
since there are some trivial initial cases, 
 e.~g., $\mathcal V_2^a $ and $\mathcal V_2^c $
are the same variety.

A finer relation can be considered: write
$\mathcal {V} \sim_{Rel} \mathcal {W}$
to mean that the set of reflexive and admissible  relations of algebras in 
$\mathcal {V}$ and $\mathcal W$ satisfy the same identities
expressed in the language with $+$, $\cdot$
and $\circ$. Results from
\cite{ia,B} show that 
 $ \sim_{Rel}$ is strictly finer than
$ \sim_{Co}$.  

In conclusion, there are viable intermediate alternatives strictly between
congruence identities and the lattice of interpretability types
\cite{GT,N}.
 \end{remark}

In order to keep the following list within a reasonable length,
we have sometimes quoted survey works rather than  original sources.
The reader is advised to consult the quoted works for credits
to the original studies.

\end{document}